\documentclass[12pt,reqno,final]{amsart}
\usepackage[utf8]{inputenc} 
\usepackage{amsmath,amsfonts,amssymb,amscd,amsthm,amsbsy}
\usepackage{graphicx}
\graphicspath{{images/}}
\usepackage{xcolor}

\usepackage{caption}
\usepackage{dsfont}
\usepackage{comment}
\usepackage{marginnote}
\usepackage{hyperref}
\usepackage{cleveref}
\usepackage{multirow}
\crefformat{equation}{(#2#1#3)}

\usepackage{makecell}

\DeclareMathOperator{\Tr}{Tr}
\newcommand\equ[1]{{\rm (\ref{#1})}}
\newcommand{\R}{\mathbb{R}}
\newcommand{\Z}{\mathbb{Z}}
\newcommand{\N}{\mathbb{N}}
\newcommand{\C}{\mathbb{C}}
\newcommand{\eps}{\varepsilon}
\newcommand{\ecc}{e}

\newtheorem{theorem}{Theorem}
\newtheorem{meta-thm}[theorem]{Meta-Theorem}
\newtheorem{lemma}[theorem]{Lemma}

\newtheorem{proposition}[theorem]{Proposition}

\theoremstyle{remark}
\newtheorem{remark}[theorem]{Remark}
\theoremstyle{definition}
\newtheorem{definition}[theorem]{Definition}

\newcommand{\parentesis}[1]{\left (  #1 \right ) }

\newcommand{\matdos}[4]{
    \begin{pmatrix}
        #1 & #2 \\
        #3 & #4
    \end{pmatrix}
}

\newtheorem{assumption}{Assumption}
\crefname{assumption}{Assumption}{Assumptions}
\renewcommand{\vec}[1]{\mathbf{#1}}
\newcommand{\comentario}[1]{\color[rgb]{1,0,0} #1 \color[rgb]{0,0,0}}
\renewcommand{\d}{\operatorname{d}\!}

\begin{document}

\title[The spin-spin problem in Celestial Mechanics]
{The spin-spin problem in Celestial Mechanics}

\author[A. Celletti]{Alessandra Celletti}

\address{Department of Mathematics, University of Rome Tor Vergata,
  Via della Ricerca Scientifica 1, 00133 Rome (Italy)}

\email{celletti@mat.uniroma2.it}

\author[J. Gimeno]{Joan Gimeno}

\address{Department of Mathematics, University of Rome Tor Vergata,
  Via della Ricerca Scientifica 1, 00133 Rome (Italy)}

\email{gimeno@mat.uniroma2.it}

\author[M. Misquero]{Mauricio Misquero}

\address{Department of Mathematics, University of Rome Tor Vergata,
  Via della Ricerca Scientifica 1, 00133 Rome (Italy)}

\email{misquero@mat.uniroma2.it}

\thanks{%
  A.C. and M.M. thank MSCA-ITN-ETN Stardust-R, Grant Agreement 813644
  and J.G. thanks MIUR-PRIN 20178CJA2B ``New Frontiers of Celestial
  Mechanics: theory and Applications''. J.G. has also been supported
  by the Spanish grants PGC2018-100699-B-I00 (MCIU/AEI/FEDER, UE) and
  the Catalan grant 2017 SGR 1374.}

\begin{abstract}
We study the dynamics of two homogeneous rigid ellipsoids subject to
their mutual gravitational influence. We assume that the spin axis of
each ellipsoid coincides with its shortest physical axis and is
perpendicular to the orbital plane. Due to such assumptions, the
problem is planar and depends on particular parameters of the
ellipsoids, most notably, the equatorial oblateness and the flattening
with respect to the shortest physical axes. We consider two models for
such configuration: while in the full model, there is a coupling
between the orbital and rotational motions, in the Keplerian model,
the centers of mass of the bodies are constrained to move on coplanar
Keplerian ellipses. The Keplerian case, in the approximation that
includes the coupling between the spins of the two ellipsoids, is what
we call spin-spin problem, that is a generalization of the classical
spin-orbit problem. In this paper we continue the investigations of
\cite{mis2021} on the spin-spin problem by comparing it with the
spin-orbit problem and also with the full model.

Beside detailing the models associated to the spin-orbit and spin-spin
problems, we introduce the notions of standard and balanced
resonances, which lead us to investigate the existence of periodic and
quasi-periodic solutions.  We also give a qualitative description of
the phase space and provide results on the linear stability of
solutions for the spin-orbit and spin-spin problems.  We conclude by
providing a comparison between the full and the Keplerian models with
particular reference to the interaction between the rotational and
orbital motions.
\end{abstract}


\maketitle


\section{Introduction}\label{sec:introduction}

The dynamics of two rigid bodies orbiting under their mutual
gravitational attraction is a classical problem of Celestial Mechanics
known as the {\sl Full Two-Body problem}. In this context, Kinoshita
investigated the problem by using Hori-Deprit perturbation theory
\cite{Kinoshita}, assuming that one of the bodies is spherical and the
other body is triaxial. Later, the problem of two extended rigid
bodies was studied in \cite{mac1995} as a Hamiltonian system with
respect to a non-canonical structure, which is used to characterize
the relative equilibria. A seminal work was performed in
\cite{bou2017} to which we refer for an alternative description of the
model of the full two rigid body problem using spherical harmonics and
Wigner D-matrices.  In \cite{sch2009}, the problem is restricted to a
planar configuration with the potential expanded to order $1/r^3$,
where $r$ is the relative distance between the two rigid bodies; under
this condition, \cite{sch2009} describes the relative equilibria and
their stability properties.

In this paper, we investigate different simplified models of
rotational dynamics of celestial bodies, subject to the mutual
gravitational attraction.  The spin-spin problem was introduced in
\cite{mis2021} as a planar version of the Full Two-Body problem for
ellipsoids (compare with \cite{batmor2015,boulas2009}), by using the
expansion of the potential up to order $1/r^5$, which results in the
coupling of the spins of both bodies.  An equivalent model was studied
in \cite{jaf2016} (see also \cite{hou2017}).

Indeed, we consider a hierarchy of models with different
complexity. In particular, we start by considering two homogeneous
rigid ellipsoidal bodies subject to the following assumptions:

\begin{assumption}\label{a:planar_motion}
 The spin axis of each ellipsoid is perpendicular to the orbital
 plane.
\end{assumption}

\begin{assumption}\label{a:planar_symmetry}
 The spin axis of each ellipsoid is aligned with the shortest physical
 axis of the satellite.
\end{assumption}

\cref{a:planar_motion,a:planar_symmetry} imply that the motion takes
place on a plane. Following \cite{mis2021}, we introduce a Hamiltonian
function that includes both the orbital and rotational motions.  Using
the conservation of the angular momentum, the system is described by a
Hamiltonian with 3 degrees of freedom that depends on several
parameters of each ellipsoid, among which there are the equatorial
oblateness and the flattening with respect to the shortest physical
axis. Such parameters are typically small for natural bodies of the
solar system. We refer to this model as the {\sl full} problem, since
it includes the coupling between the orbital and rotational motions.

The potential of the problem can be written as
$V=V_0+\sum_{l=1}^\infty V_{2l}$, where $V_0$ denotes the Keplerian
potential and the terms $V_{2l}$ are proportional to $1/r^{2l+1}$,
where $r$ is the instantaneous distance of the two centers of mass. If
we consider the expansion up to order $l=2$, say $V=V_0+V_2+V_4$, we
obtain that the model includes the coupling of the spins of the two
ellipsoids through the term $V_4$, which contains combinations of the
rotation angles of the two satellites. When the two spins interact, we
refer to the problem as the {\sl full spin-spin} model. If, instead,
we limit the potential to $V=V_0+V_2$, then we obtain two decoupled
systems, the {\sl full spin-orbit models} (see, e.g.,
\cite{bel1966,cel1990,gol1966}). When one of the bodies is spherical,
its spin is uniform and the dynamics of the spin-spin model becomes
very similar to that of the spin-orbit model, but including the terms
in $V_4$.

Next, we introduce another assumption, namely:

\begin{assumption}\label{a:keplerian}
 The orbital motion of the ellipsoids coincides with that of two point
 masses, so that both centers of mass move on coplanar Keplerian
 orbits with eccentricity $e\in[0,1)$ and with a common focus at the
   barycenter of the system.
\end{assumption}

\cref{a:keplerian} implies that the orbit is not affected by the
rotational motion. To the {\sl full} spin-spin and spin-orbit
problems, it corresponds the {\sl Keplerian} spin-spin and spin-orbit
models, described by a Hamiltonian function with an explicit periodic
time dependence.

Note that we are considering rigid bodies only, which means that
dissipative effects due to tidal torques are not considered. We refer
to \cite{cel2008,cel2009,misort2020,mis2021,gol1966} for a description
of the dissipative spin-orbit and spin-spin problems.

\vskip .1in

The previously described models have some symmetries that are a direct
consequence of the mirror symmetries of the ellipsoids. This fact
leads us to introduce the following two types of resonances within the
spin-orbit problem of a single ellipsoid:

\begin{enumerate}
\renewcommand*{\theenumi}{\textbf{R\arabic{enumi}}}
\renewcommand*{\labelenumi}{(\theenumi)}
  \item \label{R1} We call {\sl standard} $m:n$ spin-orbit resonance,
    for some integers $m$, $n$, when the spinning body makes $m$
    rotations during $n$ orbital revolutions;
  \item \label{R2} We call {\sl balanced} $m:2$ spin-orbit resonance,
    for some integers $m$, when the spinning body makes $m/2$ of a
    rotation during one orbital revolution.
\end{enumerate}

We remark that a balanced spin-orbit resonance is also a standard
resonance, but the converse is not always true. For example, a
balanced $2k:2$ resonance for $k\in\Z$ is equivalent to a $k:1$
spin-orbit resonance, but there is not such an equivalence for order
$(2k+1):2$.  Both definitions \eqref{R1} and \eqref{R2} extend to the
spin-spin problem of type $(m_1:n_1,m_2:n_2)$ for integers $m_1$,
$m_2$, $n_1$, $n_2$, when the first ellipsoid is in a $m_1:n_1$
spin-orbit resonance and the second ellipsoid in a $m_2:n_2$
spin-orbit resonance.

We also stress that spin-orbit resonances find many applications in
the solar system; in fact, the Moon is an example of a $1:1$
spin-orbit resonance\footnote{A $2:2$ balanced spin-orbit resonance is
equivalent to a $1:1$ standard spin-orbit resonance.}, since it makes
a rotation in the same period it takes to make an orbit around the
Earth. This is also called a {\sl synchronous} spin-orbit resonance,
which is common to many satellites of other planets, including Mars,
Jupiter, Saturn, Uranus, Neptune. Among the planets, Mercury is locked
in a $3:2$ spin-orbit resonance around the Sun. On the other hand, the
Pluto-Charon system is locked in the double synchronous spin-spin
resonance $(1:1,1:1)$.

\vskip .1in

In this work, we study the behavior of the solutions of the spin-orbit
and spin-spin problems as the parameters and the initial conditions
are varied. In particular, we investigate the boundary conditions that
lead to the existence of symmetric periodic orbits.  Such results (see
Propositions~\ref{prop:restypes} and \ref{prop:resSStypes}) use some
symmetry properties of the equations of motion. We remark that these
symmetries are lost if we include dissipation; however, such periodic
solutions might be continued to the dissipative setting as shown in
\cite{misort2020, mis2021}. Beside the study of the periodic orbits,
we provide the conditions for the existence of quasi-periodic
solutions of the Keplerian version of the spin-spin model.

We also give a qualitative study of the spin-orbit problem as well as
the spin-spin problem with spherical and non-spherical
companion. Within such investigation, we discover some new features,
like the measure synchronization (see, e.g., \cite{ham1999}) for the
spin-spin problem with identical bodies.  Our study leads to analyze
the multiplicity of solutions and the linear stability of the periodic
orbits (compare with \cite{celchi2000}). In general, we find that
there is not a unique solution associated to a particular resonance;
however, for some values of the parameters such a uniqueness
exists. Finally, we provide some results on the comparison between the
full and Keplerian models, as well as on the interaction between the
spin and the orbital motion, motivated by the fact that the coupling
between the rotational and orbital motions has not been much explored
in the literature.

\vskip .1in

This work is organized as follows. In Section~\ref{sec:models} we
present the spin-orbit and spin-spin models. The definition of
resonances and the existence of periodic and quasi-periodic orbits are
given in Section~\ref{sec:periodic}. A qualitative description of the
phase space is given in Section~\ref{sec:graph}. The linear stability
of symmetric periodic orbits is investigated in
Section~\ref{sec:linear}. Finally, a comparison between the full and
Keplerian models is presented in Section~\ref{sec:comparison}.

\section{The models for the spin-orbit and spin-spin coupling}
\label{sec:models}

The aim of this section is to present the so-called {\sl spin-orbit}
and {\sl spin-spin} models, that we are going to introduce as
follows. The assumptions and the notations are given in
Section~\ref{sec:assumptions}; different models, subject to some or
all the assumptions listed in Section~\ref{sec:assumptions}, are
presented in Sections~\ref{sec:full} and \ref{sec:Kepler}.

\subsection{General assumptions}\label{sec:assumptions}

Consider two homogeneous rigid ellipsoids, say $\mathcal E_1$ and
$\mathcal E_2$, with masses $M_1$ and $M_2$, respectively.  Let $
\mathcal A_j< \mathcal B_j< \mathcal C_j$, $j=1,2$, be their principal
moments of inertia with corresponding principal semi-axes $\mathsf
a_j>\mathsf b_j>\mathsf c_j$.  We refer to the {\sl Full Two-Body
  Problem} (hereafter F2BP) as the problem of two rigid bodies
interacting gravitationally (see, e.g., \cite{sch2002,sch2009}). When
the bodies have ellipsoidal shape, we speak of the {\sl ellipsoidal
  F2BP}, where we make the assumptions 1, 2, 3 of
Section~\ref{sec:introduction}:

\begin{figure}[ht]
    \centering
    \scalebox{0.5}{\includegraphics{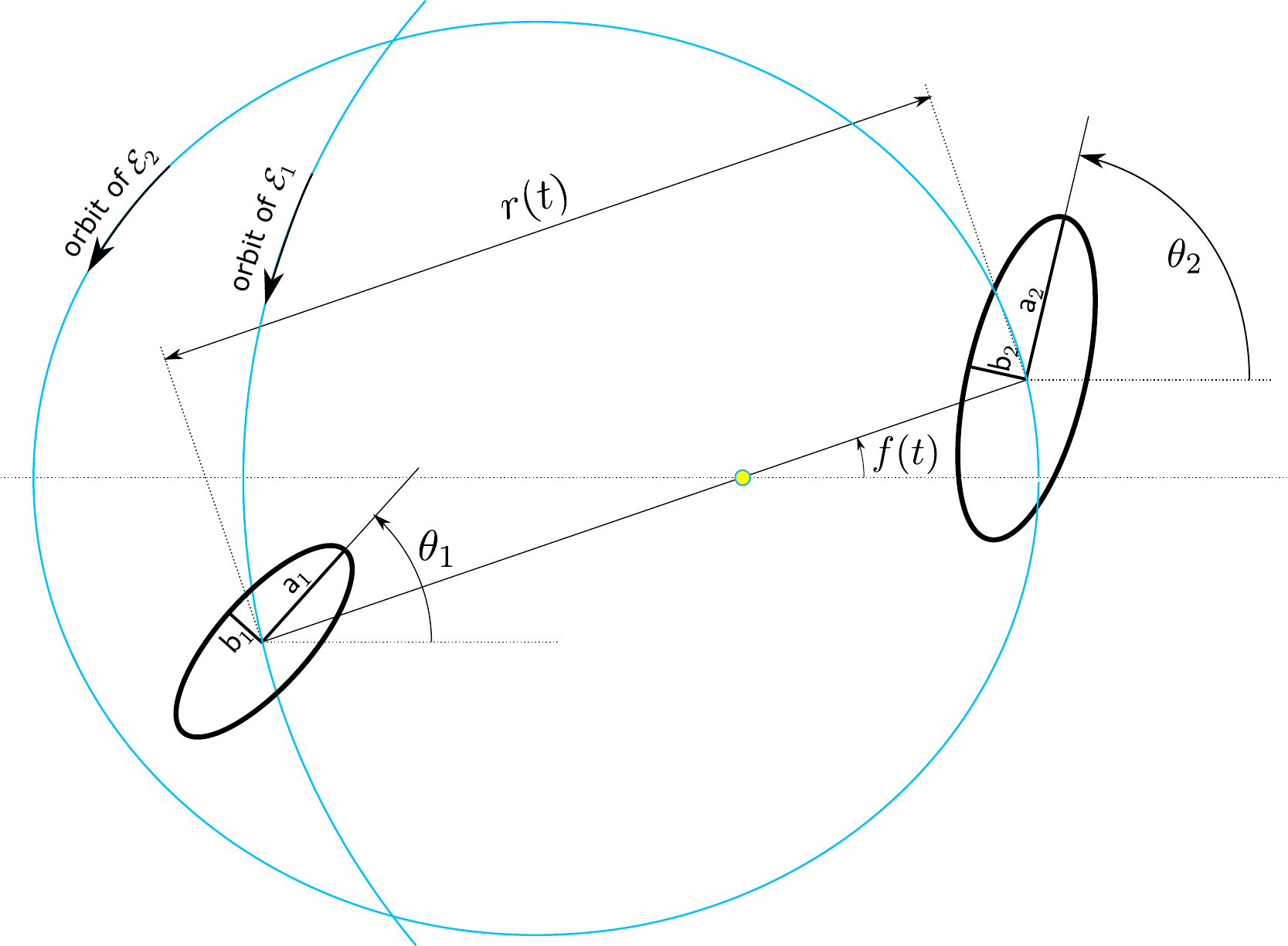}}
    \caption{The planar spin-spin problem considering
      \Cref{a:keplerian,a:planar_motion,a:planar_symmetry}.}
    \label{fig:spin-spin}
\end{figure}

\Cref{a:planar_motion,a:planar_symmetry} guarantee that the problem we
deal with is a planar problem. Additionally, \Cref{a:keplerian}
restricts the problem so that we obtain a model with two degrees of
freedom and a periodic time dependence. This assumption is equivalent
to say that the spin motion will not influence the orbital
motion. Besides, note that we are neglecting the gravitational
contribution of other bodies, we are not considering any dissipative
effect that might arise, for example, from the non-rigidity of the
ellipsoidal bodies and we do not take into account the obliquity,
namely the inclination of the spin-axes with respect to the orbital
plane.

We are going to work with units adapted to the system. If we call
$\tau$ the orbital period of the Keplerian orbit, then we will use
units such that
\[
  M_1+M_2=1,\quad \mathcal C_1+\mathcal C_2=1,\quad \tau=2\pi\ .
\]
Recall Kepler's third law for the Two-Body Problem
\[
  G(M_1+M_2) \parentesis{\frac{\tau}{2\pi}}^2=a^3 \ ,
\]
where $G$ is the gravitational constant, $a$ is the semi-major axis
associated to the motion of the reduced mass of the system, say
$\mu=M_1M_2$ in our units. In consequence, $G=a^3$ in our units.

Let us now define the parameters for each ellipsoid
\[
  d_j=\mathcal  B_j-  \mathcal A_j, \qquad
  q_j=2 \mathcal C_j-  \mathcal B_j-  \mathcal A_j \ ;
\]
the quantity $d_j/ \mathcal C_j$ measures the equatorial oblateness of
each ellipsoid with respect to the plane formed by the directions of
$\mathsf a_j$ and $\mathsf b_j$, whereas $q_j/\mathcal C_j$ measures
the flattening with respect to the direction corresponding to the
$\mathsf c_j$-axis.

Note that if $ \mathcal A_j \le \mathcal B_j\le \mathcal C_j$, then,
in our units there are some bounds for the parameters of the system
given by
\begin{equation}\label{a_j}
    0\le d_j \le \mathcal C_j\le 1,\quad d_j\le q_j\le 2\mathcal
    C_j\le 2,\quad M_j \mathsf a_j^2=\frac{5}{2}(\mathcal C_j+d_j)\le
    5\mathcal C_j \le 5 \ .
\end{equation}
The last relation in \equ{a_j} comes from the fact that the moments of
inertia of an ellipsoid hold the identities
\[
{\mathcal A_j}=\frac{1}{5}M_j({\mathsf b_j^2}+{\mathsf c_j^2}) ,\qquad
{\mathcal B_j}=\frac{1}{5}M_j({\mathsf a_j^2}+{\mathsf c_j^2}) ,\qquad
{\mathcal C_j}=\frac{1}{5}M_j({\mathsf a_j^2}+{\mathsf b_j^2})\ .
\]

\subsection{The full models}\label{sec:full}

First, let us derive the equations of the full models of spin-orbit
and spin-spin coupling, for which only
\Cref{a:planar_motion,a:planar_symmetry} hold, namely we do not
constrain the centers of mass of $\mathcal E_1$ and $\mathcal E_2$ to
move on Keplerian ellipses.

The equations of motion are obtained by computing the Hamiltonian
function through a Legendre transformation of the Lagrangian, say
$L=T-V$, where $T$ is the kinetic energy and $V$ the potential energy
of the system. We split $T$ in two parts, associated respectively to
the orbital and rotational motions, say $T=T_{orb}+T_{rot}$.

Let us identify the orbital plane with the complex plane $\C$,
consider the center of mass of the system fixed in the origin and let
the position of each ellipsoid be $\vec r_j\in \C$. Then, by
definition of the barycenter we have
\[
M_1 \vec r_1+M_2 \vec r_2=\vec 0\ .
\]
If we define the relative position vector $\vec r= \vec r_2-\vec r_1$,
since in our units $M_1 +M_2 =1$, then we have
\begin{equation}\label{r12}
\vec r_1=-M_2 \vec r\ ,\quad \vec r_2=M_1 \vec r\ .
\end{equation}

\begin{figure}[ht]
    \centering
    \scalebox{0.5}{\includegraphics{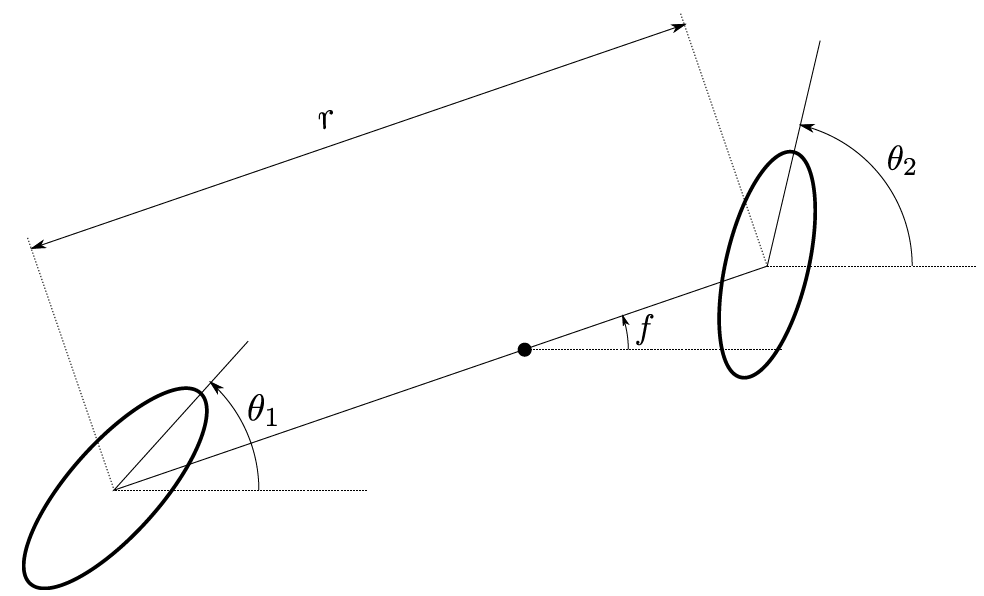}}
    \caption{Generalized coordinates of the full planar model.}
    \label{fig:variables}
\end{figure}

We introduce the Lagrangian generalized coordinates
$(r,f,\theta_1,\theta_2)$, illustrated in \Cref{fig:variables}, in the
following way. The distance $r>0$ and the angle $f$ define the
relative position vector $\vec r = r \exp (if) \in \C $ with respect
to an inertial reference frame with a fixed $x$-axis. The angles
$\theta_1$, $\theta_2$ provide the orientation of the axes $\mathsf
a_1$, $\mathsf a_2$ with respect to the $x$-axis. The variables $r$
and $f$ define the orbital motion of the system and $\theta_j$ the
spin motion of the ellipsoid $\mathcal E_j$.

To write the kinetic energy, we notice that we can express $\vec r$ in
components as ${\vec r}=(r\cos f,r\sin f)$, which gives $\dot {\vec
  r}=(\dot r\cos f-r\dot f\sin f, \dot r\sin f+r\dot f\cos f)$. Then,
using \equ{r12}, the total orbital and the rotational kinetic energies
are given by
\begin{equation*}
  T_{orb}=\frac{1}{2} (M_1 \dot{\vec r}_1 ^2 +M_2 {\dot{\vec r}_2}^2)
  = \frac{\mu}{2} \dot{\vec r}^2 = \frac{\mu}{2} (\dot r^2 + r^2 \dot
  f^2),\quad T_{rot}=\frac{1}{2} \mathcal C_1 \dot \theta_1^2 +
  \frac{1}{2} \mathcal C_2 \dot \theta_2^2\ ,
\end{equation*}
where $\mu=M_1M_2$ in our units. The Lagrangian is given by
\[
  L(r,f,\theta_1,\theta_2,\dot r,\dot f,\dot\theta_1,\dot\theta_2)
  =T_{orb}(r,\dot r,\dot f) + T_{rot}(\dot \theta_1,\dot
  \theta_2)-V(r,f,\theta_1,\theta_2)\ ,
\]
where, according to \cite{mis2021,bou2017}, the full expansion of the
potential energy of the system $V=V(r,f,\theta_1,\theta_2)$ takes the
form
\begin{equation}\label{V_full}
    V(r,f,\theta_1,\theta_2)=-\frac{G M_1 M_2 }{ r}
    \displaystyle\sum_{\substack{(l_1,m_1)\in \Upsilon \\ (l_2,m_2)\in
        \Upsilon}} \frac{\Lambda_{l_2,m_2}^{l_1,m_1}}{r^{2(l_1+l_2)}}
    \cos (2m_1 (\theta_1-f)+2m_2 (\theta_2-f)) \ ,
\end{equation}
where
\begin{equation*}
    \Upsilon= \{ (l,m)\in \mathbb Z^2 \colon \ 0\le |m| \le l \}\ ,
\end{equation*}
and the constants $\Lambda_{l_2,m_2}^{l_1,m_1}$ are defined in
\Cref{app:coef}; we refer to \cite{mis2021} for full details.

Let us define the momenta as
\begin{equation*}
    p_r=\partial_{\dot r} L = \mu \dot r,\quad
    p_f=\partial_{\dot f} L=\mu r^2\dot f,\quad
    p_j = \partial_{\dot \theta_j} L = \mathcal C_j \dot \theta_j\ ;
\end{equation*}
then, the Hamiltonian of the system becomes
\begin{equation}\label{H_theta}
    H(r,f,\theta_1,\theta_2,p_r,p_f,p_1,p_2)=\frac{p_r^2}{2\mu}
    + \frac{p_f^2}{2\mu r^2}
    + \frac{p_1^2}{2\mathcal C_1} + \frac{p_2^2}{2\mathcal C_2}
    +V(r,f,\theta_1,\theta_2) \ .
\end{equation}
Consequently, the equations of motion are
\begin{align}
\label{orbital_eqs}
    \dot r = \frac{p_r}{\mu}, \quad
    \dot f &= \frac{p_f}{\mu r^2},\quad
    \dot p_r =  \frac{p_f^2}{\mu r^3} - \partial_r V , \quad
    \dot p_f =  - \partial_{f} V
\intertext{and}
\label{spin-spin_V}
\dot \theta_j &=\frac{p_j}{\mathcal C_j},\quad
\dot p_j =  - \partial_{\theta_j} V \ .
\end{align}
We remark that with these variables, the problem splits in two parts:
equations \equ{orbital_eqs} describe the {\sl orbital} motion, while
equations \equ{spin-spin_V} describe the {\sl rotational} motion. The
evolution of both parts is coupled through the potential $V$.

The potential $V$, expanded in \cref{V_full}, can be written in a
perturbative way as
\begin{equation}\label{V0per}
  V(r,f,\theta_1,\theta_2) = V_0(r) + V_{per}(r,f,\theta_1,\theta_2),
  \quad V_0(r)=-\frac{GM_1M_2}{r}\ .
\end{equation}

We remark that the term $V_0$ corresponds to the classical Keplerian
form for the potential and the term $V_{per}$ provides the coupling
between the spin and the orbital motions. Moreover, we can expand
$V_{per}$ as $ V_{per}=\sum_{l=1}^\infty V_{2l}\ , $ where $V_{2l}$
are suitable terms proportional to $1/r^{2l+1}$. A truncation of the
expansion of $V_{per}$ will result in an approximated dynamics of our
system. The explicit expressions of the first two terms of such a
expansion are given in \cite{mis2021} and we report them here:
\begin{eqnarray}\label{V024}
    V_2&=&-\frac{GM_2}{4r^3} \parentesis{q_1+3d_1\cos(2\theta_1-2f)}
    -\frac{GM_1}{4r^3} \parentesis{q_2+3d_2\cos(2\theta_2-2f)}\ ,\nonumber\\
    V_4&=&-\frac{3G}{4^3 r^5} \Bigg \{
    12q_1q_2 + \frac{15}{7} [\frac{M_2}{M_1}d_1^2+2\frac{M_2}{M_1}q_1^2 + \frac{M_1}{M_2}d_2^2+2\frac{M_1}{M_2}q_2^2]\nonumber\\
    &&\hspace{1.5cm} +d_1 M_2\left \{ [20\frac{q_2}{M_2} + \frac{100}{7}\frac{q_1}{M_1}]\cos (2\theta_1-2f) +25 \frac{d_1}{M_1}\cos (4\theta_1-4f)
    \right \}\nonumber\\
    &&\hspace{1.5cm}  +d_2 M_1 \left \{[20\frac{q_1}{M_1} + \frac{100}{7}\frac{q_2}{M_2}]\cos (2\theta_2-2f) + 25 \frac{d_2}{M_2}\cos (4\theta_2-4f)
    \right \}\nonumber\\
    &&\hspace{1.5cm} +6d_1 d_2 \cos (2\theta_1-2\theta_2)+70d_1 d_2  \cos (2\theta_1+2\theta_2-4f) \Bigg \}\ .
\end{eqnarray}

From now on, we will refer to \equ{orbital_eqs} and \equ{spin-spin_V}
as the full models: {\sl full spin-orbit model} if we take
$V_{per}=V_2$, in which the angles $\theta_1$, $\theta_2$ appear in
different trigonometric terms, and {\sl full spin-spin model} if
$V_{per}=V_2+V_4$, which contains trigonometric terms with
combinations of the rotation angles $\theta_1$, $\theta_2$. These
names are motivated by the well-known spin-orbit model,
\cite{cel2010}, and the spin-spin model from \cite{mis2021}.  If we
consider the models under \Cref{a:keplerian} which gives a constraint
on the orbit, we speak of {\sl Keplerian spin-orbit model} and {\sl
  Keplerian spin-spin model} (compare with \Cref{sec:Kepler}).

\subsubsection{Conservation of the angular momentum}

Note that the Hamiltonian $H$ in \cref{H_theta} is invariant under the
transformation $(r,f,\theta_1, \theta_2) \mapsto (r, f+\delta f,
\theta_1+\delta f, \theta_2+\delta f)$, where $\delta f$ is an
infinitesimal angular increase, because the angular arguments of
$V(r,f,\theta_1,\theta_2)$ only depend on the differences $\theta_1-f$
and $\theta_2-f$. This symmetry is related, by Noether's theorem
\cite{gol1980}, with a conserved quantity, say, the total angular
momentum $p_f+p_1+p_2$.  This can be proved through the following
change of variables

\begin{equation*}
    (r,f,\theta_1,\theta_2,p_r,p_f,p_1,p_2)\mapsto
    (r,f,\phi_1,\phi_2,p_r,P_f,p_1,p_2)\ ,
\end{equation*}

\noindent where

\begin{equation}\label{canonical}
    \phi_j=\phi_j(f,\theta_j)=\theta_j-f,\quad
    P_f=P_f(p_f,p_1,p_2)=p_f+p_1+p_2\ .
\end{equation}

The transformation of coordinates \cref{canonical} is canonical, since
\[
\d r\wedge \d p_r + \d f\wedge \d p_f + \sum_{j=1}^2\d \theta_j\wedge
\d p_j = \d r\wedge \d p_r + \d f\wedge \d P_f + \sum_{j=1}^2\d
\phi_j\wedge \d p_j\ .
\]
Then, the Hamiltonian \equ{H_theta} in this new set of variables is
given by
\begin{equation}\label{H_phi}
    \mathcal H(r,f,\phi_1,\phi_2,p_r,P_f,p_1,p_2)=\frac{p_r^2}{2\mu}
    + \frac{(P_f-p_1-p_2)^2}{2\mu r^2}
    + \frac{p_1^2}{2\mathcal C_1} + \frac{p_2^2}{2\mathcal C_2}
    +\mathcal V(r,\phi_1,\phi_2)\ ,
\end{equation}
where $\mathcal V(r,\phi_1,\phi_2)=V(r,f,\phi_1+f,\phi_2+f)$. Now it
is clear that $f$ is an ignorable variable in \cref{H_phi} and that
$P_f$ is a constant of motion, corresponding to the total angular
momentum of the system.

In summary, in the evolution of the system, there is a transfer of
angular momentum between the spin part, given for each body by $
p_j=\mathcal C_j \dot \theta_j$, and the orbital part, given by $p_f =
\mu r^2\dot f$.

\subsection{The Keplerian models}\label{sec:Kepler}

In this section, we introduce the \Cref{a:keplerian} to the model of
Section~\ref{sec:full}.  From \cref{orbital_eqs} and
\cref{spin-spin_V}, it is straightforward to constrain the orbit to be
Keplerian; only in the orbital part we retain just the term $V_0$ of
the potential $ V=V_0+ V_{per}$ in \equ{V0per}, thus obtaining the
following equations:
\begin{align}\label{orbital_eqs_kep}
  \dot r = \frac{p_r}{\mu}, \quad
  \dot f &= \frac{p_f}{\mu r^2},\quad \dot p_r =  \frac{p_f^2}{\mu r^3} - \partial_r V_0 , \quad
  \dot p_f =  - \partial_{f} V_0 = 0,
\intertext{and}
\label{spin-spin_V_trun}
  \dot \theta_j &=\frac{p_j}{\mathcal C_j},\quad
  \dot p_j =  - \partial_{\theta_j} V=- \partial_{\theta_j}V_{per}\ .
\end{align}

A convenient procedure to numerically integrate the equations of
motion \equ{spin-spin_V_trun} is presented in \Cref{app:diff}.

\subsubsection{Orbital motion}
Note that since $\partial_{\theta_j} V_0=0$, the system
\cref{orbital_eqs_kep} is now decoupled from
\cref{spin-spin_V_trun}. Moreover, \cref{orbital_eqs_kep} is the
Kepler problem, whose solutions depend on the eccentricity $e$ and the
semi-major axis $a$ of the orbit. Here we assume for simplicity that
the orbit is a $2\pi$-periodic Keplerian ellipse of eccentricity
$e\in[0,1)$ with focus at the origin and with the periapsis on the
  positive $x$-axis.

Since the orbital period is $2\pi$, then, we can take the time $t$ to
coincide with the {\sl mean anomaly}. We denote by $u$ the {\sl
  eccentric anomaly}, which, in our units, is related to the mean
anomaly by Kepler's equation
\begin{equation}\label{t}
t=u-e\sin u\ .
\end{equation}
The orbital radius is related to $u$ by
\begin{equation}\label{r}
r=a(1-e \cos u)\ .
\end{equation}
We can write the vector $\vec r\in \C$ in terms of the eccentric
anomaly also as
\begin{equation}\label{f}
r\exp (i f)=a(\cos u-e + i \sqrt{1-e^2} \sin u)\ .
\end{equation}
Note that for $t=0$ we assumed, without loss of generality, that
$f=u=0$, and consequently, $f=u=\pi$ when $t=\pi$. From \cref{f} we
obtain the following useful relations between $f$ and $u$
\begin{equation}\label{ff}
\cos f = \frac{\cos u-e}{1-e \cos u}\ ,\qquad \sin f =
\frac{\sqrt{1-e^2}\sin u}{1-e \cos u}\ .
\end{equation}

With the previous definitions, the Keplerian orbit of eccentricity $e$
and semi-major axis $a$ is given by the functions
\begin{equation}\label{kepler_orbit}
    r=r(t;a,e),\quad f=f(t;e),\quad
    p_r=p_r(t;a,e),\quad p_f=p_f(a,e)=\mu a^2\sqrt{1-e^2}\ ,
\end{equation}
that correspond to the solution of equations \cref{orbital_eqs_kep}
generated by the initial conditions
\begin{equation}\label{kepler_in}
    r(0)=a(1-e),\quad p_r(0)=0,\quad f(0)=0,\quad p_f(0)=\mu
    a^2\sqrt{1-e^2}\ .
\end{equation}

\vskip .1in

\subsubsection{Spin motion}\label{sub:kepler_spin}
The spin motion is described by equations \cref{spin-spin_V_trun} with
the Keplerian periodic input \cref{kepler_orbit} given implicitly by
\cref{t,r,f}. This motion can be described by the non-autonomous
Hamiltonian
\begin{equation}\label{H_K}
    H_K(t,\theta_1,\theta_2,p_1,p_2) =
    H(r(t;a,e),f(t;e),\theta_1,\theta_2,p_r(t;a,e),p_f(a,e),p_1,p_2)\ ,
\end{equation}
where $H(r,f,\theta_1,\theta_2,p_r,p_f,p_1,p_2)$ is the Hamiltonian of
the full model defined in \cref{H_theta}. The Hamiltonian \equ{H_K} is
hence of the form
\begin{equation}\label{HK}
H_K(t,\theta_1,\theta_2,p_1,p_2)={{p_1^2}\over {2\mathcal
    C_1}}+{{p_2^2}\over {2\mathcal C_2}}+W(t,\theta_1,\theta_2)\ ,
\end{equation}
where the potential $W$ is $2\pi$-periodic in $t$ and $\pi$-periodic
in $\theta_1$ and $\theta_2$. The equations of motion
\cref{spin-spin_V_trun} take the form
\begin{equation}\label{spin_spin_ham}
    \dot \theta_j = \frac{p_j}{\mathcal C_j},\quad \dot p_j =
    -\partial_{\theta_j} W(t,\theta_1,\theta_2)\ .
\end{equation}

Let us define the non-dimensional parameters of the model:
\begin{equation}\label{lambda_j}
    \lambda_j = 3\frac{\mu}{M_j} \frac{d_j}{\mathcal C_j} \ , \qquad
    \sigma_j=\frac{1}{3}\frac{\mathcal C_j}{\mu a^2},\qquad \hat
    q_j=\frac{q_j}{M_j a^2}\ ,
\end{equation}
where $\lambda_j$ represents the equatorial oblateness of $\mathcal
E_j$; $\sigma_j$ is the ratio between the moment of inertia of
$\mathcal E_j$ and the orbital one; and $\hat q_j$ measures the
flattening of $\mathcal E_j$ with respect to the size of the
orbit. Note that the parameters in \cref{lambda_j} are small for
bodies that are close to spherical. Besides, not all the parameters
defined previously are free, because we have the constraint $\mathcal
C_1\sigma_2=\mathcal C_2\sigma_1$.

If we take $V_{per} = V_2$, then the system \cref{spin-spin_V_trun}
becomes
\begin{equation}\label{spin_orbit_j}
  \ddot \theta_j + \frac{\lambda_j}{2} \parentesis{\frac{a}{
      r(t;e)}}^{3} \sin(2 \theta_j-2f(t;e))=0 ,\qquad j=1,2\ ,
\end{equation}
that is a system of two uncoupled spin-orbit problems. Each of these
problems depends just on two parameters: $(e,\lambda_j)$. On the other
hand, if $V_{per} = V_2+V_4$, from \cref{spin-spin_V_trun} and
\equ{V024} we obtain the following system for $j=1,2$,
\begin{multline}\label{spin_spin_j}
    0=   \ddot \theta_j +
    \frac{\lambda_j}{2} \Bigg \{   \parentesis{\frac{a}{ r(t;e)}}^{3} \sin(2 \theta_j-2f(t;e)) + \\
    +  \parentesis{\frac{a}{ r(t;e)}}^{5} \Bigg [
    \frac{5}{4} \parentesis{ \hat q_{3-j}+ \frac{5}{7}  \hat q_j} \sin(2 \theta_j-2f(t;e))
    +\frac{25 }{8}  \lambda_{j}\sigma_{j}  \sin(4 \theta_j-4f(t;e))  \\
    + \lambda_{3-j}\sigma_{3-j}\parentesis{\frac{3}{8}  \sin(2 \theta_j-2\theta_{3-j})
    + \frac{35}{8}  \sin(2 \theta_{3-j}+2\theta_j-4f(t;e))}
    \Bigg ] \Bigg\}\ ,
\end{multline}
that we call spin-spin problem. From the previous discussion, this
model depends on seven independent parameters\footnote{In
\cite{mis2021} there was an error because $\hat q_1$ and $\hat q_2$
are actually independent, it is not always true that $\mathcal
C_1\lambda_1 \hat q_2=\mathcal C_2\lambda_2 \hat q_1$, but it can be
regarded as an additional constraint.} $(e;\mathcal C_1, \lambda_1,
\lambda_2, \sigma_1,\hat q_1, \hat q_2)$.  Note that in
\cref{spin_spin_j}, the coupling between the dynamics of $\theta_1$
and $\theta_2$ is given by $\sigma_1$ and $\sigma_2$. Moreover, if
$\hat q_j=\sigma_j = 0$, the spin-spin problem \cref{spin_spin_j} is
reduced to a pair of spin-orbit problems \cref{spin_orbit_j}.

Let us now consider \cref{spin_spin_j} in the case that $\mathcal E_2$
is a sphere, that is, $d_2=q_2=0$.  Then, $\lambda_2=\sigma_2=\hat q_2
=0$, which implies that $\mathcal E_2$ is in uniform rotation
$\theta_2(t)=\dot \theta_2(0)t + \theta_2(0)$. The dynamics of
$\theta_1$ is uncoupled from $\theta_2$ and is given by
\begin{multline}\label{spin_spin_1}
    0=   \ddot \theta_1 +
    \frac{\lambda_1}{2} \Big \{   \parentesis{\frac{a}{ r(t;e)}}^{3} \sin(2 \theta_1-2f(t;e)) + \\
    +  \parentesis{\frac{a}{ r(t;e)}}^{5} \Big [
    \frac{25 \hat q_1}{28}  \sin(2 \theta_1-2f(t;e))
    +\frac{25  \lambda_1 \sigma_1}{8}  \sin(4 \theta_1-4f(t;e))
    \Big ] \Big\}\ ,
\end{multline}
that is a spin-orbit problem up to order $1/r^5$. An equivalent system
was studied previously in \cite{jaf2016b}.  Here the parameters
$\sigma_1$ and $\hat q_1$ perturb the framework of the spin-orbit
problem \cref{spin_orbit_j}.

\subsubsection{Reversing symmetries.}
The equations for the spin motion in the Keplerian models have some
reversing symmetries, i.e., transformations in the phase space that
keep invariant the equations of motion with a time reversal.

\begin{definition}
Consider the differential equation
  \begin{equation}\label{vecF}
     \frac{\d \vec x}{\d t} = \vec F( \vec x),\quad \vec x\in \R^n\ ,
  \end{equation}
  and let $\vec x(t;\vec x_0)$ be the solution of
  \cref{vecF} with initial condition $\vec x(0;\vec x_0)=\vec x_0$.
  \begin{enumerate}
    \item A transformation $R\colon\R^n\to \R^n$ is called a reversing
      symmetry of \cref{vecF} if
      \begin{equation*}
         \frac{\d R(\vec x)}{\d t} = -\vec F( R(\vec x))\ .
      \end{equation*}
    \item The fixed point set of $R$ is given by
      $\operatorname{Fix}(R) = \{\vec x\in \R^n\colon \ R(\vec x)=\vec
      x\}$.
    \item An orbit $o(\vec x_0)=\{\vec x(t;\vec x_0)\colon \ t\in \R
      \}$ is $R$-symmetric if $R(o(\vec x_0))=o(\vec x_0)$.
  \end{enumerate}
\end{definition}

The system \cref{spin_orbit_j} with $j=1$ can be written in the
autonomous form \cref{vecF} with
\begin{equation}\label{F_spin-orbit}
  \vec x = (t,\theta_1,\dot \theta_1), \quad \vec F (\vec x) =
  \parentesis{1,\dot \theta_1,- \frac{\lambda_1}{2}
  \parentesis{\frac{a}{ r(t;e)}}^{3} \sin(2 \theta_1-2f(t;e))}\ .
\end{equation}
One can easily check that each transformation defined by
\begin{equation}\label{RR}
    R_{\alpha,\beta}(\vec x) = (2\alpha-t,2\beta-\theta_1,\dot
    \theta_1), \quad \text{with} \quad (\alpha,\beta) \in \pi \Z
    \times \frac{\pi}{2} \Z\ ,
\end{equation}
is a reversing symmetry of equation \cref{vecF} with $\vec F$ as in
\equ{F_spin-orbit} because
\[
 f(2\alpha-t;e) = 2\alpha-f(t;e),\quad
 r(2\alpha-t;e)= r(t;e)\ .
\]
The same is true replacing in \equ{RR} the quantities $\theta_1$,
$\dot\theta_1$ by $\theta_2$, $\dot\theta_2$.

On the other hand, the spin-spin problem in its Hamiltonian
formulation is given by \cref{spin_spin_ham}, that can be written as
\cref{vecF} with
\begin{equation}\label{F_spin-spin}
    \vec x = (t,\theta_1,\theta_2,p_1,p_2), \quad \vec F (\vec x) =
    \parentesis{1,\frac{p_1}{\mathcal C_1},\frac{p_2}{\mathcal C_2},-
      \partial_{\theta_1}W(t,\theta_1,\theta_2),-
      \partial_{\theta_2}W(t,\theta_1,\theta_2)}\ .
\end{equation}
Each transformation defined by
\begin{equation}\label{RRR}
    R_{\alpha,\beta_1,\beta_2}(\vec x) =
    (2\alpha-t,2\beta_1-\theta_1,2\beta_2-\theta_2,p_1,p_2), \quad
    \text{with} \quad (\alpha,\beta_1,\beta_2) \in \pi \Z \times
    \frac{\pi}{2} \Z\times \frac{\pi}{2} \Z\ ,
\end{equation}
is a reversing symmetry of equation \cref{vecF}.

In the following sections we are going to emphasize the study of
orbits that are symmetric with respect to the mentioned reversing
symmetries using the following lemma.

\begin{lemma}[from Theorem 4.1 in \cite{lamb1998}] \label{lemmafix}
  An orbit of \cref{vecF} is $R$-symmetric if, and only if, it
  intersects the fixed point set $\operatorname{Fix} (R)$.
\end{lemma}

\section{Periodic and quasi-periodic solutions of the Keplerian
models} \label{sec:periodic}

In this section we provide the definitions of spin-orbit and spin-spin
resonances, giving some results on the existence of periodic orbits
(Section~\ref{sec:resonances}) and KAM tori (Section~\ref{sec:KAM}).

\subsection{The spin-orbit and spin-spin resonances}
\label{sec:resonances}

The periodic solutions of the Keplerian models presented in the
Section \ref{sec:Kepler} correspond to resonances between the orbital
motion and the spin motion. The expansion of the potential in
\equ{V024} contains much interesting information concerning such
resonances. First, let us introduce the definition of spin-orbit
resonances.

\begin{definition}\label{def:so_res}
We say that the ellipsoid $\mathcal E_1$ is in a standard spin-orbit
resonance of order $m:n$ with $m\in\Z,n\in\Z \setminus \{0\}$, if
\begin{equation}\label{resSO}
    \theta_1(t+2\pi n)=\theta_1(t)+2\pi m\ .
\end{equation}
The associated resonant angle $\psi_1^{m:n}(t) = m t- n \theta_1(t)$
is a periodic function of period $2\pi n$. The same definition holds
for $\mathcal E_2$.
\end{definition}

Recalling that we have normalized the mean motion to unity, we remark
that \Cref{def:so_res} states that the ratio of the orbital period of
$\mathcal E_j$ over its period of rotation is $m/n$,
$n\ne0$. Additionally, according to \Cref{def:so_res}, the resonance
$m:n$, $n\ne0$, is also of order $km:kn$, $k\in \Z \backslash\{0\}$,
but the converse is not true in general. For example, with
\cref{resSO}, the resonance $1:1$ is also of order $2:2$, but a
resonance $2:2$ may not be of order $1:1$. We can say that the
resonance $m:n$ is of higher order than the resonance $m':n'$ if
$m/n>m'/n'$. This will be denoted using the notation $m:n>m':n'$.

Next, we introduce a different definition of spin-orbit resonance.

\begin{definition}\label{def:so_res_bal}
  We say that the ellipsoid $\mathcal E_1$ is in a \textit{balanced}
  spin-orbit resonance of order $m:2$, $m\in\Z$, if
  \begin{equation}\label{resSO_bal}
     \theta_1(t+2\pi)=\theta_1(t)+ m\pi\ .
  \end{equation}
  In this case, the resonant angle $\psi_1^{m:2}(t)$ is
  $2\pi$-periodic. The same definition holds for $\mathcal E_2$.
\end{definition}
Notice that the two notions \equ{resSO} and \equ{resSO_bal} are not
equivalent: \equ{resSO_bal} implies \equ{resSO} for $m:2$, but the
converse is not true. Actually, note that a balanced $2k:2$ resonance,
with $k\in \Z$, is a spin-orbit resonance of order $k:1$. This new
definition was motivated by \cite{bel1975}, where the solutions
associated to the resonance $3:2$ of the spin-orbit problem (say,
\cref{spin_orbit_j} with $j=1$) were studied numerically. Basically,
they found out that the solutions satisfying \equ{resSO}, but not
\equ{resSO_bal}, appear only for large values of $\lambda_1$ ($\gtrsim
1$), also, for a given point $(e,\lambda_1)$ the solutions appear in
multiplets, and finally, the corresponding resonant angles have large
amplitudes ($|\psi_1^{3:2}(t)| \gtrsim 0.75$), see Table 1 in
\cite{bel1975}. On the other hand, solutions that obey \equ{resSO_bal}
exist for any point in the $(e,\lambda_1)$-plane, including large
regions of uniqueness of solution and resonant angle with small
amplitude. Note that such amplitude is a measure of the deviation of
the solution with respect to the uniform rotation of angular velocity
$\frac{3}{2}t$. In \Cref{def:so_res_bal} we generalize these two types
of resonances for any order $m:2$, since this let us determine the
main resonances in the first orbital revolution.

From \cite{dev1958}, for instance, we know that a good tool to study a
differential equation that has some reversing symmetries is by
studying the periodic orbits that are invariant under such
transformations.

In the next proposition we provide some boundary conditions that
characterize the symmetric orbits in spin-orbit resonances.

\begin{proposition}\label{prop:restypes}
  The following statements hold for the spin-orbit problem
  \cref{spin_orbit_j}, with $j=1$ (ellipsoid $\mathcal E_1$):
  \begin{enumerate}
     \item \label{prop:restypes-1} Any $R_{\alpha,\beta}$-symmetric
       orbit, with $R _{\alpha,\beta}$ defined in \eqref{RR},
       associated to a $m:n$ spin-orbit resonance is equivalent to a
       solution that satisfies the following Dirichlet conditions:
        \begin{equation}\label{Dirichlet_odd}
            \theta_1(\alpha)=\beta,\quad \theta_1(\alpha+n\pi)=\beta +
            m\pi\ ,
        \end{equation}
        with $\alpha\in\{0,\pi\}$ and $\beta\in \{0,\frac{\pi}{2}\}$
        (four combinations). Moreover, such solution satisfies the
        following symmetry property
        $\theta_1(t)=2\beta-\theta_1(2\alpha-t)$.

     \item \label{prop:restypes-2} There are two independent types of
       $R_{\alpha,\beta}$-symmetric orbits representing a balanced
       $m:2$ spin-orbit resonance and are given by:
       \begin{align}
           \text{Type } 0: \qquad \theta_1(0) &=0,\quad \theta_1(\pi)
           = \frac{m\pi}{2}\ , \label{Dirichlet_I} \\ \text{Type }
           \frac{\pi}{2}: \qquad \theta_1(0) &= \frac{ \pi}{2},\quad
           \theta_1(\pi) =\frac{(m+1)\pi}{2}\ .\label{Dirichlet_II}
       \end{align}
       Moreover, the corresponding symmetry relations are:
       $\theta_1(t)=-\theta_1(-t)$ for type I and
       $\theta_1(t)=\pi-\theta_1(-t)$ for type II.
    \end{enumerate}
    The same is true for \cref{spin_orbit_j}, with $j=2$ (ellipsoid
    $\mathcal E_2$), and also for \cref{spin_spin_1}.
\end{proposition}

\begin{proof}
Let us apply \Cref{lemmafix} to \cref{vecF} with $\vec F(\vec x)$
given by \cref{F_spin-orbit}. The fixed point set of each reversing
symmetry $R_{\alpha,\beta}$ is $\operatorname{Fix} (R_{\alpha,\beta})=
\{ (t,\theta_1,\dot \theta_1)\colon \ t=\alpha,\theta_1=\beta\}$,
then, the symmetric orbits can be found with initial conditions
$\theta_1(\alpha)=\beta$. Since $\vec F(\vec x)$ is $2\pi$-periodic in
$t$ and $\pi$-periodic in $\theta_1$, it is enough to consider
$\alpha\in\{0,\pi\}$ (the periapsis and the apoapsis) and $\beta\in
\{0,\frac{\pi}{2}\}$.

Now, since $R_{\alpha,\beta}$ is a reversing symmetry, if
$\theta_1(t)$ is a solution of \cref{spin_orbit_j} with $j=1$, so
it is $\psi(t)=2\beta-\theta_1(2\alpha-t)$. Additionally, if
$\theta_1(\alpha) = \beta$, then, both solutions coincide, so the
symmetry relation $\theta_1(t)=2\beta-\theta_1(2\alpha-t)$ holds
for it. Now, if $\theta_1(t)$ is in a $m:n$ spin-orbit resonance,
then, replacing $t=\alpha-n\pi$ in \cref{resSO}, we get
\[
\theta_1(\alpha+n\pi)=\theta_1(\alpha-n\pi)+2\pi m\ .
\] From the
symmetry relation we get additionally that
\[
\theta_1(\alpha-n\pi)=2\beta-\theta_1(\alpha+n\pi)\ .
\]
\noindent Combining the two expressions we prove \cref{Dirichlet_odd}.

Now let us prove the converse, that a solution $\theta_1(t)$
satisfying the conditions \cref{Dirichlet_odd} is in $m:n$ spin-orbit
resonance. Let the initial conditions of such solution be
\begin{equation}\label{init_res}
    \theta_1(\alpha+n\pi)=\beta + m\pi, \quad
    \dot\theta_1(\alpha+n\pi)=\tilde\beta \ ,
\end{equation}
for some real constant $\tilde\beta$.  Since $\theta_1(\alpha)=\beta$,
the solution has the symmetry relations
$\theta_1(t)=2\beta-\theta_1(2\alpha-t)$ and $\dot \theta_1(t)=\dot
\theta_1(2\alpha-t)$. Using such relations we get that
\begin{equation}\label{init_res2}
    \theta_1(\alpha-n\pi)=\beta - m\pi, \quad
    \dot\theta_1(\alpha-n\pi)=\tilde\beta\ .
\end{equation}
The initial conditions \cref{init_res} and \cref{init_res2} show that,
while the time increases $2\pi n$, the angle $\theta_1$ increases in
$2\pi m$, and the angular velocity is the same for both cases, then,
the solution is periodic. With this, we have proved
\cref{prop:restypes-1}.

Let us now consider the balanced $m:2$ case in
\cref{prop:restypes-2}. Note that, using the definition
\cref{resSO_bal} instead of \cref{resSO}, we can follow the same
procedure as before to prove that a balanced $m:2$ solution satisfies
the following boundary conditions
\begin{equation}\label{alphabeta_bal}
    \theta_1(\alpha)=\beta,\quad \theta_1(\alpha\pm\pi)=\beta \pm
    \frac{m\pi}{2}\ ,
\end{equation}
and the symmetry relation $\theta_1(t) =
2\beta-\theta_1(2\alpha-t)$. Then, we get that the four types are in
this case \cref{Dirichlet_I}, with $\alpha=0,\beta=0$;
\cref{Dirichlet_II}, with $\alpha=0,\beta=\pi/2$;
\[
  \text{Type } 0': \quad \theta_1(0)=- \frac{m \pi}{2},\quad
  \theta_1(\pi) =0\ ,
\]
with $\alpha=\pi,\beta=0$, and
\[
  \text{Type } \frac{\pi}{2}':\quad
  \theta_1(0)=\frac{(1-m)\pi}{2},\quad \theta_1(\pi) =\frac{\pi}{2}\ ,
\]
with $\alpha=\pi,\beta=\pi/2$. Since an ellipsoid has a mirror
symmetry with respect to any plane containing a pair of semi-axes, the
angle $\theta_1$ is equivalent to $\theta_1+k\pi$, $k\in\Z$. In
consequence, if $m=2k_1, k_1\in \Z$, type $0'$ is equivalent to type 0
and type $\frac{\pi}{2}'$ is equivalent to type
$\frac{\pi}{2}$. Likewise, for $m=2k_2+1$, $k_2\in \Z$, type $0'$ is
equivalent to type $\frac{\pi}{2}$ and type $\frac{\pi}{2}'$ is
equivalent to type 0. Then, for resonances of order $m:2$, $m\in \Z$,
we will take types 0 and $\frac{\pi}{2}$ as representatives. With this
we have proved \cref{prop:restypes-2}.

The previous facts rely only on the symmetries and the periodicity of
equation \cref{spin_orbit_j}, including the discussion in
\cite{celchi2000}. Since equation \cref{spin_spin_1} for the spin
motion of $\mathcal E_1$, with $\mathcal E_2$ spherical, has exactly
the same properties, then, the proof above is also valid in such case.
\end{proof}

\Cref{prop:restypes} let us characterize all the balanced spin-orbit
resonances in the first half of an orbital revolution, additionally,
the corresponding solutions have a certain symmetry relation. In the
generalization to spin-spin resonances we want to combine different
spin-orbit resonances and we will use the boundary conditions in the
same time interval.

\begin{remark}\label{rem:negative}
 Note that solutions of type $\frac{\pi}{2}$ can be recovered with the
 conditions of type 0 by considering negative values of
 $\lambda_1$. More precisely, solutions of type $\frac{\pi}{2}$ of
 equation \cref{spin_orbit_j} with $\lambda_1=\lambda_*>0$ are
 equivalent to solutions of \cref{spin_orbit_j} with
 $\lambda_1=-\lambda_*$ satisfying conditions of type 0. The same
 holds for \cref{spin_spin_1}.
\end{remark}

\begin{remark}
 \Cref{prop:restypes} gives a way to numerically search for balanced
 resonances. Indeed, \cref{Dirichlet_I,Dirichlet_II} can be used to
 apply a Newton method, that is, to find the initial condition $\dot
 \theta _1(0)$ such that the conditions for either Type $0$ or Type
 $\frac{\pi}{2}$ are satisfied.
\end{remark}

Next we introduce the following definition, which deals with the spins
of both objects.

\begin{definition}
 We say that the ellipsoids $\mathcal E_1$, $\mathcal E_2$ are in a
 standard spin-spin resonance\footnote{We remark that this is a
 practical definition because it is useful for the physical
 interpretation. However, we notice that there is a more general
 definition given by the resonant combination
 $n_0t-n_1\theta_1-n_2\theta_2$, with $n_0,n_1,n_2$ integers.} of
 order $(m_1:n_1,m_2:n_2)$ with $m_1,m_2\in\Z,n_1,n_2\in\Z \setminus
 \{0\}$, if the ellipsoid $\mathcal E_j$ is in a $m_j:n_j$ spin-orbit
 resonance.  In such case, the resonant angles
 $\psi^{m_1:n_1}_{m_2:n_2}(t) =
 \psi_1^{m_1:n_1}(t)\pm\psi_2^{m_2:n_2}(t)$ are $2\pi n$-periodic
 functions, where $n$ is the least common multiple of $n_1$ and $n_2$.
\end{definition}

An analogous definition holds for resonances of the balanced type.

\begin{definition}
 We say that the ellipsoids $\mathcal E_1$, $\mathcal E_2$ are in a
 balanced spin-spin resonance of order $(m_1:2,m_2:2)$ with
 $m_1,m_2\in\Z$, if the ellipsoid $\mathcal E_j$ is in a $m_j:2$
 spin-orbit resonance for $j = 1,2$.
\end{definition}

Note that a spin-spin resonance of order $(m_1:n_1,m_2:n_2)$ is also
of order $(\kappa_1m_1:n,\kappa_2m_2:n)$, where
$n=\kappa_1n_1=\kappa_2n_2$ is the least common multiple of $n _1$ and
$n_2$. Again, the converse is not true in general. For example, the
resonance $(1:1,3:2)$ is of order $(2:2,3:2)$, but not the
opposite. However, a balanced resonance $(2:2,3:2)$ is a spin-spin
resonance of order $(1:1,3:2)$.

The following proposition generalizes \Cref{prop:restypes} to the
spin-spin problem.

\begin{proposition}\label{prop:resSStypes}
  The following statements hold for the spin-spin problem
  \cref{spin_spin_j}:
  \begin{enumerate}
    \renewcommand*{\theenumi}{\roman{enumi}}
    \renewcommand*{\labelenumi}{(\theenumi)}
      \item \label{prop:resSStypes-1} Any
        $R_{\alpha,\beta_1,\beta_2}$-symmetric orbit, with $R
        _{\alpha,\beta_1,\beta_2}$ defined in \eqref{RRR}, associated
        to a $(m_1:n,m_2:n)$ spin-orbit resonance is equivalent to a
        solution that satisfies the following Dirichlet conditions:
        \[
            \theta_j(\alpha)=\beta_j,\quad
            \theta_j(\alpha+n\pi)=\beta_j + m_j\pi\ ,
        \]
        with $\alpha\in\{0,\pi\}$ and $\beta_j\in \{0,\frac{\pi}{2}\}$
        (eight combinations). Moreover, such solution satisfies the
        following symmetry property
        $\theta_j(t)=2\beta_j-\theta_j(2\alpha-t)$.

      \item \label{prop:resSStypes-2} There are four independent types
        of $R_{\alpha,\beta_1,\beta_2}$-symmetric orbits representing
        a balanced $(m_1:2,m_2:2)$ spin-spin resonance and are given
        by:
        \begin{equation}\label{Dirichlet_balSS}
            \text{Type } (\beta_1,\beta_2): \quad
            \theta_j(0)=\beta_j,\quad \theta_j(\pi) =
            \beta_j+\frac{m_j\pi}{2}\ ,
        \end{equation}
        with $\beta_j\in \{0,\frac{\pi}{2}\}$. Moreover, the
        corresponding symmetry relation is
        $\theta_j(t)=2\beta_j-\theta_j(-t)$.
    \end{enumerate}
\end{proposition}

\begin{proof}
This proof is based on the fact that the same arguments used to prove
\Cref{prop:restypes} can be generalized in a straightforward way for
the spin-spin problem \cref{spin_spin_j}.

Now we apply \Cref{lemmafix} to \cref{vecF} with $\vec F(\vec x)$
given by \cref{F_spin-spin}. The fixed point set of each reversing
symmetry $R_{\alpha,\beta_1,\beta_2}$ is
\[
 \operatorname{Fix} (R_{\alpha,\beta_1,\beta_2})= \{
 (t,\theta_1,\theta_2,p_1,p_2) \colon \
 t=\alpha,\theta_1=\beta_1,\theta_2=\beta_2\}\ .
\]
Then, due to the periodicity of $W$ in \equ{HK}, the periodic orbits
can be found at $\theta_j(\alpha)=\beta_j$, where we can take any
combination between $\alpha\in\{0,\pi\}$ and $\beta_j\in
\{0,\frac{\pi}{2}\}$. A similar method was used in \cite{gre1979} for
the standard map and was generalized for the spin-orbit problem in
\cite{celchi2000} and for a standard map of two degrees of freedom in
\cite{CFL}.

The rest of the proof of \cref{prop:resSStypes-1} follows analogously
to the proof of item (1) of \Cref{prop:restypes} using the reversing
symmetries.

The proof of \cref{prop:resSStypes-2} needs more detail. In the same
way as \cref{alphabeta_bal}, we obtain easily that the balanced
resonances $(m_1:2,m_2:2)$ are given by
\begin{equation}\label{alphabeta_balj}
    \theta_j(\alpha)=\beta_j,\quad \theta_j(\alpha\pm\pi)=\beta_j
    \pm \frac{m_j\pi}{2}\ .
\end{equation}
We see that \cref{Dirichlet_balSS} corresponds to
\cref{alphabeta_balj} with $\alpha=0$ and the positive sign. From the
case $\alpha=\pi$ and the negative sign we have that
\begin{equation}\label{Dirichlet_balSS2}
  \theta_j(0)=\beta_j- \frac{m_j\pi}{2},\quad \theta_j(\pi)=\beta_j\ .
\end{equation}
Note that if, for example, we take $j=1$ and $m_1=2k_1$, with $k_1\in
\Z$, then, the conditions \cref{Dirichlet_balSS} and
\cref{Dirichlet_balSS2} are equivalent. On the other hand, if
$m_1=2k_1+1$, then, the conditions \cref{Dirichlet_balSS} with
$\beta_1=0$ and $\frac{\pi}{2}$ are equivalent respectively to
\cref{Dirichlet_balSS2} with $\beta_1=\frac{\pi}{2}$ and $0$. The same
is true for $j=2$. Then, with \cref{Dirichlet_balSS} all the
possibilities are covered.
\end{proof}

\begin{remark}
 Results in \Cref{prop:resSStypes} allow to apply a Newton method to
 compute resonances for the spin-spin problem just considering as
 unknowns $\dot \theta _j(0)$ and correct them by imposing the
 conditions in \cref{Dirichlet_balSS} or \cref{Dirichlet_balSS2}.
\end{remark}

For circular orbits ($e=0$), each of the spin-orbit models
\cref{spin_orbit_j} is a classical pendulum whose only stable
equilibrium point corresponds to a $1:1$ resonance that is given by
$\theta_j(t)=t$. Similarly, the spin-spin model \cref{spin_spin_j}
consists of two coupled penduli whose only stable solution is
$\theta_1(t)=\theta_2(t)=t$ that is a $(1:1,1:1)$ resonance. However,
for $e \ne 0$, $f(t;e)$ does not coincide with $t$ and more stable
spin-orbit resonances may appear.

In order to study the spin-orbit and spin-spin resonances, it is
useful to compute the expansion of $V_2$ and $V_4$ up to some
power of the eccentricity. This expansion is obtained solving
Kepler's equation \equ{t} up to a finite order in the
eccentricity, then inserting the solution $u=u(t)$ in \equ{r},
\equ{ff}, expand them in series of the eccentricity and finally
expanding the trigonometric terms appearing in \equ{V024}.

This procedure leads to the expansions of $V_2$ and $V_4$ that, for
simplicity, we give up to the order 2 in the eccentricity in the
Appendix~\ref{app:expansion}. In those expressions, the trigonometric
terms with arguments $\psi_j^{m_j:n_j}(t)=m_j t-n_j \theta_j$ are
associated to $m_j:n_j$ spin-orbit resonances for $\mathcal E_j$,
whereas the terms with argument $\psi^{m_1:n_1}_{m_2:n_2}(t) = (m_1\pm
m_2) t - n_1 \theta_1(t) \mp n_2 \theta_2(t)$ correspond to spin-spin
resonances by combining spin-orbit resonances of orders $m_1:n_1$ and
$m_2:n_2$. For each order of the expansion in the eccentricity $e$,
there are some resonances appearing. They are shown in
\Cref{tab:V_res}, where we can recognize a hierarchy: the most
important spin-orbit resonance is $1:1$, then we have $3:2$, $1:2$ and
so on, because they appear for low orders of the
eccentricity. Resonances of further orders are relevant only for large
eccentricities.

Note that the most relevant spin-orbit resonances are of order $m:2$,
$m\in \Z$. Additionally, spin-spin resonances appearing at order
$e^{\alpha}$ in $V_4$ are obtained by combining spin-orbit resonances
appearing at order $e^{\alpha_1}$ and $e^{\alpha_2}$ in $V_2$ such
that $\alpha_1+\alpha_2=\alpha$.

Finally, note that for $V_2$, the coefficients associated to
spin-orbit resonances are of order one in $d_1,d_2$, whereas for
$V_4$, the spin-orbit coefficients are of order two in
$d_1,d_2,q_1,q_2$, and the spin-spin coupling coefficients are of
order two in $d_1,d_2$.

\begin{table}[ht]
    \begin{center}
        \begin{tabular}{|c||c|c|c|}
            \hline
            &\multicolumn{2}{c|}{spin-orbit resonances}&  spin-spin resonances \\ \hline
            order &   $V_2$ &  $V_4$ & $V_4$         \\ \hline  \hline
            $e^0$ &  $1:1$    & $1:1$  &  combine $1:1$ with $1:1$ \\ \hline
            $e^1$  & $3:2$, $1:2$ &  $3:2$, $1:2$, $3:4$, $5:4$    & combine $1:1$ with $3:2$ and $1:2$ \\   \hline
            $e^2$ &  $1:1$,  $2:1$, $0:1$  & \makecell[c]{$1:1$, $2:1$, $0:1$, \\  $1:2$, $3:2$} & \makecell[c]{combine $1:1$ with $1:1$, $2:1$ and $0:1$ \\
            all combinations between $3:2$ and $1:2$}   \\  \hline
        \end{tabular}
        \vskip .1in
        \caption{Resonances appearing in the expansion of the
          potential $V_2+V_4$ for each order of the eccentricity.}
        \label{tab:V_res}
    \end{center}
\end{table}

\subsection{KAM tori in the spin-spin problem} \label{sec:KAM}
Now we deal with quasi-periodic solutions of the Keplerian version of
the spin-spin model. We denote by
\begin{equation}\label{om}
    \underline{\omega}=(1,\omega_1,\omega_2)
\end{equation}
the frequency vector with
\[
  \omega_1={p_1\over C_1}\ ,\qquad \omega_2={p_2\over C_2}\ .
\]
The Hessian matrix associated to \equ{HK} has determinant different
from zero, whenever $p_1\not=0$, $p_2\not=0$. This implies that
\equ{HK} satisfies the Kolmogorov non-degeneracy condition, which is a
requirement for the applicability of KAM theorem
\cite{kolmogorov1954conservation,Arnold63a,Moser62}.  The other
essential requirement in KAM theory is the assumption that the
frequency satisfies a Diophantine inequality, namely there exist $C>0$
and $\xi\geq 2$, such that
\begin{equation}\label{DC}
 |\underline{k} \cdot \underline{\omega}|^{-1}\leq C|\underline{k}|^\xi
\end{equation}
for $\underline{k}\in\Z^3\setminus\{\underline{0}\}$.

We remark that a possible choice of $\underline{\omega}$ satisfying
\equ{DC} can be obtained as follows.  Let $\alpha$ be an algebraic
number of degree 3, namely the solution of a polynomial equation of
degree 3 with integer coefficients, not being the root of polynomial
equations of lower degree. Let us consider the vector
$\underline{\omega}=(1,\omega_1,\omega_2)$ obtained as
\begin{equation}\label{ba}
    \begin{pmatrix}
        1 \\
        \omega_1 \\
        \omega_2
    \end{pmatrix}=
    \begin{pmatrix}
        1 & 0 & 0 \\
        b_1 & a_{11} & a_{12} \\
        b_2 & a_{21} & a_{22}
    \end{pmatrix}%
    \begin{pmatrix}
        1 \\
        \alpha \\
        \alpha^2
    \end{pmatrix}\ ,
\end{equation}
where $(b_1,b_2)$ and the matrix $A\equiv (a_{mn})$ have rational
coefficients and $\det A\not=0$. By number theory results (see, e.g.,
\cite{CFL}), a vector $\underline{\omega}$ as in \equ{ba} satisfies
\equ{DC} with $\xi=2$. Under smallness conditions of the parameters,
say $\lambda_j$ in \equ{lambda_j}, KAM theory ensures the existence of
a quasi-periodic torus with Diophantine frequency. We remark that the
theory presented in \cite{CCGL20b} for the spin-orbit problem (see
also \cite{CCGL20a,CCGL20c}) can be extended to provide explicit
estimates for \equ{HK} and an explicit algorithm to construct
quasi-periodic solutions.


\section{Qualitative description of the spin models}\label{sec:graph}

In this section, we give a qualitative description of the phase space
associated to the spin-orbit problem (Section~\ref{sec:SO}), the
spin-spin problem with spherical companion
(Section~\ref{sec:spherical}) and with non-spherical companion
(Section~\ref{sec:nonspherical}).

\subsection{Spin-orbit problem ($V_{per} = V_2$)}\label{sec:SO}

Since the system \cref{spin_orbit_j} corresponds to two uncoupled
spin-orbit problems, then the Poincar\'e map of the whole system can
be understood as the direct product of the Poincar\'e maps for each of
the bodies. Consider for example the dynamics of $\mathcal
E_1$. \Cref{fig:Poincare_spin-orbit} is a typical Poincar\'e map of
the spin-orbit problem obtained using a Taylor integrator
\cite{JorbaZ2005} and using a similar approach as the one explained in
\Cref{app:diff}, in this case with parameters $(e,\lambda_1) = (0.06,
0.05) $; it represents solutions at $t=2\pi k$, $k\in \Z$, in the
plane $(\theta_1, \dot \theta_1)$ restricted to $\theta_1\in [-\pi,
  \pi]$. The main stable resonances are tagged with their
corresponding order $m:n$. The Poincar\'e map for sufficiently small
parameters $(e,\lambda_1) $ has the following features:

\begin{enumerate}
\renewcommand{\labelenumi}{\theenumi)}
    \item The main stable spin-orbit resonances are represented by
      fixed points in the plane $(\theta_1, \dot \theta_1)$ surrounded
      by islands of invariant librational tori. High order resonances
      appear above low order resonances.

    \item It looks that the spin-orbit resonances of order $m:2\ge
      1:1$ are balanced: solutions of type 0 \cref{Dirichlet_I} are
      stable and those of type $\frac{\pi}{2}$ \cref{Dirichlet_II} are
      unstable. On the contrary, for the $1:2$ resonance, type 0 is
      unstable and type $\frac{\pi}{2}$ is stable. Concerning other
      resonances \cref{Dirichlet_odd}, for instance in the $3:4$
      resonance, types with $\alpha=0$ are unstable and types with
      $\alpha=\pi$ are stable. Exactly the opposite occurs for the
      case $5:4$.

    \item It is possible to have stable secondary resonances, namely
      small resonances surrounding other resonances. This is clear for
      the $1:1$ resonance. Beyond the librational islands associated
      to the resonance $1:1$ there is a chaotic region that includes
      the unstable resonances and that is larger for large parameters
      $(e,\lambda_1) $. The chaotic region can appear for other
      resonances and is limited by rotational tori that also
      distinguish the domains of resonances of different orders.
\end{enumerate}

\begin{figure}[ht]
    \centering
    \scalebox{0.6}{\includegraphics{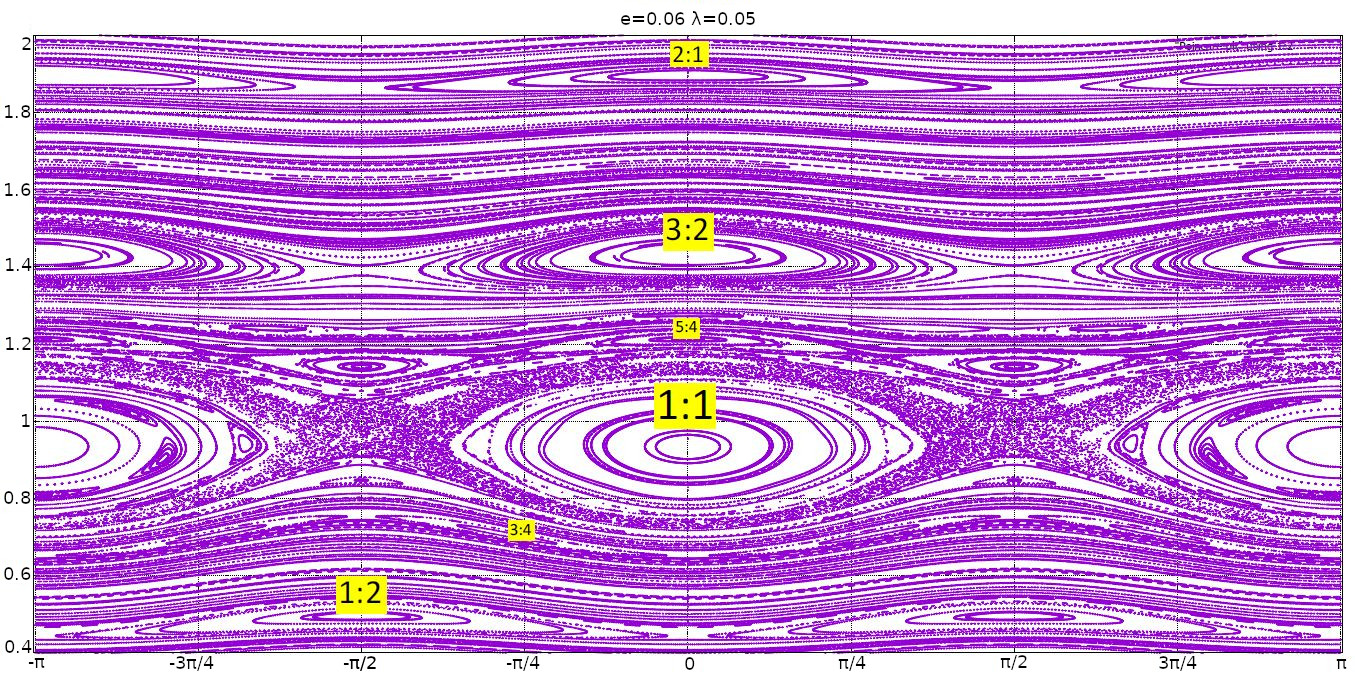}}
    \caption{Poincar\'e map for the spin-orbit problem.}
    \label{fig:Poincare_spin-orbit}
\end{figure}

\subsection{Spin-spin problem ($V_{per} = V_2+V_4$) with spherical companion}
\label{sec:spherical}

Let us now consider the case when $\mathcal E_2$ is a sphere.  Then,
$\theta_2(t)=\theta_2(0)+\dot \theta_2(0)t$ and the dynamics of
$\theta_1$ is given by \cref{spin_spin_1}, that depends on the
parameters $(e,\lambda_1,\hat q_1,\sigma_1)$. Here the parameters
$\sigma_1$ and $\hat q_1$ perturb the previous framework of the
spin-orbit problem, see \Cref{sec:SO}. We can see a comparison between
both problems in \Cref{fig:Poincare_spin-orbit_spin-orbit-high}: we
see that the only qualitative difference between both cases is that
the Poincar\'e map associated to equation \cref{spin_spin_1} is
slightly more chaotic. This minor difference is due to the fact that
we take $\sigma_1=\hat q_1=0.01$, that are small parameters.  As we
will see in \Cref{sec:comparison}, the spin-spin model is a good
approximation for the dynamics of two ellipsoids for $\sigma_j$ and
$\hat q_2$ up to the order of magnitude of about $10^{-2}$, because
larger values could lead to a collision, see \Cref{sec:quantitative}.

\begin{figure}[ht]
    \centering
    \scalebox{0.3}{\includegraphics{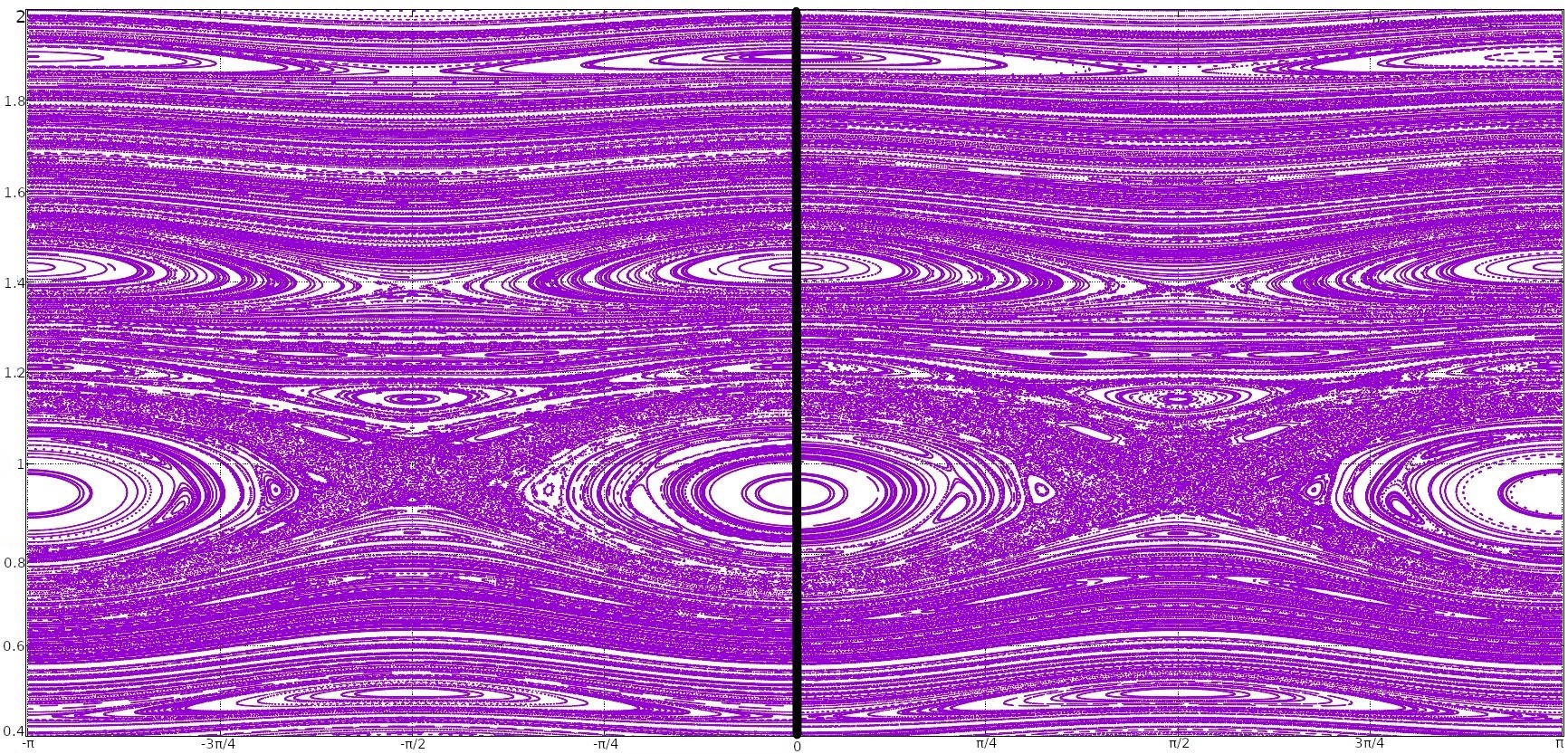}}
    \caption{Poincar\'e maps. Left: usual spin-orbit problem with
      $(e,\lambda_1)=(0.06,0.05)$. Right: spin-orbit problem up to
      order $1/r^5$ with $(e,\lambda_1,\sigma_1,\hat q_1)=(0.06,0.05,
      0.01,0.01)$.}
    \label{fig:Poincare_spin-orbit_spin-orbit-high}
\end{figure}

\subsection{Spin-spin problem ($V_{per} = V_2+V_4$) with non-spherical companion}
\label{sec:nonspherical}

Now we deal with the general system \cref{spin_spin_j}. Let
$\Psi(t)=(\theta_1(t),\theta_2(t),\dot \theta_1(t),\dot \theta_2(t))$
be a solution of \cref{spin_spin_j} and its respective projections
$\Psi_1(t)=(\theta_1(t),\dot \theta_1(t))$ and
$\Psi_2(t)=(\theta_2(t),\dot \theta_2(t))$. From now on we restrict
$\theta_j(t)$ to $[-\pi,\pi]$. Let the Poincar\'e map associated to
such solution be defined by $\mathcal P(\Psi(0))= \Psi(2\pi)$, and its
projections by $\mathcal P_j(\Psi_j(0))= \Psi_j(2\pi)$. It is not
possible to represent the iteration of the map $\mathcal P$ in a
single plot, because it is 4-dimensional, so we will represent the
projections $\mathcal P_j$. That is to say, for a solution $\Psi(t)$,
we are interested in the behavior of the two families of points
$\Psi_1(2\pi k)$ and $\Psi_2(2\pi k)$, $k\in \Z$, in a single
2-dimensional strip $\Pi = \{(x,y)\in \R^2\colon |x|\le \pi\}$. We
recognize the following features:

\begin{enumerate}
\renewcommand{\labelenumi}{\theenumi)}
    \item A solution in spin-spin resonance corresponds to a family of
      isolated recurrent points of $\mathcal P$ (namely, the
      successive points on the Poincar\'e map), whose projections are
      represented in $\Pi$ as a pair of families of recurrent points.

    \item If the spin-spin resonance is stable, then nearby solutions
      would librate around such points. While in the uncoupled system,
      librating solutions belong to 2-dimensional tori, here tori can
      be higher dimensional. As a result, the projected points
      represented in $\Pi$ are distributed in two clouds of points
      surrounding each recurrent point. A cloud of this kind covers an
      annulus-like region of a certain thickness that is usually
      thicker for stronger couplings. A similar behavior occurs for
      rotational solutions, whose corresponding clouds are distributed
      in strips of certain thickness, see
      \Cref{fig:spin-spin_libr_rot}.

    \begin{figure}[ht]
        \centering
        \includegraphics[scale=.63]{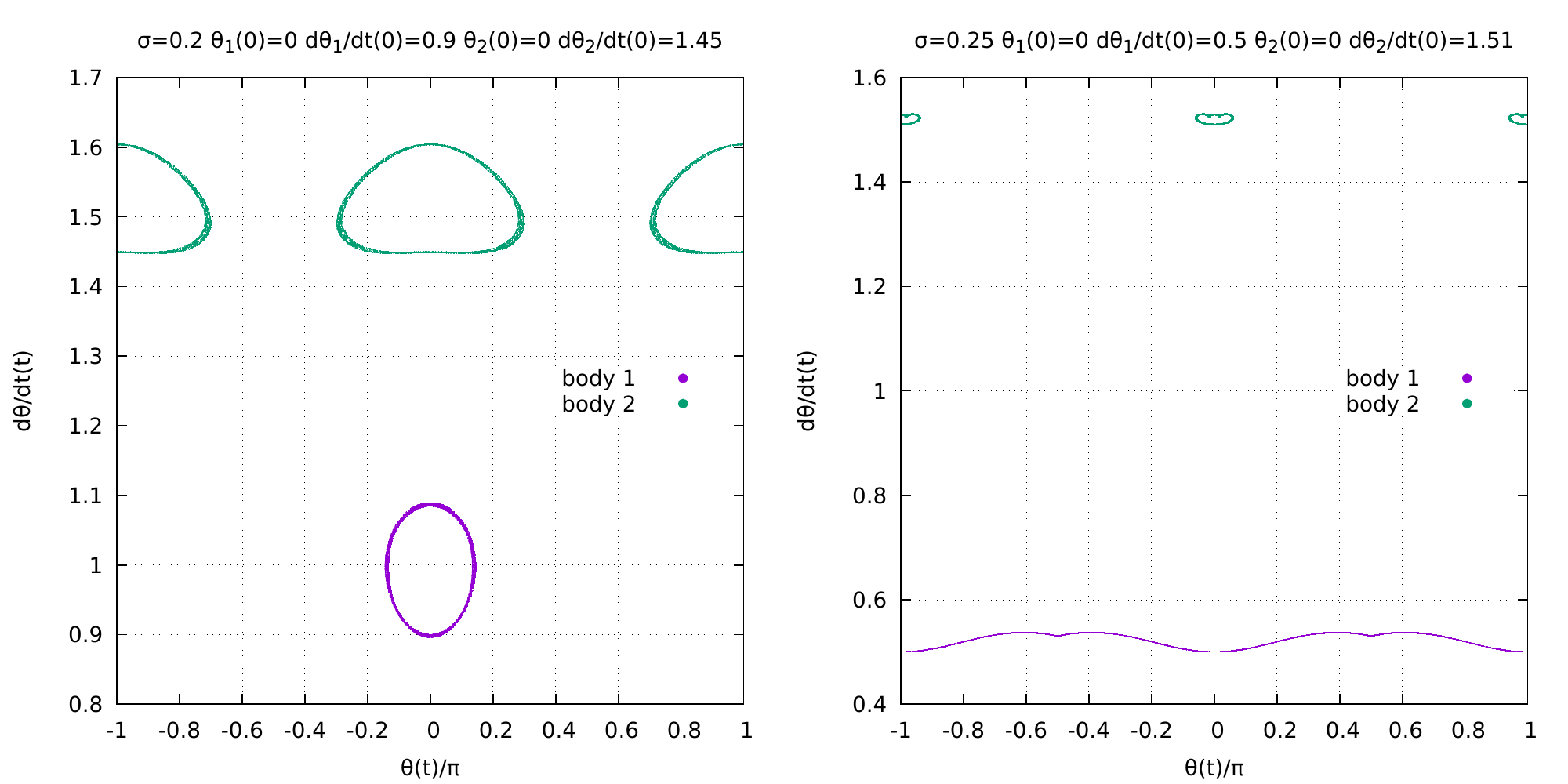}
        \caption{Left: Projections $\Psi_1(2\pi k)$ and $\Psi_2(2\pi
          k)$, $k\in \N$ of a solution $\Psi(2\pi k)$ librating around
          a stable $(1:1,3:2)$ spin-spin resonance. Right: Solution
          for which one of the bodies librates around a stable
          spin-orbit resonance and the other one has a rotational
          behavior.  The common parameter values are $e=0.06$,
          $\lambda _1 = \lambda _2 = 0.05$, and $\hat q _1 = \hat q _2
          = 0.001$.}
        \label{fig:spin-spin_libr_rot}
    \end{figure}

    \item We expect that, for small enough coupling parameters
      $\sigma_j$, the spin-spin resonances are located close to the
      corresponding spin-orbit resonances for each ellipsoid, and
      would keep the same stability as for the uncoupled
      problem. However, we have found that, for larger $\sigma_j$, the
      stability may change with respect to the uncoupled system, see
      \Cref{fig:sync_3_2_to_1_2}.


    \item The coupled system is 5-dimensional, then, invariant tori,
      if there exist, would not confine solutions in determined
      regions (as in the uncoupled system), but Arnold diffusion is
      expected to take place.
\end{enumerate}

\begin{figure}[ht]
    \centering
    \includegraphics[scale=.63]{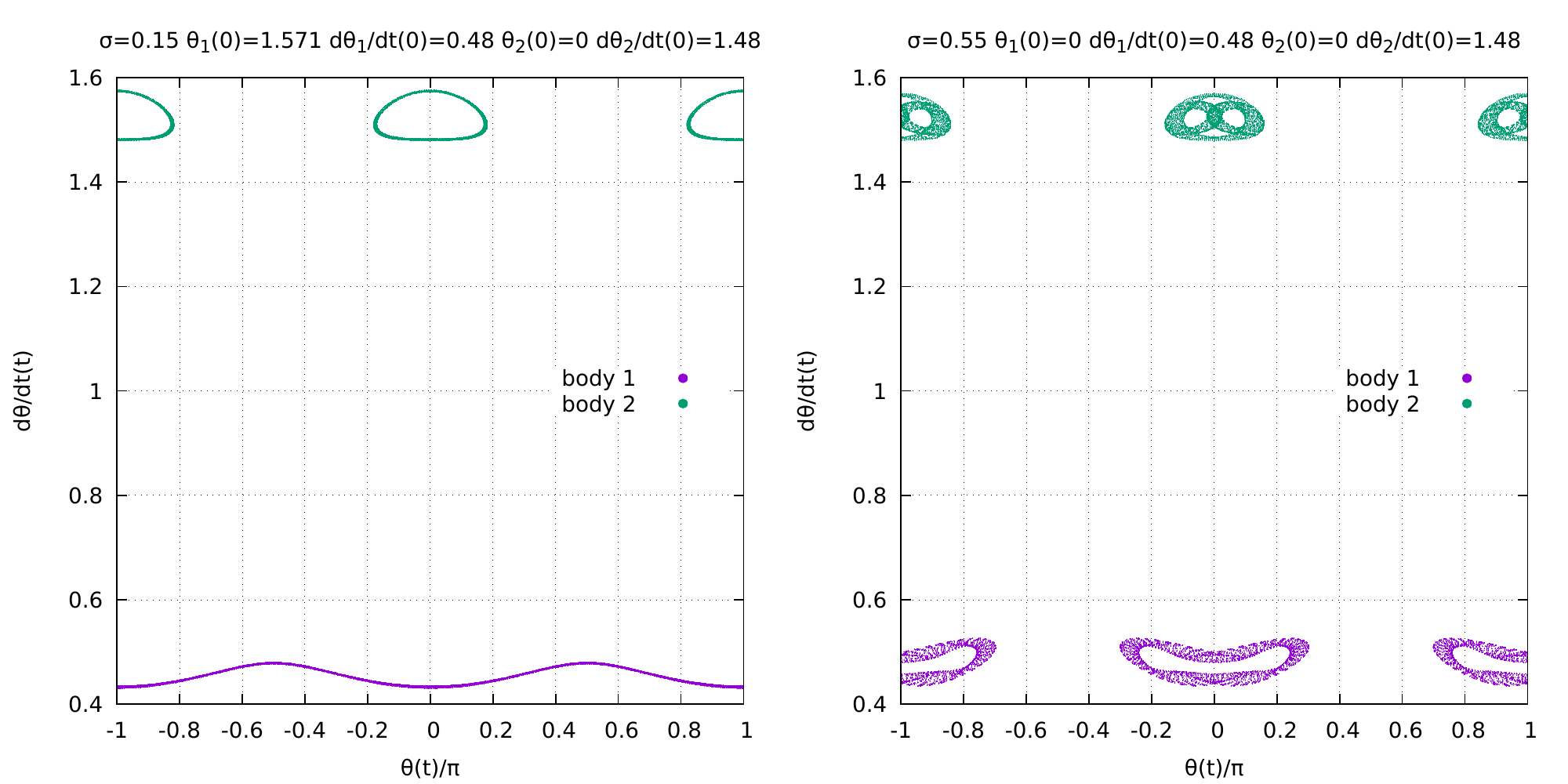}
    \caption{Left: Projections $\Psi_1(2\pi k)$ and $\Psi_2(2\pi k)$,
      $k\in \Z$ of a solution $\Psi(t)$ close to an unstable
      $(1:2,3:2)$ spin-spin resonance. Right: Representation of a
      solution with identical $\Psi(0)$ for larger coupling, now the
      nearby spin-spin resonance has become stable.  The common
      parameter values are $e=0.06$, $\lambda _1 = \lambda _2 = 0.05$,
      and $\hat q _1 = \hat q _2 = 0.001$.}
    \label{fig:sync_3_2_to_1_2}
\end{figure}

\begin{figure}[ht]
\centering
\includegraphics[scale=.63]{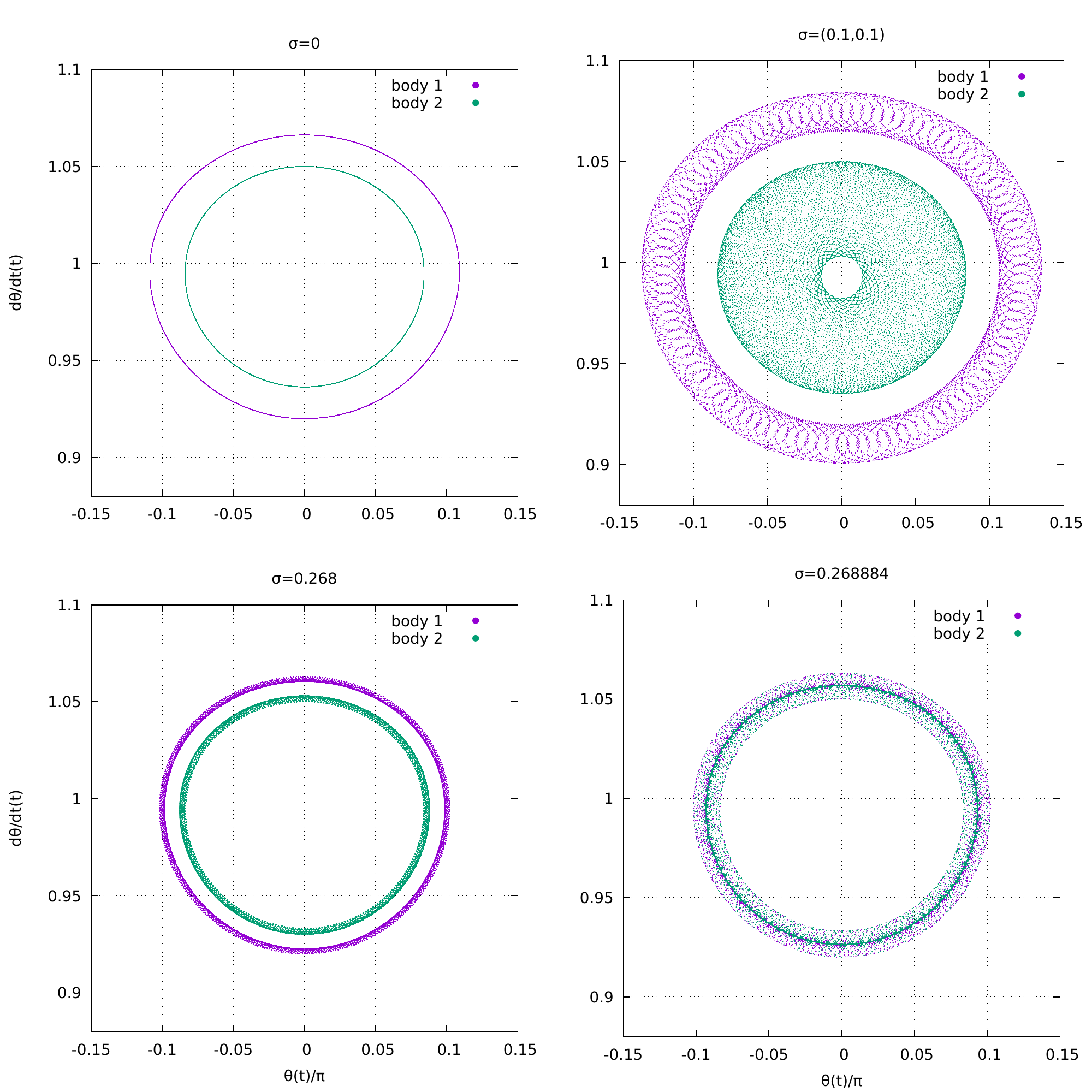}
\caption{Projections $\Psi_1(2\pi k)$ and $\Psi_2(2\pi k)$, $k\in \Z$
  of a solution $\Psi(t)$ close to a stable $(1:1,1:1)$ spin-spin
  resonance for different values of $\sigma _1$ and $\sigma
  _2$. Keeping the same parameters $e=0.06$, $\lambda _1 = \lambda _2
  = 0.05$, $\hat q _1 = \hat q _2 = 0.001$, $\theta _1(0) = \theta
  _2(0) = 0$, $\dot \theta _1(0) = 0.92$ and $\dot \theta _2(0) =
  1.05$.  The external ring (body 1) keeps similar thickness, the
  internal ring changes from a thin one to another that occupies
  values close to (0,0) to thin one to finaly collapse with the
  exterior one.}
\label{fig:sync_1_1_to_1_1}
\end{figure}

\begin{figure}[ht]
\centering
\includegraphics[scale=.63]{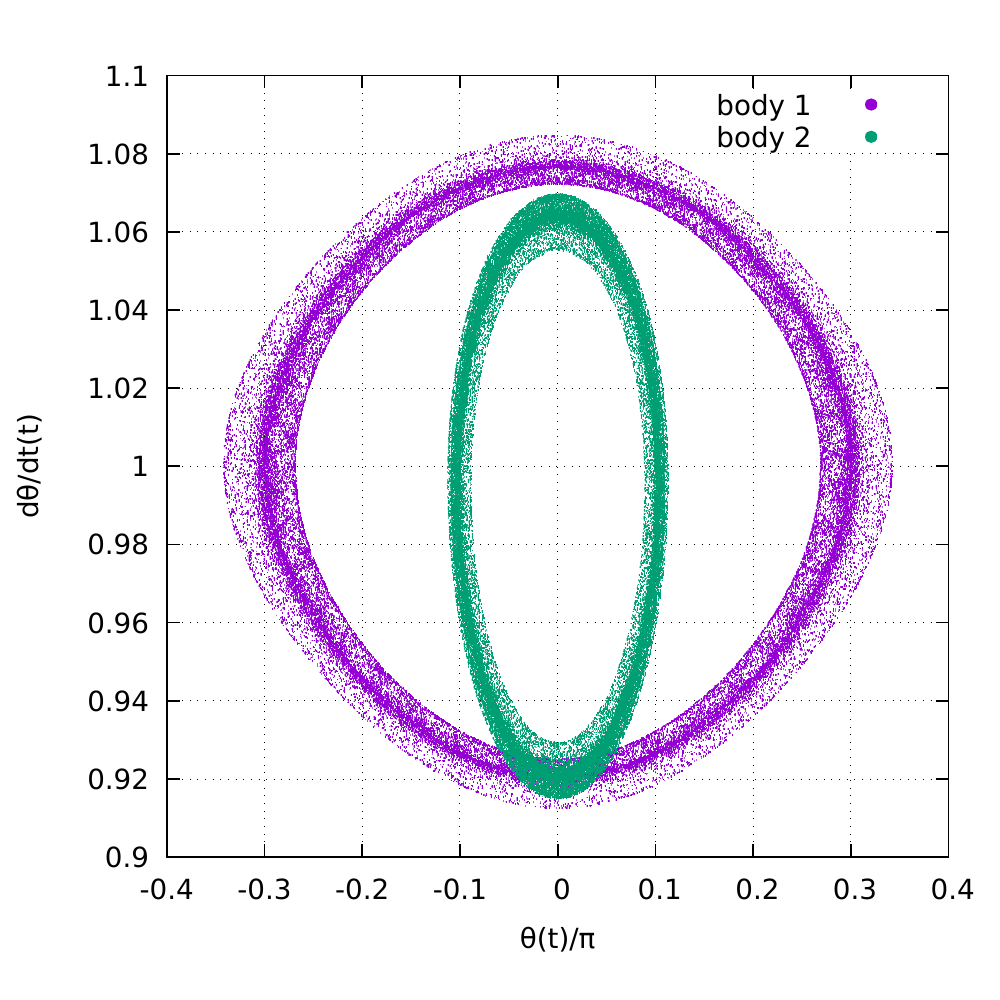}
\caption{Partial overlapping of the domains of $\Psi_1(2\pi k)$ and
  $\Psi_2(2\pi k)$, $k\in \Z$ of a solution $\Psi(t)$ close to a
  stable $(1:1,1:1)$ spin-spin resonance in the case of different
  bodies: no measure synchronization.  The parameter values are
  $e=0.06$, $\lambda _1 = 0.009$, $\lambda _ 2 = 0.05$, $\hat q _1 =
  \hat q _2 = 0.001$, $\sigma _1 = \sigma _2 = 0.3$, $\theta _1(0) =
  \theta _2(0) = 0$, $\dot \theta _1 (0) = 0.92$, and $\dot \theta
  _2(0) = 1.064$.}
\label{fig:nosync}
\end{figure}

A particular behavior occurs only when both bodies are identical
($\mathcal C_1=\mathcal C_2 = 0.5$, $\lambda=\lambda_1=\lambda_2$,
$\sigma=\sigma_1=\sigma_2$ and $\hat q=\hat q_1=\hat q_2$), the
so-called \textit{measure synchronization}. This phenomenon was
observed numerically in \cite{ham1999} for an autonomous Hamiltonian
system of two degrees of freedom (a pair of identical coupled
oscillators): the system librates around a stable periodic solution in
a very particular way described as follows for our system (see the
phenomenon illustrated in \Cref{fig:sync_1_1_to_1_1}). Take a solution
$\Psi(t)$ librating around a stable spin-spin resonance of order
$(m:n,m:n)$. Consider the two families of points $\Psi_1 (2\pi k)$ and
$\Psi_2(2\pi k)$ and their corresponding annulus-like region in the
plane $\Pi$. There are two possibilities: either both clouds of points
are distributed in separated annulus-like regions or both regions
coincide. That is to say, either the overlap is empty or there is a
complete overlap. Moreover, if we start with a solution with separated
regions, we can obtain the complete overlap by increasing the coupling
parameter $\sigma$ (keeping the same $\Psi(0)$). The phenomenon takes
place suddenly for a specific $\sigma=\sigma_*$ when the outer
boundary of the inner ring touches the inner boundary of the outer
ring. At that moment, there is a concentration of density of points in
the contact region.

The phenomenon of measure synchronization disappears when the bodies
are not equal. See in \Cref{fig:nosync} how the domains of both
ellipsoids can overlap without merging into a single ring.



\section{Linear stability of resonances}\label{sec:linear}

In this section we analyze the stability (in the linear approximation)
of solutions associated to the resonances in different models, namely
the spin-orbit problem (\Cref{sub:lin_spin_orbit}) and the spin-spin
problem with spherical (\Cref{spher}) and non-spherical
(\Cref{nonspher}) companion.

In all cases, we will only deal with balanced resonances
\cref{resSO_bal}, because they appear to be simpler and more relevant
(see \Cref{sec:graph}) than resonances of the general type
\cref{resSO}. Actually, we can establish regions in the space of
parameters where solutions associated to balanced resonances are
unique or have some low multiplicity. In the case of the spin-spin
problem, we restrict ourselves to regions of uniqueness. The study of
linear stability of such periodic solutions in the space of parameters
complements the understanding of the dynamics that we presented in
previous sections, especially \Cref{sec:graph}.

\subsection{Spin-orbit problem ($V_{per}=V_2$)}
\label{sub:lin_spin_orbit}

Consider the spin-orbit problem \cref{spin_orbit_j} with $j=1$, that
is, the motion of the ellipsoid $\mathcal E_1$. Let
$\theta_1=\theta^*(t)$ be a solution in a balanced $m:2$ spin-orbit
resonance, whose associated variational equation is
\begin{equation}\label{spin_orbit_lin}
    \ddot y + \lambda_1 \parentesis{\frac{a}{ r(t;e)}}^{3} \cos(2
    \theta^*(t)-2f(t;e))y=0,\qquad y\in \R .
\end{equation}
For $e\ne0$, \cref{spin_orbit_lin} is a linear equation with a $2\pi
$-periodic coefficient. Particularly, this is a Hill's equation
\cite{mag}.  Assume that $\Phi(t)$ is a matrix solution of
\cref{spin_orbit_lin} with $\Phi(t_0)=\mathds 1_2$, the identity
matrix 2$\times$2. The stability of \cref{spin_orbit_lin} is
determined by the structure of the matrix $M=\Phi(t_0+2\pi )$, called
monodromy matrix. If $|\Tr(M)|<2$, we have elliptic stability, whereas
if $|\Tr(M)|>2$ we have hyperbolic instability. In the parabolic case,
when $|\Tr(M)|=2$, the system is stable if the Jordan canonical form
of $M$ is $\mathds 1_2$ or $-\mathds 1_2$, otherwise the system is
unstable. Actually, if the system is parabolic unstable, the
instability of the linear system is linear in time, but hyperbolic
instability is associated to an exponential divergence in time. In our
case we want to distinguish regions of stability and instability in
the $(e,\lambda_1)$-plane for a given solution, which is continuous in
$(e,\lambda_1)$. From properties of Hill's equations, regions of
elliptic stability are separated from regions of hyperbolic
instability by parabolic curves ($|\Tr(M)|=2$). These curves are made
of unstable points, except if there are intersections of parabolic
curves, because points of intersection become stable. This phenomenon
is called \textit{coexistence}, \cite{mag}.

For a given point $(e,\lambda_1)$ and a given balanced $m:2$
resonance, we want to know how many solutions there are of each type
\cref{Dirichlet_I} or \cref{Dirichlet_II}, and their linear
stability. First, recall \Cref{rem:negative}: solutions of type
$\frac{\pi}{2}$ satisfy conditions of type 0 for the equation
\cref{spin_orbit_j} with $j=1$, taking $-\lambda_1$ instead of
$\lambda_1$. Consequently, for each $(e,\lambda_1)$, with $e\in[0,1)$
  and $\lambda_1\in (-3,3)$, we can obtain all the solutions
  corresponding to a balanced resonance by applying the shooting
  method: take a solution $\theta_1(t)$ with initial
  conditions\footnote{We choose to take initial conditions at $t=\pi$
  and not $t=0$ because the values of $\dot \theta_1(0)$ producing
  spin-orbit resonances for large $e$ and $\lambda_1$ are too large to
  be represented in a 3-dimensional plot as in
  \Cref{fig:lin_stab_spin_orbit}.}  $\theta_1(\pi)=m\pi/2$, $\dot
  \theta_1(\pi)=\gamma\in \mathbb R$ and let $\gamma$ vary until the
  boundary condition $\theta_1(0)=0$ is reached.  Finally, we obtain
  the linear stability of the solution by computing $\Phi(t)$ such
  that $\Phi(\pi)=\mathds 1_2$. Actually, for this procedure, we can
  take any of the boundary conditions in
  \cref{Dirichlet_I,Dirichlet_II},
  we just need one type to generate all solutions.

The results of this method for the main balanced spin-orbit resonances
are shown in \Cref{fig:lin_stab_spin_orbit}. For these computations we
used a Runge-Kutta Verner 8(9) integrator \cite{Verner78}, instead of
a Taylor integrator \cite{JorbaZ2005}; the reason is that, for some
parameter values, the solutions are constant or polynomials and the
Taylor method suffers in choosing a good step size. Thus,
\Cref{fig:lin_stab_spin_orbit} required around 3.5 days with 34 CPUs
to be generated with a discretization mesh size of $2000\times
2000\times 2000$ for $(e,\lambda_1,\gamma)$.

We can recognize the following characteristics\footnote{The linear
stability of the multiple solutions associated to the resonances $1:1$
and $3:2$ was already studied in \cite{zla,bel1966,bel1975}, but we
include them here in order to have a more complete view.}:

\begin{figure}[ht]
    \centering
    \includegraphics[scale=.37]{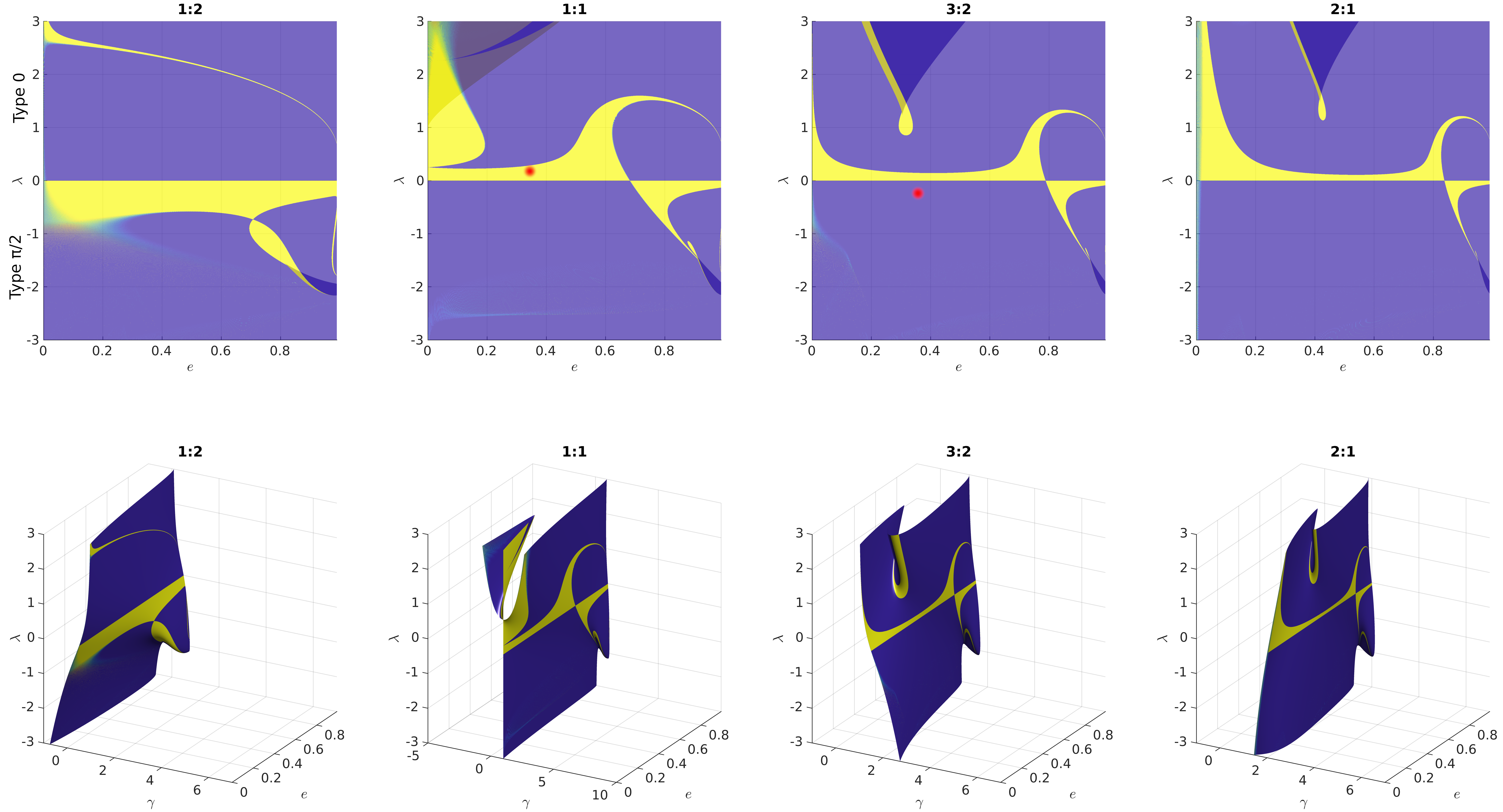}
    \caption{Diagrams of linear stability for the balanced resonances
      of order $1:2$ (left), $1:1$ (middle left), $3:2$ (middle right)
      and $2:1$ (right). Blue-instability and yellow-stability. The
      upper diagrams are projections of the lower diagrams in the
      $(e,\lambda_1)$-plane with some transparency in order to
      identify regions with multiplicity of solutions.}
    \label{fig:lin_stab_spin_orbit}
\end{figure}

\begin{enumerate}
\renewcommand{\labelenumi}{\theenumi)}
    \item Each of the balanced resonances is represented in the
      $(e,\lambda_1,\gamma)$-space by a continuous surface. In the
      case $1:1$, the surface is made of two sheets connected only in
      one point $(e,\lambda_1)=(0,1)$.

    \item The region of uniqueness in the $(e,\lambda_1)$-plane is
      quite large. The multiplicity is generated by bifurcation of
      solutions: the surface folds generating multiple solutions (from
      one to three, as far as we see). Particularly, the bifurcation
      in the case $1:1$ occurs at $(e,\lambda_1)=(0,1)$ producing a
      secondary sheet behind the main one. In general, the
      multiplicity takes place for some regions with
      $|\lambda_1|>1$. The resonances of order $m:2$ with $m>2$, have
      two characteristic folds, one with a V-like shape in solutions
      of type 0 for $\lambda_1>1$, and another small one in solutions
      of type $\frac{\pi}{2}$ for $\lambda_1\sim -2$ and very large
      eccentricities.

    \item Instability is predominant in the diagrams, especially in
      solutions of type 0 for the resonance 1:2 and of type
      $\frac{\pi}{2}$ for the rest of the resonances. We see that the
      main regions of linear stability are continuation of stable
      solutions from $e=0$, much of which are close to small
      $|\lambda_1|$. Except for the resonance 1:2, the stability
      region for large eccentricities ($e>0.6$) of the other
      resonances has a similar shape, characterized by a bifurcation
      with an interchange of stability. It is interesting to note that
      the folds producing the multiplicity are associated to some
      stable regions with peculiar shapes. In the resonance $1:1$ we
      find two unstable regions bifurcating from the exact solution
      $\theta_1(t)=t$ for $e=0$: one at $\lambda_1=1/4=0.25$ (main
      sheet) and the other one at $\lambda_1=9/4=2.25$ (secondary
      sheet).
\end{enumerate}

Now let us consider both bodies. Since the system \cref{spin_orbit_j}
is uncoupled, then, the multiplicity and stability associated to a
spin-spin resonance is given by each of the separated problems. For
example, take the $(1:1,3:2)$ balanced spin-spin resonance of type
$(0,\frac{\pi}{2})$ for $(e,\lambda_1,\lambda_2)=(0.3,0.1,0.5)$, that
is, the red points shown in \Cref{fig:lin_stab_spin_orbit}. It has
associated a unique solution and it is unstable because it is so for
$\theta_2$.

\subsection{Spin-spin problem ($V_{per} = V_2+V_4$) with spherical
  companion}\label{spher}

In this case we know that $\mathcal E_2$ is in uniform rotation
$\theta_2(t)=\theta_2(0)+\dot \theta_2(0)t$, while the dynamics of
$\theta_1$ is given by \cref{spin_spin_1}, that depends on
$(e,\lambda_1,\hat q_1,\sigma_1)$. For this problem, we can proceed in
the same way as for the spin-orbit problem of
\Cref{sub:lin_spin_orbit}. On one hand, the variational equation
associated to a solution in a spin-orbit resonance is a Hill's
equation like \cref{spin_orbit_lin}, so the linear stability of the
solution is characterized by the corresponding monodromy matrix. On
the other hand, we can find all the solutions associated to a balanced
spin-orbit resonance using the shooting method for only one type of
boundary conditions by including negative values of $\lambda_1$, see
\Cref{rem:negative}.

Comparing \Cref{fig:lin_stab_spin_orbit_high,fig:lin_stab_spin_orbit}
we can see, for example, how the balanced $3:2$ resonance is perturbed
when we turn on the parameters $(\hat q_1,\sigma_1)$:

\begin{figure}[ht]
    \centering
    \includegraphics[scale=.37]{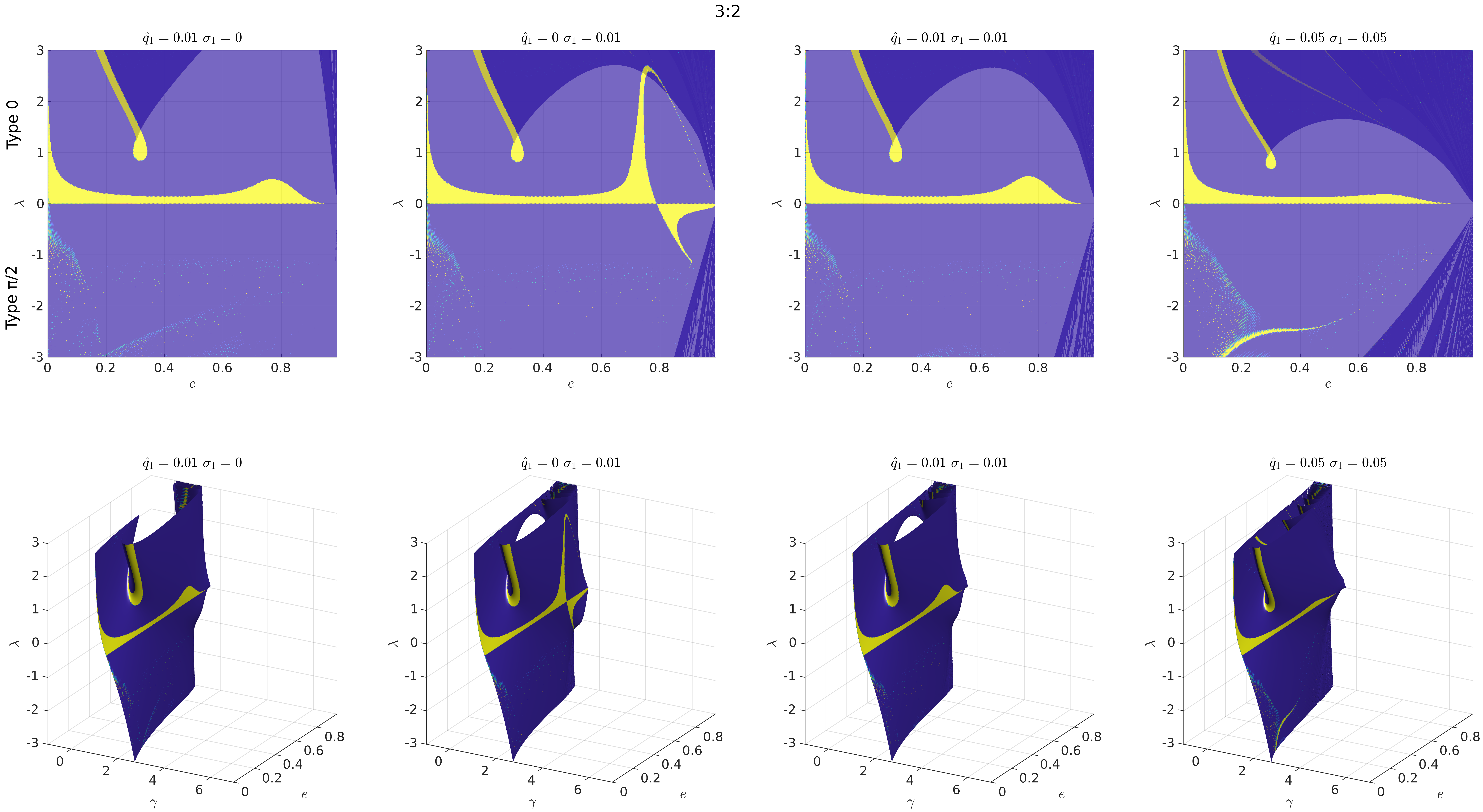}
    \caption{Diagrams of linear stability for the $3:2$ balanced
      resonance of \cref{spin_spin_1} for different values of the
      indicated parameters $(\hat q_1,\sigma_1 )$. Blue denotes
      instability and yellow denotes stability. Each of the columns
      have required around 15h using 24 CPUs and a mesh size of
      $512\times 512\times 512$.}
    \label{fig:lin_stab_spin_orbit_high}
\end{figure}

\begin{enumerate}
\renewcommand{\labelenumi}{\theenumi)}
    \item The effect of the new parameters is remarkable for large $e$
      and $|\lambda_1|$. Especially when $\hat q_1$ and $\sigma_1$ are
      large.

    \item At very large eccentricities, the surface has a complicated
      structure resulting in multiple solutions. The effect of $\hat
      q$ is mainly to alter the stability for large $e$: solutions of
      type $\frac{\pi}{2}$ become always unstable, while stable
      regions of solutions of type 0 are more concentrated. The growth
      of $\hat q$ also increases the multiplicity of solutions of type
      0 and very large $e$. On the other hand, increasing $\sigma_1$
      has a more dramatic effect on the complexity of the surface and
      also modifies the stability for large $e$. Actually, for some
      values of $\sigma_1$, the existing V-shaped fold connects with
      the complex structure of large eccentricities in the upper part
      of the diagram.
\end{enumerate}

\subsection{Spin-spin problem ($V_{per} = V_2+V_4$) with non-spherical companion}
\label{nonspher}

We study the linear stability of the resonances of the full coupled
spin-spin model in \cref{spin_spin_j}. In this case linear stability
is determined by a more general theory and we will restrict ourselves
to zones in the space of parameters where there is uniqueness.

Since we have two degrees of freedom, the first variation at a
particular resonance is not a Hill's equation anymore, but a coupled
system of second order. A Hill's equation is a particular case of
linear Hamiltonian system with periodic coefficients, hereafter LPH
systems.

If we define $z=(\theta_1,\theta_2,p_1,p_2)^{\mathsf{T}}$, where
$p_j=\mathcal C_j \dot \theta_j$, the spin-spin problem in the
Hamiltonian form \cref{spin_spin_ham} can be written as
\begin{equation}\label{SS_Ham_sys}
    \dot z = J_2 \partial_z  H_K(t,z),
\end{equation}
where $J_l$ is the square matrix of order $2l$ given by
\begin{equation}\label{J_l}
    J_l = \matdos{0}{\mathds 1_l}{-\mathds 1_l}{0},\quad l=1,2,\dotsc
\end{equation}
with $\mathds 1_l$ the unit matrix of order $l$. The non-autonomous
Hamiltonian $H_K(t,z)$ is given in \cref{H_K} for
$V_{per}=V_2+V_4$. Let $z=z^*(t)$ be a solution of \cref{SS_Ham_sys}
that is in a balanced spin-spin resonance of type
$(m_1:2,m_2:2)$. Then, the first variation at such solution has the
form
\begin{equation}\label{SS_Ham_sys_lin}
    \dot y = J_2 \partial_{z,z}  H_K(t,z^*(t)) y, \quad y\in \R^4,
\end{equation}
where $\partial_{z,z} H_K$ denotes the Hessian matrix of the
Hamiltonian $H_K$ in the 4-dimensional variable $z$. The system
\cref{SS_Ham_sys_lin} is an LPH system of period $2\pi$. Assume that
$\Phi(t)$ is a matrix solution of \cref{SS_Ham_sys_lin} with
$\Phi(t_0)=\mathds 1_4$. The stability of \cref{SS_Ham_sys_lin} is
determined by the Floquet multipliers, that are the eigenvalues of the
monodromy matrix $\Phi(t_0+2\pi)$.

An LPH system has particular stability properties, see the general
theory in \cite{yak1975, eke} and an application to the double
synchronous spin-spin resonance in \cite{mis2021}. For example, assume
that $\varphi \in \C$ is a Floquet multiplier of an LPH system. Then,
its inverse $\varphi^{-1}$, its complex conjugates $\bar \varphi$ and
$\bar\varphi^{-1}$ are also multipliers and have the same multiplicity
as $\varphi$. That is, the Floquet multipliers have a symmetric
distribution with respect to the real line and the unit circle of the
complex plane. In consequence, a necessary condition for stability is
that all multipliers have modulus 1.  Moreover, if all multipliers
have modulus 1 and they are all different, then the system is
stable. When all the multipliers have modulus 1 and some of them
coincide, the situation is not trivial and the stability depends on
further algebraic properties of the multipliers (Krein's theory,
\cite{Krein}). Then, unlike for Hill's equations, here we do not have
a quantity like the trace of the monodromy matrix in order to
characterize the boundary of stability/instability regions. Instead,
we will use the following definition.

\begin{definition}\label{def:hyperbolic}
    We will say that an LPH system is \textit{hyperbolic unstable} if
    \begin{equation*}
        \max_{k} |\varphi_k|>1,
    \end{equation*}
    where $\varphi_k$, $k=1,2,\dotsc$, are all the Floquet multipliers
    of the system.
\end{definition}

Assume that the solution $z^*(t)$ is continuous in some domain of the
parameters of the model $(e;\mathcal
C_1,\lambda_1,\lambda_2,\sigma_1,\hat q_1,\hat q_2)$.  The equation
$\displaystyle \max_{k=1,\dotsc,4} |\varphi_k|=1+\varepsilon$, with
$\varepsilon>0$, defines a 1-parametric family of hyperbolic unstable
manifolds in the space of parameters; then, the boundary of hyperbolic
instability will be found if we take the limit of such manifolds as
$\varepsilon\rightarrow 0$.

On the other hand, it is possible to find all the solutions of a given
type of a balanced spin-spin resonance using the shooting method as in
the previous section. However, since the phase space and the space of
parameters have large dimensions, our approach will be to obtain the
solutions by continuation. Note that the solutions of
\cref{spin_spin_j} for $\lambda_j=0$ and any $e\in [0,1)$ are exactly
  given by $\theta_j(t)=\theta_j(0)+\dot \theta_j(0)t$. Then, the
  unique solution of type $(\beta_1,\beta_2)$ of a balanced spin-spin
  resonance $(m_1:2,m_2:2)$ is just
  $\theta_j(t)=\beta_j+\frac{m_j}{2}t$. Such solution can be continued
  for $|\lambda_j|>0$. For small enough $|\lambda_j|$, the systems in
  \cref{spin_orbit_j,spin_spin_j,spin_spin_1} can be regarded as
  different perturbations of the system $\ddot \theta_j=0$, then,
  \Cref{fig:lin_stab_spin_orbit,fig:lin_stab_spin_orbit_high} give us
  a quite clear idea of a region of uniqueness of a given type
  associated to a balanced spin-spin resonance of
  \cref{spin_spin_j}. Moreover, such solution can be found by
  continuation in the space of parameters.

\begin{figure}[ht]
    \centering
    \includegraphics[scale=.35]{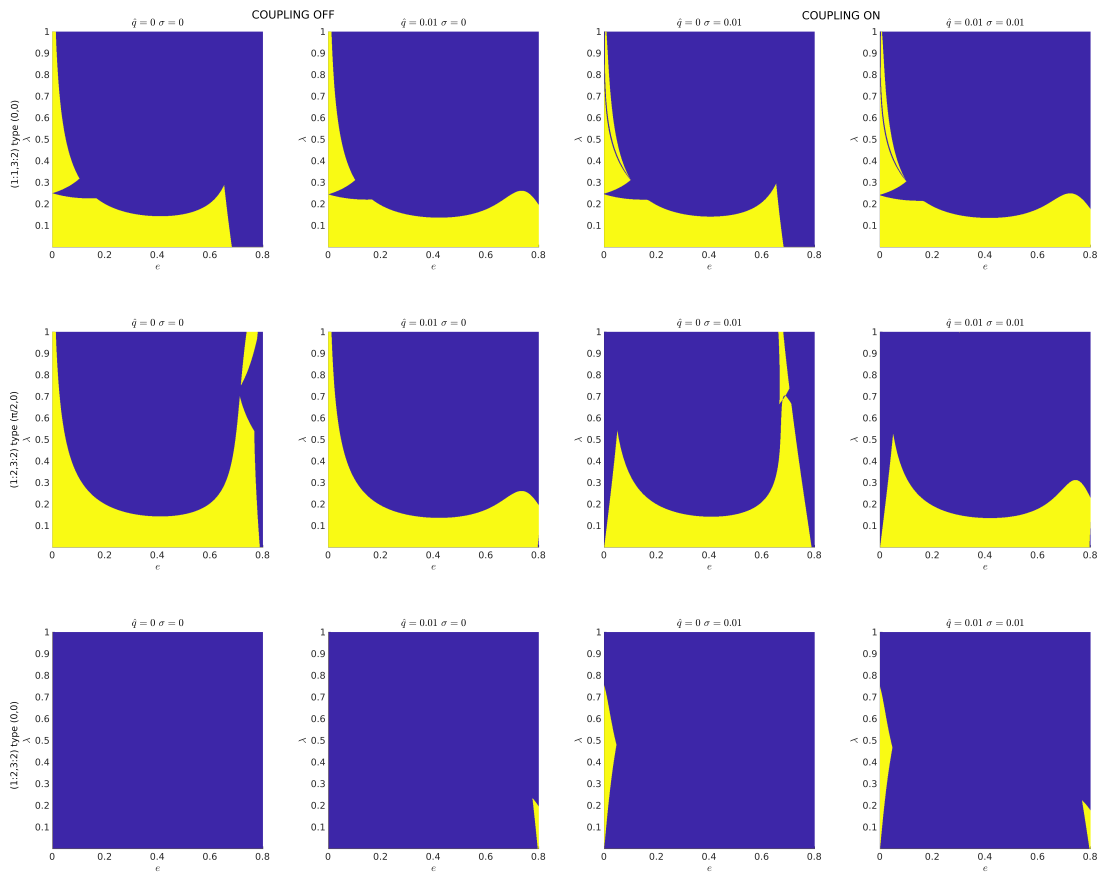}
    \caption{Regions of (hyperbolic) linear instability of solutions
      associated to different spin-spin resonances of the problem
      \cref{spin_spin_j} (case: equal bodies) for different values of
      the parameters $(\hat q,\sigma)$.  Blue denotes instability and
      yellow denotes stability. Each plot has required around 10
      minutes using 15 CPUs and a mesh size $750 \times 750$.}
    \label{fig:lin_stab_spin_spin_equal_same}
\end{figure}

\subsubsection*{Case of equal bodies $\mathcal E_1=\mathcal E_2$}
In this case, the equation \cref{spin_spin_j} depends on four
parameters $(e,\lambda,\hat q,\sigma)$, where
$\lambda=\lambda_1=\lambda_2$, $\hat q=\hat q_1=\hat q_2$ and
$\sigma=\sigma_1=\sigma_2$. From \Cref{fig:lin_stab_spin_orbit_high},
it looks that a good choice for doing the continuation is in the range
$\ 0\le e\lesssim 0.8$, $ 0\le \lambda\lesssim 1$ and $0\le \hat q,
\sigma \lesssim 0.01$. In this range, the linear stability of the
solution of type $(0,0)$ associated to the resonance of order
$(1:1,1:1)$ was investigated in
\cite{mis2021}. \Cref{fig:lin_stab_spin_spin_equal_same} shows the
stability diagrams of the resonance $(1:1,3:2)$ of type $(0,0)$ and
the resonance $(1:2,3:2)$ of type $(\frac{\pi}{2},0)$ and $(0,0)$ for
different $(\hat q,\sigma)$. Let us point out some properties:

\begin{enumerate}
\renewcommand{\labelenumi}{\theenumi)}
 \item Note that a plot with $(\hat q,\sigma)=(0,0)$ is just the
   superposition of diagrams in \Cref{fig:lin_stab_spin_orbit}, while
   a plot with $(\hat q,\sigma)=(0.01,0)$ is the superposition of
   diagrams similar to those in
   \Cref{fig:lin_stab_spin_orbit_high}. Then, the effect of the
   coupling can be seen in the plots with $\sigma\ne 0$.

 \item The effect of the coupling is different in each case: for the
   resonance $(1:2,3:2)$ of type $(0,0)$, we only see an additional
   thin unstable region for $e<0.1$ and $0.25<\lambda<1$. For the
   resonance $(1:2,3:2)$ of type $(\frac{\pi}{2},0)$, we see that the
   region with small $e$ becomes unstable, whereas for large $e$ the
   stability is somewhat altered. Finally, without coupling, the
   resonance $(1:2,3:2)$ of type $(0,0)$ is unstable for almost all
   the points, but the coupling introduces the stability for small
   $e$. Note that this in agreement with \Cref{fig:sync_3_2_to_1_2}.
\end{enumerate}

\section{Comparison between the full and the Keplerian models}
\label{sec:comparison}

In this section, we provide some results on the comparison between the
full and Keplerian problems (Section~\ref{sub:Ham}), and we give some
numerical results on the interaction between the spin and orbital
motions (Section~\ref{sec:quantitative}).

\subsection{Hamiltonian approach}\label{sub:Ham}

It will be useful to write the dynamical equations of the Hamiltonian
of the full model \cref{H_theta} in the compact form
\begin{equation}\label{full_Ham_sys}
    \dot z = J_4 \partial_z  H(z),
\end{equation}
where $z=(r, f, \theta_1, \theta_2, p_r, p_f, p_1, p_2)^{\mathsf T}$
and $J_4$ is defined by \cref{J_l}.

Let us now formulate the Keplerian models (spin-orbit and spin-spin,
see \Cref{sec:Kepler}) as perturbations of the full model, so we can
compare both families of models. Consider a function $\zeta(t)$ that
satisfies the equation
\begin{equation}\label{per_Ham_sys}
    \dot \zeta(t) = J_4 \partial_z H(\zeta(t)) - h(\zeta(t))\ ,
\end{equation}
where the vector function
\begin{equation*}
  h(z)=(0, 0, 0, 0, -\partial_r V_{per}(z), -\partial_f V_{per}(z), 0,
  0)^{\mathsf T}
\end{equation*}
is responsible of subtracting the perturbative part of the potential
(see equation \cref{V0per}) only in the equations for $\dot p_r$ and
$\dot p_f$. Thus, equation \cref{per_Ham_sys} represents the Keplerian
models including the orbital part: the spin-orbit model corresponds to
\cref{per_Ham_sys} with $V_{per}=V_2$ and the spin-spin model with
$V_{per}= V_2+V_4$. The corresponding equations of motion are
\cref{orbital_eqs_kep} and \cref{spin-spin_V_trun}, so we can write
the solution in the form
\begin{equation*}
    \zeta(t) = (r(t;a,e), f(t;e), \theta_1(t), \theta_2(t),
    p_r(t;a,e), p_f(a,e), p_1(t), p_2(t)),
\end{equation*}
where $a$ and $e$ are the semimajor axis and the eccentricity of the
Keplerian orbit, see \cref{kepler_orbit}. Now it is clear that the
Keplerian model \cref{per_Ham_sys} is not Hamiltonian, even though we
can split it into one autonomous Hamiltonian system (orbital part with
$V=V_0$) and another non-autonomous one (spin part with
$V=V_0+V_{per}$). Moreover, unlike for the full model, in a Keplerian
model neither the total angular momentum $P_f=p_f+p_1+p_2$ is
conserved\footnote{Instead, the Keplerian assumption results in the
conservation of the orbital angular momentum $p_f(t)=\mu r(t;a,e)^2
\dot f(t;e)=\mu a^2 \sqrt{1-e^2}$.}, since we can compute that
\begin{equation*}
    \dot P_f (t) = -\sum_{j=1}^2\partial_{\theta_j}
    V_{per}(r(t;a,e),f(t;e),\theta_1(t),\theta_2(t)).
\end{equation*}
We want to know if $\zeta(t)$ is a good approximation for a solution
of the full problem \cref{full_Ham_sys} with the same initial
conditions. Consider the solution $z=z(t)=\zeta(t)-\delta z(t)$ of
\cref{full_Ham_sys} such that $z(0)=\zeta(0)$. Then, expanding the
equation \cref{full_Ham_sys} up to first order in $\delta z$, we
obtain that the function $\delta z =\delta z(t)$ satisfies the
equation
\begin{equation}\label{ec_delta_z}
    \frac{\d }{\d t} \delta z = J_4 \partial^2_{z,z} H (\zeta(t))
    \delta z - h(\zeta(t))
\end{equation}
with initial condition $\delta z(0)=0$. In equation \cref{ec_delta_z},
$\partial^2_{z,z} H(z)$ is the Hessian matrix associated to the
Hamiltonian in \cref{H_theta}. If the system \cref{ec_delta_z} is
stable, then the norm $||\zeta(t)-z(t)||$ is bounded. Additionally, it
is easy to see that the system \cref{ec_delta_z} is Lyapunov stable if
and only if the trivial solution of the homogeneous part
\begin{equation}\label{hom_linear}
    \dot y = J_4 \partial^2_{z,z} H (\zeta(t))  y, \quad y\in \R^8,
\end{equation}

\noindent is Lyapunov stable. A general form of $\zeta(t)$ is unknown
because, although the orbital part is given by the classical Kepler
problem, the spin part is given by a non-autonomous periodic system of
nonlinear equations. However, $\zeta(t)$ is a periodic solution at
spin-spin resonances, with period $2\pi$ for balanced resonances. In
such cases, equation \cref{hom_linear} is an LPH system, some of whose
properties were mentioned in \Cref{nonspher}.

At this point, we wonder if it is possible for the periodic system
\cref{hom_linear} to be stable. Let us point out an argument
supporting a negative answer, even for stable spin-spin resonances. It
is known that the periodic solutions of the planar Kepler problem are
not linearly stable. This can be easily checked using Poincar\'e
variables (action-angle variables): we obtain a positive eigenvalue
for a linear system of constant coefficients. See further related
discussions in \cite{bosort2016,sch1972}.  We can see that 1 is the
only associated Floquet multiplier and has multiplicity four; then,
the instability of the periodic solutions of the Kepler problem is not
hyperbolic, according to \Cref{def:hyperbolic}, but rather of
parabolic kind. From this discussion, we expect that $\zeta(t)$ and
$z(t)$ are divergent functions, but we want to know if such divergence
is exponential in time, that is, when \cref{hom_linear} is hyperbolic
unstable.

Suppose that the function $\zeta(t)$ is continuous on a domain of the
space of parameters $(a,e;\mathcal C_1, \lambda_1, \lambda_2,
\sigma_1,\hat q_1,\hat q_2)$. Note that we add $a$ to the parameters
of the spin-spin model because it varies with the orbital initial
conditions. On the other hand, the Hamiltonian $H$ depends on the
parameters of the full model, we take the independent set
$(\mu,\mathcal C_1,d_1,d_2,q_1,q_2)$, where the value of $\mu$ informs
us about the disparity in the masses of the bodies because $\mu=M_1
M_2=M_1(1-M_1)$. We can obtain $(\mu,\mathcal C_1,d_1,d_2,q_1,q_2)$
from $(a,e;\mathcal C_1, \lambda_1, \lambda_2, \sigma_1,\hat q_1,\hat
q_2)$ as follows: First, from \cref{lambda_j}, we see that, assuming
$\sigma_1>0$,
\begin{equation}
  \label{eq.sigma1-and-a}
  \mu = \frac{\mathcal C_1}{3 \sigma_1 a^2}\ ;
\end{equation}
from this value we can obtain $M_1$ because $\mu =M_1(1-M_1) $ with
$M_1\in (0,1)$. This is enough to write $M_2=1-M_1$, $\mathcal
C_2=1-\mathcal C_1$ and, from the definitions \cref{lambda_j}, it
holds that
\begin{equation*}
  d_1=\frac{\lambda_1 \mathcal C_1}{3M_2},\quad
  d_2=\frac{\lambda_2 \mathcal C_2}{3M_1},\quad
  q_1=\hat q_1M_1a^2,\quad
  q_2=\hat q_2M_2a^2\ .
\end{equation*}

\noindent From these relations, it is obvious that, in order to
compare the full and the Keplerian models, the spin-orbit versions are
not enough, and we need to take $V_{per}=V_2+V_4$, say, the spin-spin
models.

Additionally, the initial conditions associated to $\zeta(t)$ are
given by \cref{kepler_in}, for the orbital part, and the values
$\theta_j(0),\dot \theta_j(0)$ are such that the spin part satisfies
the boundary conditions for the spin-spin problem in a spin-spin
resonance of a certain type as in \cref{Dirichlet_balSS}. Recall also
that, in our units, the Keplerian orbital period is $T=2\pi$ and
$G=a^3$.

As in \Cref{nonspher}, we can find the region (in the parameters
space) of hyperbolic instability and its boundary, that is given by
the limit as $\varepsilon\rightarrow 0$ of the manifolds that satisfy
$\displaystyle \max_{k=1,\dotsc,8} |\varphi_k|=1+\varepsilon$, where
$\varphi_k$ are the Floquet multipliers of \cref{hom_linear}. It is
important that we focus on a region of the parameters where the
corresponding solution associated to a spin-spin resonance is linearly
stable, so we can see what is the effect of putting the orbital and
the spin parts altogether. Actually, not even the spin part of
\cref{hom_linear} is completely equivalent to \cref{SS_Ham_sys_lin},
because all the derivatives of $V_{per}(r,f,\theta_1,\theta_2)$,
evaluated at $\zeta(t)$, appear in \cref{hom_linear} for both orbital
and spin variables.

\begin{figure}[ht]
    \centering
    \scalebox{1}{\includegraphics[scale=.34]{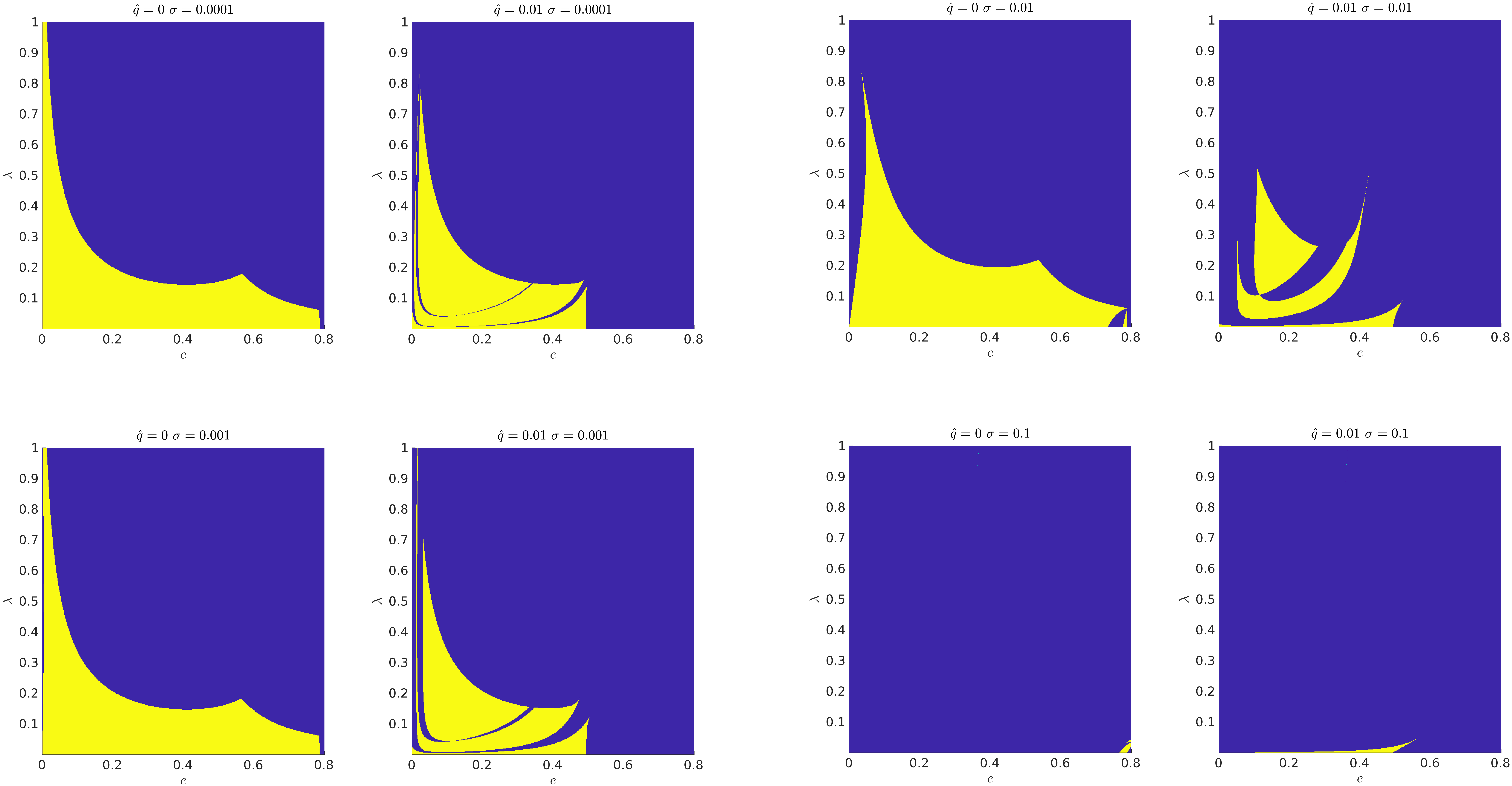}}
    \caption{Regions of (hyperbolic) linear instability of
          $(1:2,3:2)$ of type $(\pi /2, 0)$ of the equation
          \Cref{hom_linear} for equal bodies and different values of
          the parameters. Blue denotes instability and yellow denotes
          stability. Each plot has required around 10 minutes using 15
          CPUs and a mesh size of $750\times 750$.}
    \label{fig:Ham_comp}
\end{figure}

\Cref{fig:Ham_comp} shows the regions of hyperbolic instability of
\cref{hom_linear} in the case of equal bodies for a particular
spin-spin resonance with different values of the parameters. We
observe the following:

\begin{enumerate}
  \item We can compare \Cref{fig:Ham_comp} with
    \Cref{fig:lin_stab_spin_spin_equal_same} for similar values of
    $\hat q$ and $\sigma$. In the case of equal bodies ($\mu=0.25$ and
    $\mathcal{C}_1 = 0.5$), the value of $a$ is determined by
    \cref{eq.sigma1-and-a}: smaller $\sigma$ also implies further
    bodies. Then, the plots with $\sigma=0.0001$ in
    \Cref{fig:Ham_comp} and be compared approximately with those with
    $\sigma=0$ in \Cref{fig:lin_stab_spin_spin_equal_same}.

  \item From the comparison, we see that the plots in
    \Cref{fig:Ham_comp} with \Cref{fig:lin_stab_spin_spin_equal_same}
    with $\hat q=0$ are similar for small eccentricities, but, for
    large $e$, the unstable regions of plots in \Cref{fig:Ham_comp}
    are larger. On the other hand, the plots in \Cref{fig:Ham_comp}
    with $\hat q=0.01$ include unstable stripes invading the stable
    regions for small $e$ (they grow with $\sigma$), whereas for large
    $e$ the diagrams become totally unstable.

  \item As we increase $\sigma$ (closer bodies), the plots become more
    and more unstable.
\end{enumerate}

We conclude that the spin-spin model should be a reliable
simplification of the full model in the regions of the parameters with
small $e$, large $a$ and close to regions with stable spin-spin
resonances. We remark also that instability of spin-spin resonances
not necessarily implies hyperbolic instability of
\cref{SS_Ham_sys_lin}: there are some regions for which the Floquet
multipliers of the spin-spin resonance are not too far from the unit
circle, so that when we consider the corresponding system
\cref{SS_Ham_sys_lin}, all its Floquet multipliers belong to the unit
circle. It is also remarkable that, although the comparison makes
sense only between the full and Keplerian spin-spin models, we obtain
almost identical plots as in \Cref{fig:Ham_comp} independently if we
use $V_{per}=V_2$ or $V_{per}=V_2+V_4$ in the Hamiltonian in
\cref{hom_linear}.

\subsection{Quantitative numerical approach}
\label{sec:quantitative}

In this section, we want to investigate, from a numerical point of
view, two questions that the Hamiltonian approach left unanswered:
(Q1) Can we quantify the influence of the spin motion on the orbital
one?  (Q2) Is it possible that the bodies end up colliding, even
though \Cref{hom_linear} is not hyperbolic unstable?

In \Cref{sub:Ham}, we saw that the solution of a Keplerian model
$\zeta(t)$ corresponding to a spin-spin resonance should diverge from
a solution of the full model $z(t)$ with identical initial
conditions. In the region of parameters of hyperbolic instability the
divergence should be exponential; however, in the rest of the
parametric space, we want to quantify how different both motions
are. So, let us take the point $(a,e;\mathcal C_1, \lambda_1,
\lambda_2, \sigma_1,\hat q_1,\hat q_2)$ and compute the functions
$\zeta(t)$ and $z(t)$ in order to compare them.

The orbital motion of $\zeta(t)$ is characterized unambiguously by the
set of Keplerian elements $(a,e,\omega)$, where $\omega=0$ is the
argument of the periapsis, together with $t$, the mean anomaly. On the
other hand, let the orbital part of the solution $z(t)$ of the full
model be given by $(r_F(t),f_F(t))$. Then, we can transform the
orbital position $\vec r_F(t)=r_F(t) \exp(if_F(t))$ to the osculating
Keplerian elements $(a_F(t),e_F(t),\omega_F(t))$ of the two-body
problem using the following expressions. From the geometrical identity
$|\dot {\vec r} _F|^2 = G \parentesis{\frac{2}{r_F}-\frac{1}{a_F}}$,
we obtain
\begin{equation*}
  a_F(t) =\parentesis{\frac{2}{r_F} - \frac{\dot r_F^2+\dot
      f_F^2r_F^2}{G}}^{-1}\ .
\end{equation*}
Now let us define the orbital angular momentum per unit mass $\vec h_F
= \vec r_F \wedge \dot {\vec r}_F$, where $\wedge$ is the vector
product, and the eccentricity vector
\begin{equation*}
 \vec e_F(t) = \frac{ \dot {\vec r}_F \wedge \vec h_F}{G} -
    \frac{\vec r_F}{r_F}\ ,
\end{equation*}
whose modulus is given by
\begin{equation*}
 e _F(t) = \sqrt{1-\frac{r_F^4 \dot f_F^2}{Ga_F}}\ .
\end{equation*}
Whenever $e_F\ne0$, we can define $\omega_F\in[-\pi,\pi)$ as the polar
  angle of $\vec e_F$. The mean anomaly associated to the full model
  can be defined too, but we will not use it in our study. The
  balanced spin-spin resonance of order $(m_1:2,m_2:2)$ in the
  function $\zeta(t)$ is characterized by the modified resonant angles
  $\psi_j^{m_j:2}(t)=m_j f(t,e)-2\theta_j(t)$. Seemingly, the spin
  part of the solution $z(t)$ of the full model is given by
  $(\theta_{1,F}(t),\theta_{2,F}(t))$, so, in order to compare with
  $\zeta(t)$, let us define $\psi_{j,F}^{m_j:2}(t)=m_j
  f_F(t)-2\theta_{j,F}(t)$. Now define the functions
\begin{equation}
\label{deltas}
    \delta_a (t)= \frac{a_F(t)-a}{a},\quad \delta_e (t) = e_F(t)-e,
    \quad \delta_{res, j}(t) =\psi_{j}^{m_j:2}(t)-
    \psi_{j,F}^{m_j:2}(t)\ ,
\end{equation}
where $\delta_a$ is the relative deviation in semi-major axis, while
$\delta_e$ and $\delta_{res,j}$ are the absolute deviations in
eccentricity and resonant angles, respectively.

Finally, recall from \cref{a_j} that we have an expression for
$\mathsf a_j$ in terms of the parameters of the model. Then, for our
purpose, we will say that there is a collision if $r_F(t)\le \mathsf
a_1+\mathsf a_2$. Now we are in a position to compare the solutions of
the full model with respect with those of the Keplerian models.

\begin{figure}
    \centering
    \scalebox{1}{\includegraphics[scale=.45]{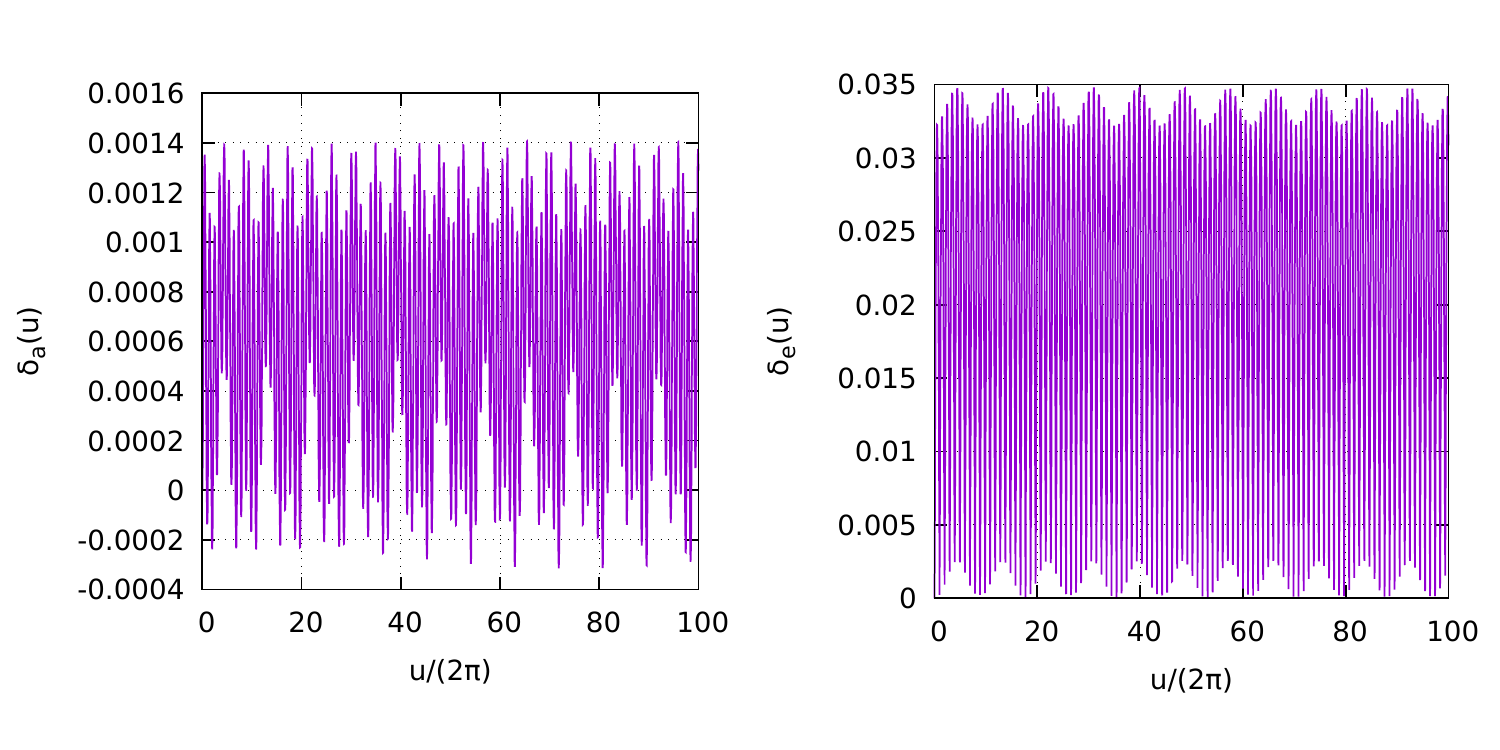}}
    \scalebox{1}{\includegraphics[scale=.45]{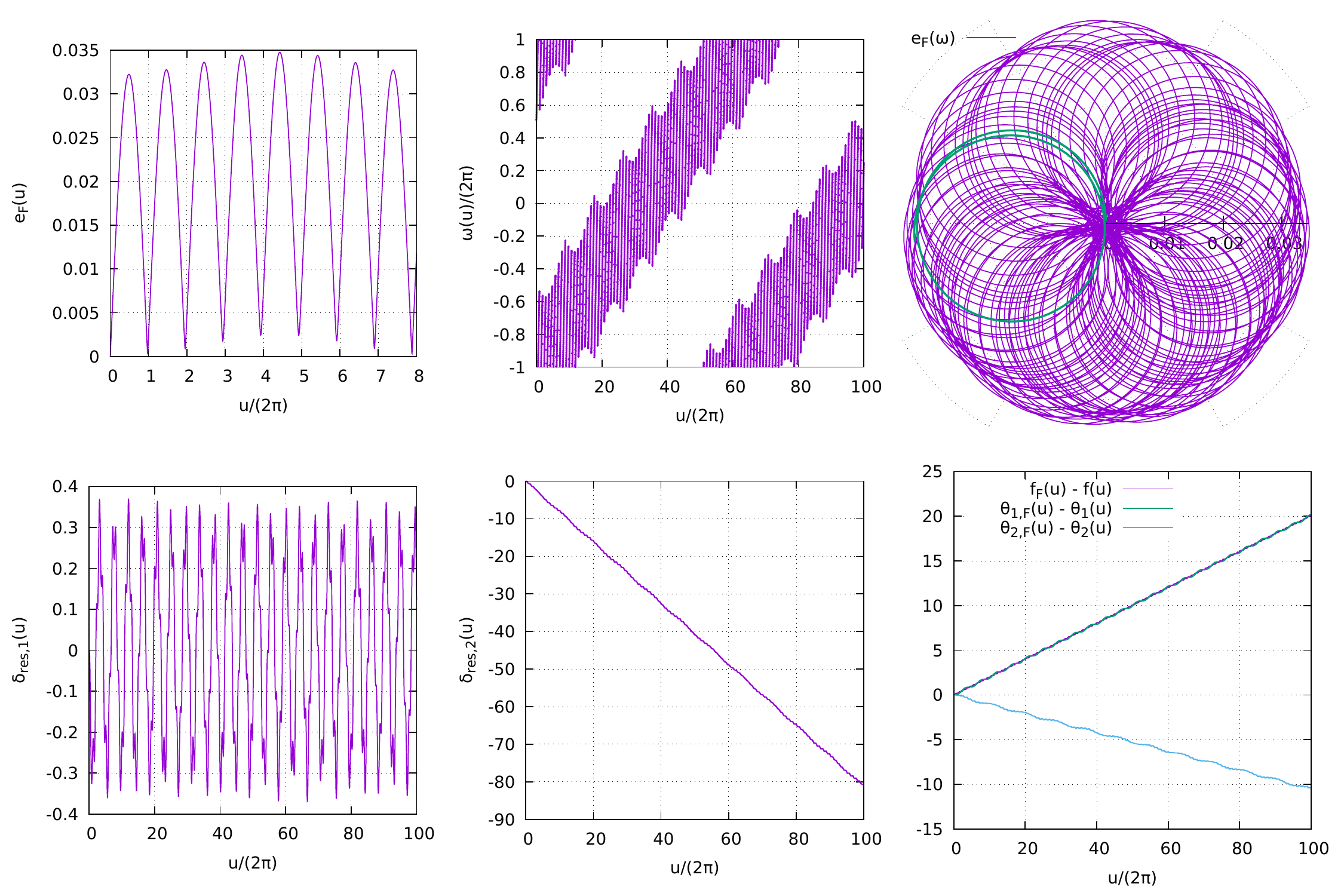}}
    \scalebox{1}{\includegraphics[scale=.45]{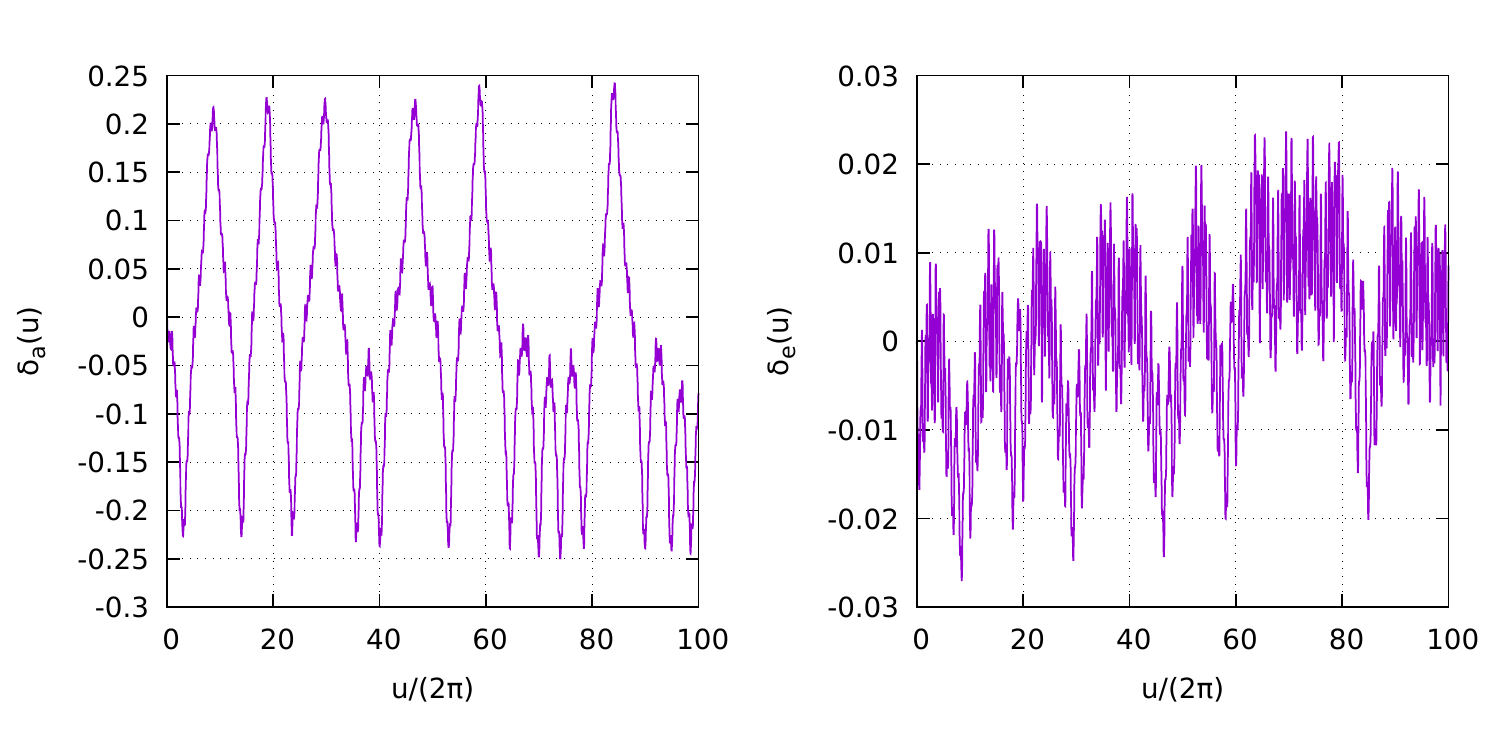}}
    \scalebox{1}{\includegraphics[scale=.45]{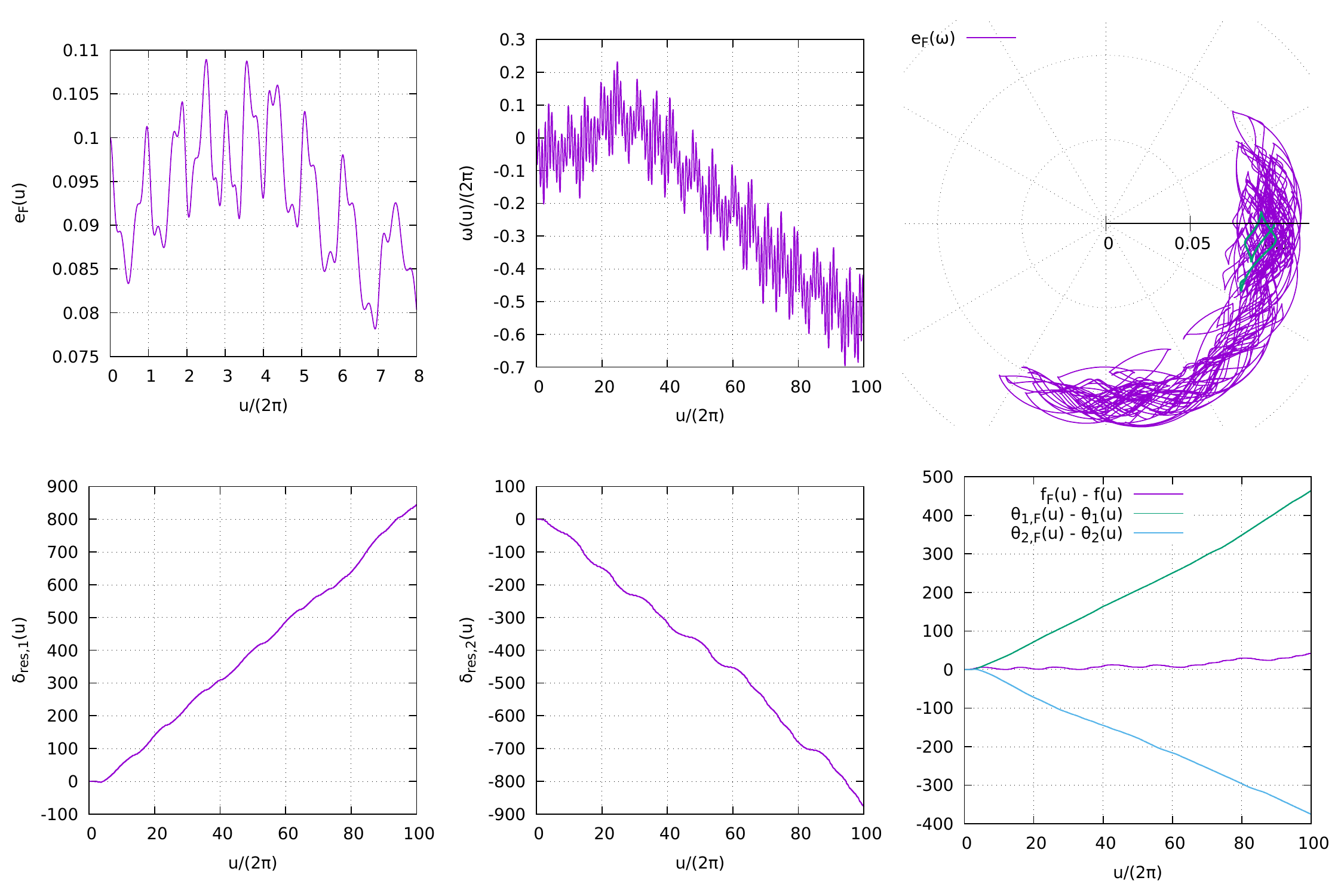}}
    \caption{Comparison between the full and the Keplerian solutions
      in a resonance $(1:1,3:2)$ of type $(0,0)$ for different
      parameter values in the case of equal bodies. Top plots: $e=0$,
      $\lambda = 0.05$, $\hat q = 0.01$, $\sigma = 10^{-3}$. Bottom
      plots: $e=0.1$, $\lambda = 0.05$, $\hat q = 0$, $\sigma =
      10^{-7}$.}
    \label{fig:oscillating1}
\end{figure}


In \Cref{fig:oscillating1} we see, for different values of the
parameters and the case of identical bodies, the comparison between
$z(t)$ and $\zeta(t)$ in a resonance $(1:1,3:2)$ of type
$(0,0)$. Particularly, we see the evolution of the $\delta$-functions
in \cref{deltas} and some Kepler elements of $z(t)$ in 100
revolutions. In addition, \Cref{tab:oscillating1} informs us about the
corresponding Floquet multipliers of both $z(t)$ and $\zeta(t)$. We
summarize the following observations:

\begin{enumerate}
  \item We chose two cases for the comparison. In the first case,
    $e=0$ and $\sigma=10^{-3}$, so the bodies are relatively close to
    each other ($a\approx 39$) in circular orbits, whereas in the
    second case, $e=0.1$ and $\sigma=10^{-7}$, so the bodies are much
    further ($a\approx 3900$) in moderately elliptic orbits. For both
    cases, we took values of parameters whose corresponding resonance
    is not hyperbolic unstable, say, $\lambda=0.05$ and $\hat q=0.01$.

  \item From \Cref{fig:oscillating1}, we see that the first case has a
    more regular behavior than the second one. In the first case, the
    orbit of the full model oscillates regularly very close to the
    circular orbit of the Keplerian model, with a precession that
    increases uniformly in average. Actually, in a shorter time scale,
    the eccentricity vector describes a circles passing through the
    origin. On the other hand, the orbit of the second case is greatly
    irregular, with variations of size up to $\sim 25\% $. Even its
    precession changes direction at some moment and the eccentricity
    vector swings chaotically in short time scales.

  \item The variation of the angles is quite regular in both
    cases. Except for the resonant angle of the resonance 1:1 of the
    first case, which oscillates regularly with small amplitude, the
    rest of the resonant angles increase/decrease more or less
    constantly, with larger variations in the second case.

  \item The Floquet multipliers in \Cref{tab:oscillating1} gives us
    more information about both cases. We remark that we display the
    multipliers of $\zeta(t)$ corresponding to spin and orbit
    separately, say, the first two rows are the four multipliers of
    the spin-spin model \cref{spin_spin_j}, whereas the third one
    shows the four coincident multipliers of the Kepler problem. We
    confirm that the spin of $\zeta(t)$ is elliptic stable in the two
    cases. On the other hand, the multipliers of $z(t)$ are computed
    with \cref{full_Ham_sys}. We find out that, although the first
    case is more regular, the solution is actually hyperbolic
    unstable, whereas the opposite occurs for the second case: it is
    irregular but not hyperbolic. This is consistent with
    \Cref{fig:Ham_comp}.
\end{enumerate}

\begin{table}[ht]
    \[
    \begin{array}{|l|l|c|l|c|} \hline
        &\multicolumn{2}{c|}{e=0, \lambda = 0.05, \hat q = 0.01, \sigma = 10^{-3}}&  \multicolumn{2}{c|}{e=0.1, \lambda = 0.05, \hat q = 0.01, \sigma = 10^{-7}} \\ \hline
        &\text{Argument}  & \text{Modulus} & \centering \text{Argument} & \text{Modulus}\\ \hline \hline
        \multirow{ 3}{*}{$\zeta(t)$} &\pm 1.42 & 1 & \pm 1.39 & 1 \\
        &\pm7.36 \cdot 10^{-8} & 1 & \pm 8.34 \cdot 10^{-1} & 1 \\
        & 0 & 1 \text{   (quadruple)} &  0 & 1 \text{   (quadruple)}\\  \hline
        \multirow{ 4}{*}{$z(t)$} &\pm 1.41 & 1 & \pm 1.39 & 1\\
        & \pm 7.64\cdot 10^{-9} & 1 & \pm 8.34 \cdot 10^{-1} & 1\\
        &\pm 1.65 \cdot 10^{-1} & 1.09 & \pm 3.77\cdot 10^{-4} & 1\\
        &\pm 1.65 \cdot 10^{-1} & 0.915 &  0 & 1 \text{   (double)} \\ \hline
    \end{array}
    \]
    \caption{Modulus and argument of the Floquet multipliers of
          the functions $z(t)$ and $\zeta(t)$ compared in
          Figure~\ref{fig:oscillating1}.}
    \label{tab:oscillating1}
\end{table}

We conclude that, even when there is hyperbolic instability, the
solution of full system remains quite close to the solution of the
Keplerian system with circular orbit. However, even with no hyperbolic
instability, for $e\ne 0$, the behavior of both solutions is
remarkably different. This fact should be attenuated with very small
eccentricities.

We can make a step further in the comparison when
$e=0$. \Cref{fig:comp_circular} shows a quantitative comparison in the
orbits of the full and Keplerian models. We vary the parameters
$\sigma $ and $\lambda$ keeping $\hat q=0$ constant. For large enough
$\sigma$ (small $a$), there is a region of parameters leading to
collision. The rest of the diagrams show different orders of magnitude
of $\delta$-functions corresponding to $a$ and $e$. We check that, for
very small $\lambda$, the orbit of the full model is very close to the
Keplerian one.

\begin{figure}[ht]
    \centering
    \includegraphics[scale=.65]{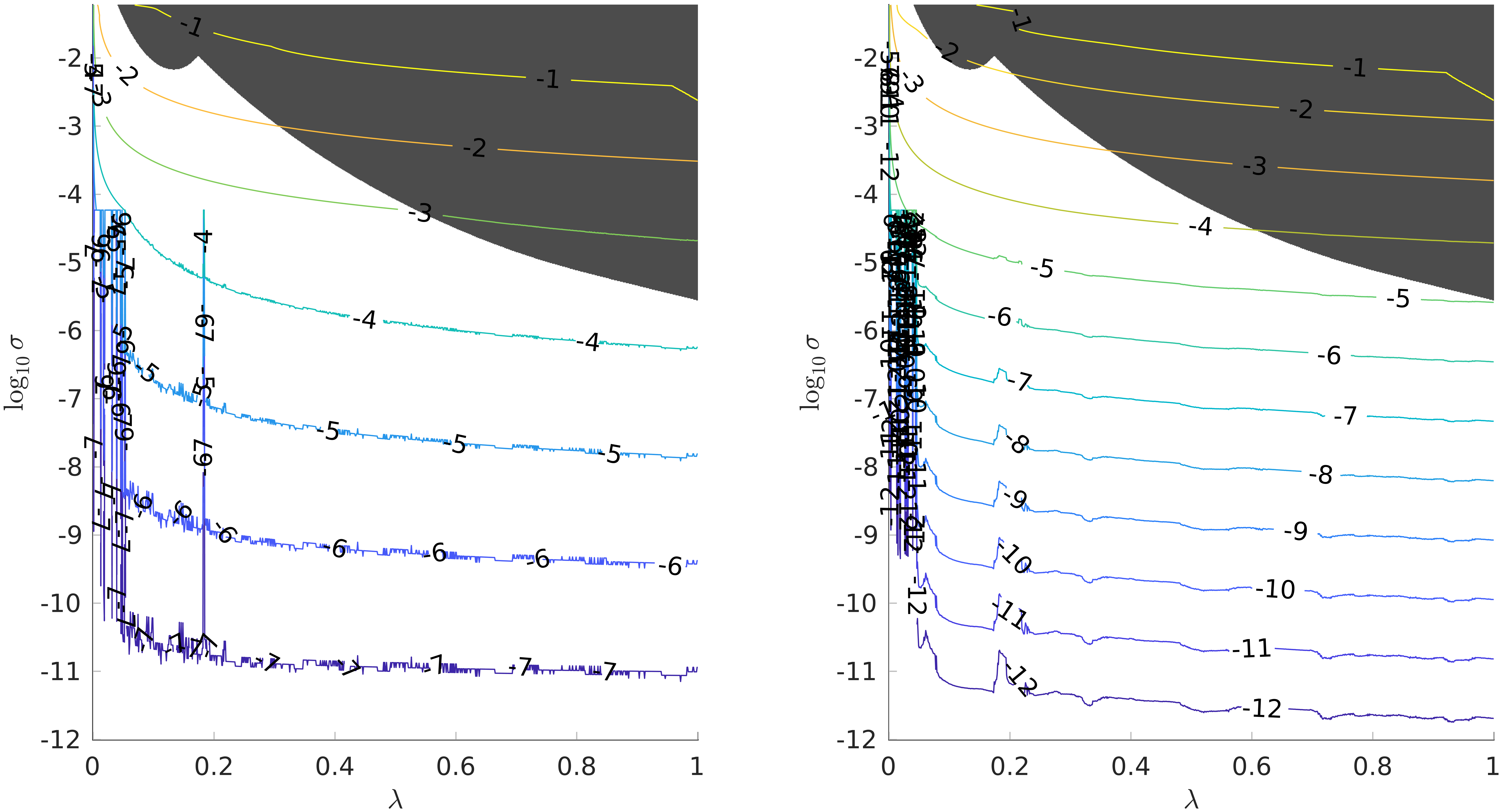}
    \caption{Contour plots of $\max \{ \log _{10}|\delta _e(u)| \colon
      u \in [0,4\pi]\}$ on the left and $\max \{\log _{10}|\delta
      _a(u)| \colon u \in [0,4\pi]\}$ on the right for the resonance
      $(1:1,3:2)$ of type $(0,0)$ and parameter values $\ecc = 0$ and
      $\hat q = 0$ for equal bodies. The black region denotes the
      collision during the integration of the full spin-spin model,
      i.e. when $r_F(u) \leq \mathsf{a}_1 + \mathsf{a} _2$ with
      $\mathsf{a} _j$ defined in \eqref{a_j}. Using 80 CPUs the plots
      required around 1 hour with a mesh of $1024^2$ and $2048$
      stopping points in $[0,4\pi]$.}
    \label{fig:comp_circular}
\end{figure}

\section{Conclusions}

This research is, primarily, a careful numerical study of the
spin-spin model, presented as such in \cite{mis2021}. For this
purpose, we have considered several parallel models describing the
planar gravitational dynamics of two ellipsoids under the
\Cref{a:keplerian,a:planar_motion,a:planar_symmetry}. We distinguish
between full and Keplerian models, and also, between spin-orbit and
spin-spin models. The known dynamics of the Keplerian spin-orbit
problem was used as reference to develop our results.

In first place, we realize that the symmetries of the ellipsoids
lead to certain symmetries in the equations of the Keplerian
models, and so, to symmetric periodic solutions that structure the
overall dynamics. We complete this general view by providing
conditions for existence of quasi-periodic solutions. Between the
periodic solutions, we focus on a special class that comprises the
simplest solutions from the physical interpretation, the balanced
spin-orbit and spin-spin resonances. We study the high-dimensional
phase space of the spin-spin problem by means of projections of
Poincar\'e maps, taking those of the spin-orbit problem as
reference.

First, we see that when one body is spherical, the dynamics of the
spin-orbit problem is reproduced with small variations. For two
ellipsoids, the projections of the Poincar\'e maps show structures
similar to those of the spin-orbit ones. There is a particular
behavior when both bodies are identical, the so-called measure
synchronization.

After that, we study existence, multiplicity and linear stability of
the solutions in balanced resonances as we vary the parameters of the
problem. We focus on the qualitative changes in the stability diagrams
produced by the variation of each parameter. We use the case of
identical bodies to illustrate our observations most of the time.

In the last part of the paper, we compare the solutions of the
full and the Keplerian models by means of rewriting the equations
in a Hamiltonian-like form. We produce stability diagrams similar
to the previous ones to show such a comparison. Representing
different aspects with similar diagrams allows one a much more
global understanding of a problem with so many variables and
parameters. We end up by comparing the evolution of particular
solutions for certain parameters. The orbit of the full problem is
characterized by Kepler's elements to facilitate the comparison.
From the parameters, we observe that the comparison only makes
sense for the spin-spin models. Here we find out that, beside the
case of circular orbits, for which full and Keplerian models show
a close behavior, for eccentric orbits, the trajectories are very
dissimilar. We finish the comparison in the case of circular
orbits by showing that the space of parameters is divided into
regions leading to collision or regions with orbits differing
specific orders of magnitude.

The topics and results presented in this work are just a taster of
some features of the rich spin-spin dynamics. We also include a
comparison with the full models, which is not widely studied in
the literature; however, this is crucial, because it shows us the
limits of validity of the models for applications. We believe that
here is still a lot to explore in this model, which can be used to
investigate problems like Arnold's diffusion or the existence of
normally hyperbolic manifolds.

\appendix

\section{Coefficients of the expansion of the potential}
\label{app:coef}

The full expansion of the potential energy of the Full Two-Body
Problem was derived in \cite{bou2017}. Later, \cite{mis2021} detailed
that expression to the case of planar motion of two ellipsoids, which
is given by
\begin{multline*}
    V=-\frac{G M_1 M_2 }{ r}
    \displaystyle\sum_{\substack{(l_1,m_1)\in \Upsilon \\ (l_2,m_2)\in \Upsilon}}    \Gamma_{l_2,m_2}^{l_1,m_1}
    \parentesis{\frac{R_1}{r}}^{2l_1} \parentesis{\frac{R_2}{r}}^{2l_2} \times \\ \times
    \mathcal Z^{1)}_{2l_1,2m_1} \mathcal Z^{2)}_{2l_2,2m_2}
    \cos (2m_1 (\theta_2-f)+2m_2 (\theta_2-f))\ ,
\end{multline*}
where
\begin{equation*}
    \Upsilon= \{ (l,m)\in \mathbb Z^2 \colon \ 0\le |m| \le l    \}\ ,
\end{equation*}
and, if we take $L=l_1+l_2$ and $M=m_1+m_2$, the constants
\begin{equation*}
  \Gamma_{l_2,m_2}^{l_1,m_1} =
  \frac{(-1)^{L-M}}{4^{L}\sqrt{(2l_1-2m_1)!(2l_1+2m_1)!(2l_2-2m_2)!(2l_2+2m_2)!}}
  \frac{(2L-2M)!(2L+2M)!}{(L-M)!(L+M)!}
\end{equation*}
are numbers dealing with the interaction between the extended
bodies. On the other hand, $R_j$ and $\mathcal Z^{j)}_{2l_j,2m_j}$
are, respectively, the mean radius and the Stokes coefficients of each
$\mathcal E_j$. The quantities $\mathcal Z^{j)}_{l,m}$ provide the
expansion of the potential created for the body $\mathcal E_j$. They
are related to the usual parameters $C^{j)}_{l,m}$ and $S^{j)}_{l,m}$
by $$ C^{j)}_{l,m}+ i S^{j)}_{l,m} = (-1)^m \frac{2}{1+\delta_{m,0}}
\sqrt{\frac{(l-m)!}{(l+m)!}} \bar {\mathcal Z}^{j)}_{l,m}, \quad m\ge
0, $$ where $\delta_{m,n}$ is the Kronecker delta.

\section{A different formulation of the equations of
motion}
\label{app:diff}

To perform the integration of \equ{spin-spin_V_trun}, it is convenient
to adopt the eccentric anomaly $u$ as independent variable, according to
the following procedure.

Let us write the equations \equ{spin-spin_V_trun} as
\begin{equation}
 \label{system}
C_j\frac{d^2 \theta _j}{dt^2}(t)  + \left(\frac{a}{ r}\right)^{5} \eps _j F
_j(t,\theta)=0, \qquad j = 1,2\ ,
\end{equation}
where $\theta = (\theta _1, \theta _2)$, $C _1 + C _2 = 1$ and
$F_j=\partial_{\theta_j}(V_2+V_4)$, $j=1,2$. The explicit expression
for $F_1$ and $F_2$ is given by
\begin{eqnarray}\label{F12}
F_1(t,\theta)&=& \biggl( \left(\frac{r}{a}\right)^2 +
\frac{5}{4} \bigl(\hat q_2+ \frac{5}{7} \hat q_1 \bigr) \biggr)\sin(2 \theta_1-2f)\nonumber\\
&+&\frac{25 \hat d_1}{8}  \sin(4 \theta_1-4f) + \frac{3 \hat
d_2}{8}  \sin(2 \theta_{1}-2\theta_{2})
+ \frac{35 \hat d_2}{8}  \sin(2 \theta_2+2\theta_1-4f)\nonumber\\
F_2(t,\theta)&=&  \biggl(\left(\frac{r}{a}\right)^2 + \frac{5}{4} \bigl(\hat
q_1+ \frac{5}{7} \hat q_2\bigr) \biggr) \sin(2 \theta_2-2f)
+\frac{25 \hat d_2}{8}  \sin(4 \theta_2-4f)\nonumber\\
&-& \frac{3 \hat d_1}{8}  \sin(2 \theta_{1}-2\theta_{2}) +
\frac{35 \hat d_1}{8}  \sin(2 \theta_2+2\theta_1-4f)\ .
\end{eqnarray}
Let us consider the change of variables given by the Kepler's equation
\[
 x _j(u) = \theta _j (u -\ecc \sin u)\ , \qquad j = 1,2\ .
\]
Then, we have
\begin{align*}
 \frac{d^2\theta _j}{dt^2}(t) = \left(\frac{a}{r(u)}\right)^2\frac{d^2 x _j}{d u^2}(u) - \left(\frac{a}{r(u)}\right)^3 \frac{d x _j}{d u}(u) \ecc \sin u, \qquad j = 1,2\ ,
\end{align*}
so that, assuming $C _j\not=0$, \eqref{system} becomes
\begin{equation}
 \label{system2}
 \frac{d^2 x _j}{d u^2}(u) -  \frac{a}{r(u)} \frac{d x _j}{d u}(u) \ecc \sin u  + \left(\frac{a}{ r(u)} \right)^3\frac{\eps _j}{C _j}  F _j(u,x)=0\ , \qquad j = 1,2\ .
\end{equation}
We can write the system \eqref{system2} as
\begin{eqnarray*}
  \frac{d x _j}{d u}(u) &=& y _j(u) \nonumber\\
  \frac{d y _j}{d u}(u) &=& \frac{a}{r(u)} y _j(u) \ecc \sin u  - \left(\frac{a}{ r(u)} \right)^3\frac{\eps _j}{C _j}  F _j(u,x) \ ,
\end{eqnarray*}
for $j = 1,2$ and $x = (x _1, x _2)$. From the well-known
relations used in the study of Kepler's problem
\[
 \cos f = \frac{\cos u - \ecc}{1 - \ecc \cos u}\ , \qquad \sin f =
 \frac{\sqrt{1 - \ecc^2} \sin u}{1 - \ecc \cos u}\ ,
\]
we can define the functions $s=s(x_j)$, $c=c(x_j)$ as
\begin{equation*}
 \begin{split}
  s(x _j) &= \sin (2x _j) (2 \cos^2 f - 1) - \cos (2x _j) 2\cos f \sin f \\
  c(x _j) &= \cos (2x _j) (2 \cos^2 f - 1) + \sin (2x _j) 2\cos f \sin f\ ,
 \end{split}
\end{equation*}
so that $F_1$ and $F_2$ in \equ{F12} take the form
\begin{eqnarray*}
F_1(u,x)&=& \biggl(\left(\frac{r(u)}{a}\right)^2 +
\frac{5}{4} \bigl(\hat q_2+ \frac{5}{7} \hat q_1 \bigr) \biggr) s(x _1)\nonumber\\
&+&\frac{25 \hat d_1}{4}  s(x _1)c(x _1) + \frac{3 \hat d_2}{8}
\sin(2 x_{1}-2 x_{2})
+ \frac{35 \hat d_2}{4}  s(x _1 + x _2)c(x _1 + x _2)\nonumber\\
F_2(u,x)&=&  \biggl(\left(\frac{r(u)}{a}\right)^2 + \frac{5}{4} \bigl(\hat q_1+
\frac{5}{7} \hat q_2\bigr) \biggr) s(x _2)
+\frac{25 \hat d_2}{4}  s(x _1)c(x _1)\nonumber\\
&-& \frac{3 \hat d_1}{8}  \sin(2 x _{1}-2 x_{2}) + \frac{35 \hat
d_1}{4}   s(x _1 + x _2)c(x _1 + x _2)\ .
\end{eqnarray*}

\section{Expansion of $V_2$ and $V_4$}
\label{app:expansion}

We give below the expansion of $V_2$ and $V_4$ up to second order in
the eccentricity:
\begin{eqnarray*}
V_2&=&-\frac{3 d_1 e^2 G M_2 \cos (2 \theta_1)}{8 a^3}-\frac{3 d_2 e^2 G M_1
   \cos (2 \theta_2)}{8 a^3}-\frac{27 d_1 e^2 G M_2 \cos (2 t-2
   \theta_1)}{8 a^3}\nonumber\\
   &-&\frac{3 d_1 e^2 G M_2 \cos (4 t-2 \theta_1)}{4
   a^3}-\frac{27 d_2 e^2 G M_1 \cos (2 t-2 \theta_2)}{8 a^3}-\frac{3 d_2
   e^2 G M_1 \cos (4 t-2 \theta_2)}{4 a^3}\nonumber\\
   &-&\frac{9 d_1 e G M_2 \cos (t-2
   \theta_1)}{8 a^3}-\frac{9 d_1 e G M_2 \cos (3 t-2 \theta_1)}{8
   a^3}-\frac{9 d_2 e G M_1 \cos (t-2 \theta_2)}{8 a^3}\nonumber\\
   &-&\frac{9 d_2 e G
   M_1 \cos (3 t-2 \theta_2)}{8 a^3}-\frac{3 d_1 G M_2 \cos (2 t-2
   \theta_1)}{4 a^3}-\frac{3 d_2 G M_1 \cos (2 t-2 \theta_2)}{4
   a^3}\nonumber\\
   &-&\frac{3 e^2 G M_2 q_1 \cos (2 t)}{8 a^3}-\frac{3 e^2 G M_1 q_2 \cos (2
   t)}{8 a^3}-\frac{9 e^2 G M_2 q_1}{8 a^3}-\frac{9 e^2 G M_1 q_2}{8 a^3}-\frac{3 e
   G M_2 q_1 \cos (t)}{4 a^3}\nonumber\\
   &-&\frac{3 e G M_1 q_2 \cos (t)}{4 a^3}-\frac{G M_2
   q_1}{4 a^3}-\frac{G M_1 q_2}{4 a^3}\ ,\nonumber
\end{eqnarray*}
\begin{eqnarray*}
V_4&=&-\frac{225 G M_2 q_1^2 e^2}{112 a^5 M_1}-\frac{225 G \cos (2 t) M_2 q_1^2 e^2}{224
   a^5 M_1}-\frac{225 G M_1 q_2^2 e^2}{112 a^5 M_2}-\frac{225 G \cos (2 t) M_1
   q_2^2 e^2}{224 a^5 M_2}\nonumber\\
   &-&\frac{105 G \cos (2 t-2 \theta_1-2 \theta_2) d_1 d_2 e^2}{16 a^5}-\frac{525 G \cos (4 t-2 \theta_1-2
   \theta_2) d_1 d_2 e^2}{16 a^5}\nonumber\\
   &-&\frac{315 G \cos (6 t-2 \theta_1-2 \theta_2) d_1 d_2 e^2}{32 a^5}
   -\frac{45 G \cos (2 \theta_1-2 \theta_2) d_1 d_2 e^2}{16 a^5}\nonumber\\
   &-&\frac{45 G \cos (2 t+2 \theta_1-2 \theta_2) d_1 d_2 e^2}{64 a^5}-\frac{45 G \cos (2 t-2 \theta_1+2 \theta_2) d_1 d_2 e^2}{64 a^5}-\frac{225 G d_1^2 M_2 e^2}{224 a^5
   M_1}\nonumber\\
   &-&\frac{225 G \cos (2 t) d_1^2 M_2 e^2}{448 a^5 M_1}-\frac{75 G \cos (2 t-4
   \theta_1) d_1^2 M_2 e^2}{32 a^5 M_1}-\frac{375 G \cos (4 t-4
   \theta_1) d_1^2 M_2 e^2}{32 a^5 M_1}\nonumber\\
   &-&\frac{225 G \cos (6 t-4
   \theta_1) d_1^2 M_2 e^2}{64 a^5 M_1}-\frac{75 G \cos (2 t-2
   \theta_2) d_2 q_1 e^2}{8 a^5}-\frac{165 G \cos (4 t-2 \theta_2)
   d_2 q_1 e^2}{64 a^5}\nonumber\\
   &-&\frac{135 G \cos (2 \theta_2) d_2 q_1 e^2}{64
   a^5}-\frac{375 G \cos (2 t-2 \theta_1) d_1 M_2 q_1 e^2}{56 a^5
   M_1}-\frac{825 G \cos (4 t-2 \theta_1) d_1 M_2 q_1 e^2}{448 a^5
   M_1}\nonumber\\
   &-&\frac{675 G \cos (2 \theta_1) d_1 M_2 q_1 e^2}{448 a^5
   M_1}-\frac{75 G \cos (2 t-2 \theta_1) d_1 q_2 e^2}{8 a^5}-\frac{165 G
   \cos (4 t-2 \theta_1) d_1 q_2 e^2}{64 a^5}\nonumber\\
\end{eqnarray*}

\begin{eqnarray*}
   &-&\frac{135 G \cos (2
   \theta_1) d_1 q_2 e^2}{64 a^5}-\frac{45 G q_1 q_2 e^2}{8 a^5}-\frac{45
   G \cos (2 t) q_1 q_2 e^2}{16 a^5}-\frac{375 G \cos (2 t-2 \theta_2) d_2
   M_1 q_2 e^2}{56 a^5 M_2}\nonumber\\
   &-&\frac{825 G \cos (4 t-2 \theta_2) d_2 M_1
   q_2 e^2}{448 a^5 M_2}-\frac{675 G \cos (2 \theta_2) d_2 M_1 q_2
   e^2}{448 a^5 M_2}-\frac{225 G d_2^2 M_1 e^2}{224 a^5 M_2}-\frac{225 G \cos (2 t)
   d_2^2 M_1 e^2}{448 a^5 M_2}\nonumber\\
   &-&\frac{75 G \cos (2 t-4 \theta_2) d_2^2 M_1
   e^2}{32 a^5 M_2}-\frac{375 G \cos (4 t-4 \theta_2) d_2^2 M_1 e^2}{32 a^5
   M_2}-\frac{225 G \cos (6 t-4 \theta_2) d_2^2 M_1 e^2}{64 a^5
   M_2}\nonumber\\
   &-&\frac{225 G \cos (t) M_2 q_1^2 e}{224 a^5 M_1}-\frac{225 G \cos (t) M_1
   q_2^2 e}{224 a^5 M_2}-\frac{525 G \cos (3 t-2 \theta_1-2 \theta_2) d_1 d_2 e}{64 a^5}\nonumber\\
   &-&\frac{525 G \cos (5 t-2 \theta_1-2 \theta_2) d_1 d_2 e}{64 a^5}-\frac{45 G \cos (t+2 \theta_1-2 \theta_2) d_1 d_2 e}{64 a^5}\nonumber\\
   &-&\frac{45 G \cos (t-2 \theta_1+2 \theta_2) d_1 d_2 e}{64 a^5}-\frac{225 G \cos (t) d_1^2 M_2 e}{448 a^5
   M_1}\nonumber-\frac{375 G \cos (3 t-4 \theta_1) d_1^2 M_2 e}{128 a^5
   M_1}\nonumber\\
   &-&\frac{375 G \cos (5 t-4 \theta_1) d_1^2 M_2 e}{128 a^5
   M_1}-\frac{75 G \cos (t-2 \theta_2) d_2 q_1 e}{32 a^5}-\frac{75 G \cos
   (3 t-2 \theta_2) d_2 q_1 e}{32 a^5}\nonumber\\
   &-&\frac{375 G \cos (t-2 \theta_1) d_1 M_2 q_1 e}{224 a^5 M_1}-\frac{375 G \cos (3 t-2 \theta_1) d_1
   M_2 q_1 e}{224 a^5 M_1}-\frac{75 G \cos (t-2 \theta_1) d_1 q_2 e}{32
   a^5}\nonumber\\
   &-&\frac{75 G \cos (3 t-2 \theta_1) d_1 q_2 e}{32 a^5}-\frac{45 G \cos
   (t) q_1 q_2 e}{16 a^5}-\frac{375 G \cos (t-2 \theta_2) d_2 M_1 q_2
   e}{224 a^5 M_2}\nonumber\\
   &-&\frac{375 G \cos (3 t-2 \theta_2) d_2 M_1 q_2 e}{224
   a^5 M_2}-\frac{225 G \cos (t) d_2^2 M_1 e}{448 a^5 M_2}-\frac{375 G \cos (3 t-4
   \theta_2) d_2^2 M_1 e}{128 a^5 M_2}\nonumber\\
   &-&\frac{375 G \cos (5 t-4
   \theta_2) d_2^2 M_1 e}{128 a^5 M_2}-\frac{45 G M_2 q_1^2}{224 a^5
   M_1}-\frac{45 G M_1 q_2^2}{224 a^5 M_2}-\frac{105 G \cos (4 t-2 \theta_1-2 \theta_2) d_1 d_2}{32 a^5}\nonumber\\
   &-&\frac{9 G \cos (2 \theta_1-2
   \theta_2) d_1 d_2}{32 a^5}-\frac{45 G d_1^2 M_2}{448 a^5 M_1}-\frac{75
   G \cos (4 t-4 \theta_1) d_1^2 M_2}{64 a^5 M_1}\nonumber\\
   &-&\frac{15 G \cos (2 t-2
   \theta_2) d_2 q_1}{16 a^5}-\frac{75 G \cos (2 t-2 \theta_1) d_1
   M_2 q_1}{112 a^5 M_1}\nonumber\\
   &-&\frac{15 G \cos (2 t-2 \theta_1) d_1 q_2}{16
   a^5}-\frac{9 G q_1 q_2}{16 a^5}-\frac{75 G \cos (2 t-2 \theta_2) d_2
   M_1 q_2}{112 a^5 M_2}-\frac{45 G d_2^2 M_1}{448 a^5 M_2}\nonumber\\
   &-&\frac{75 G \cos (4t-4 \theta_2) d_2^2 M_1}{64 a^5 M_2}\ .
\end{eqnarray*}

\bibliographystyle{alpha}
\bibliography{references}

\begin{thebibliography}{CCGdlL21b}

\bibitem[{A}rn63]{Arnold63a}
V.~I. {A}rnol'd.
\newblock Proof of a theorem of {A}. {N}. {K}olmogorov on the invariance of
  quasi-periodic motions under small perturbations.
\newblock {\em Russian Math. Surveys}, 18(5):9--36, 1963.

\bibitem[Bel66]{bel1966}
V.~V. Beletskii.
\newblock {\em Motion of an artificial satellite about its center of mass}.
\newblock Mechanics of Space Flight. Israel Program for Scientific
  Translations; [available from the U.S. Dept. of Commerce, Clearinghouse for
  Federal Scientific and Technical Information, Springfield, Va.], Jerusalem,
  1966.

\bibitem[BL75]{bel1975}
V.~V. {Beletskii} and E.~K. {Lavrovskii}.
\newblock {On the theory of the resonance rotation of Mercury}.
\newblock {\em Astronomicheskii Zhurnal}, 52:1299--1308, December 1975.

\bibitem[BL09]{boulas2009}
Gwenaël Boué and Jacques Laskar.
\newblock Spin axis evolution of two interacting bodies.
\newblock {\em Icarus}, 201(2):750 -- 767, 2009.

\bibitem[BM15]{batmor2015}
Konstantin Batygin and Alessandro Morbidelli.
\newblock {Spin-Spin coupling in the Solar System}.
\newblock {\em The Astrophysical Journal}, 810(2):110, sep 2015.

\bibitem[BO16]{bosort2016}
Alberto Boscaggin and Rafael Ortega.
\newblock Periodic solutions of a perturbed kepler problem in the plane: From
  existence to stability.
\newblock {\em Journal of Differential Equations}, 261(4):2528--2551, 2016.

\bibitem[Bou17]{bou2017}
Gwena{\"e}l Bou{\'e}.
\newblock The two rigid body interaction using angular momentum theory
  formulae.
\newblock {\em Celestial Mechanics and Dynamical Astronomy}, 128(2):261--273,
  Jun 2017.

\bibitem[CC00]{celchi2000}
Alessandra Celletti and Luigi Chierchia.
\newblock Hamiltonian stability of spin–orbit resonances in celestial
  mechanics.
\newblock {\em Celestial Mechanics and Dynamical Astronomy}, 76:229--240, 05
  2000.

\bibitem[CC08]{cel2008}
A.~{Celletti} and L.~{Chierchia}.
\newblock {Measures of basins of attraction in spin-orbit dynamics}.
\newblock {\em Celestial Mechanics and Dynamical Astronomy}, 101:159--170, May
  2008.

\bibitem[CC09]{cel2009}
A.~{Celletti} and L.~{Chierchia}.
\newblock {Quasi-Periodic Attractors in Celestial Mechanics}.
\newblock {\em Archive for Rational Mechanics and Analysis}, 191:311--345,
  February 2009.

\bibitem[CCGdlL21a]{CCGL20c}
Renato~C. Calleja, Alessandra Celletti, Joan Gimeno, and Rafael de~la Llave.
\newblock Break-down threshold of invariant attractors in the dissipative
  spin-orbit problem.
\newblock {\em Preprint}, 2021.

\bibitem[CCGdlL21b]{CCGL20a}
Renato~C. Calleja, Alessandra Celletti, Joan Gimeno, and Rafael de~la Llave.
\newblock Efficient and accurate {K}{A}{M} tori construction for the
  dissipative spin-orbit problem using a map reduction.
\newblock {\em Preprint}, 2021.

\bibitem[CCGdlL21c]{CCGL20b}
Renato~C. Calleja, Alessandra Celletti, Joan Gimeno, and Rafael de~la Llave.
\newblock {K}{A}{M} quasi-periodic tori for the dissipative spin-orbit problem.
\newblock {\em Preprint}, 2021.

\bibitem[{Cel}90]{cel1990}
A.~{Celletti}.
\newblock {Analysis of resonances in the spin-orbit problem in celestial
  mechanics: The synchronous resonance (Part I).}
\newblock {\em Zeitschrift Angewandte Mathematik und Physik}, 41:174--204,
  March 1990.

\bibitem[Cel10]{cel2010}
Alessandra Celletti.
\newblock {\em Stability and Chaos in Celestial Mechanics}.
\newblock Springer, Berlin, Heidelberg, 01 2010.

\bibitem[CFL04]{CFL}
A.~Celletti, C.~Falcolini, and U.~Locatelli.
\newblock On the break-down threshold of invariant tori in four dimensional
  maps.
\newblock {\em Regul. Chaotic Dyn.}, 9(3):227--253, 2004.

\bibitem[DeV58]{dev1958}
Ren\'e DeVogelaere.
\newblock {\em On the Structure of Symmetric Periodic Solutions of Conservative
  Systems, with Applications}, pages 53--84.
\newblock Princeton University Press, 1958.

\bibitem[Eke90]{eke}
I.~Ekeland.
\newblock {\em Convexity Methods in Hamiltonian Mechanics}, volume~19 of {\em
  Ergebnisse der Mathematik und ihrer Grenzgebiete : a series of modern surveys
  in mathematics. Folge 3}.
\newblock Springer-Verlag, Berlin Heidelberg, 1990.

\bibitem[Gol80]{gol1980}
Herbert Goldstein.
\newblock {\em Classical Mechanics}.
\newblock Addison-Wesley, 1980.

\bibitem[GP66]{gol1966}
Peter Goldreich and Stanton Peale.
\newblock {Spin orbit coupling in the Solar System}.
\newblock {\em The Astronomical Journal}, 71:425, 07 1966.

\bibitem[Gre79]{gre1979}
John~M. Greene.
\newblock {A method for determining a stochastic transition}.
\newblock {\em J. Math. Phys.}, 20(6):1183--1201, 1979.

\bibitem[HX17]{hou2017}
Xiyun Hou and Xiaosheng Xin.
\newblock A note on the spin{\textendash}orbit, spin{\textendash}spin, and
  spin{\textendash}orbit{\textendash}spin resonances in the binary minor planet
  system.
\newblock {\em The Astronomical Journal}, 154(6):257, nov 2017.

\bibitem[HZ99]{ham1999}
Alan Hampton and Dami\'an~H. Zanette.
\newblock Measure synchronization in coupled hamiltonian systems.
\newblock {\em Phys. Rev. Lett.}, 83:2179--2182, Sep 1999.

\bibitem[JA16]{jaf2016}
Mahdi {Jafari Nadoushan} and Nima Assadian.
\newblock Geography of the rotational resonances and their stability in the
  ellipsoidal full two body problem.
\newblock {\em Icarus}, 265:175 -- 186, 2016.

\bibitem[JNA16]{jaf2016b}
Mahdi Jafari~Nadoushan and Nima Assadian.
\newblock Chirikov diffusion in the sphere-ellipsoid binary asteroids.
\newblock {\em Nonlinear Dynamics}, 85, 08 2016.

\bibitem[JZ05]{JorbaZ2005}
{\`A}.~Jorba and M.~Zou.
\newblock A software package for the numerical integration of {ODE}s by means
  of high-order {T}aylor methods.
\newblock {\em Experiment. Math.}, 14(1):99--117, 2005.

\bibitem[{Kin}72]{Kinoshita}
H.~{Kinoshita}.
\newblock First-order perturbations of the two finite body problem.
\newblock {\em Publ. Astron. Soc. Japan}, 24:423, January 1972.

\bibitem[Kol54]{kolmogorov1954conservation}
Andrej~N. Kolmogorov.
\newblock On the conservation of conditionally periodic motions under small
  perturbation of the hamiltonian.
\newblock In {\em Dokl. Akad. Nauk. SSR}, volume~98, pages 2--3, 1954.

\bibitem[Kre50]{Krein}
M.G. Krein.
\newblock {Generalization of certain investigations of A.M. Lyapunov on linear
  differential equations with periodic coefficients}.
\newblock {\em Dokl. Akad. Nauk USSR}, 73:445--448, 1950.

\bibitem[LR98]{lamb1998}
Jeroen~S.W. Lamb and John~A.G. Roberts.
\newblock Time-reversal symmetry in dynamical systems: A survey.
\newblock {\em Physica D: Nonlinear Phenomena}, 112(1):1--39, 1998.
\newblock Proceedings of the Workshop on Time-Reversal Symmetry in Dynamical
  Systems.

\bibitem[Mac95]{mac1995}
Andrzej~J. Maciejewski.
\newblock Reduction, relative equilibria and potential in the two rigid bodies
  problem.
\newblock {\em Celestial Mechanics and Dynamical Astronomy}, 63(1):1--28, Mar
  1995.

\bibitem[Mis21]{mis2021}
Mauricio Misquero.
\newblock The spin-spin model and the capture into the double synchronous
  resonance.
\newblock {\em Nonlinearity}, {34}:{2191--2219}, 2021.

\bibitem[MO20]{misort2020}
Mauricio Misquero and Rafael Ortega.
\newblock Some rigorous results on the 1:1 resonance of the spin-orbit problem.
\newblock {\em SIAM J. Applied Dynamical Systems}, 19(4):2233--2267, 2020.

\bibitem[Mos62]{Moser62}
J.~Moser.
\newblock On invariant curves of area-preserving mappings of an annulus.
\newblock {\em Nachr. Akad. Wiss. G\"ottingen Math.-Phys. Kl. II}, 1962:1--20,
  1962.

\bibitem[MW79]{mag}
W.~Magnus and S.~Winkler.
\newblock {\em Hill's Equation}.
\newblock Dover, New York, 1979.

\bibitem[Sch72]{sch1972}
H.R. Schwarz.
\newblock Stability of {K}epler motion.
\newblock {\em Computer Methods in Applied Mechanics and Engineering},
  1(3):279--299, 1972.

\bibitem[Sch02]{sch2002}
D.~J. Scheeres.
\newblock {Stability in the Full Two-Body Problem}.
\newblock {\em Celestial Mechanics and Dynamical Astronomy}, 83(1):155--169,
  2002.

\bibitem[Sch09]{sch2009}
D.~J. Scheeres.
\newblock Stability of the planar full 2-body problem.
\newblock {\em Celestial Mechanics and Dynamical Astronomy}, 104(1):103--128,
  Jun 2009.

\bibitem[Ver78]{Verner78}
J.H. Verner.
\newblock Explicit {R}unge-{K}utta methods with estimates of the local
  truncation error.
\newblock {\em SIAM J. Numer. Anal.}, 15(4):772--790, 1978.

\bibitem[YS75]{yak1975}
V.~A. Yakubovich and V.~M. Starzhinskii.
\newblock {\em Linear Differential Equations with Periodic Coefficients}.
\newblock Wiley, New York, 1975.

\bibitem[ZOST66]{zla}
V.A. Zlatoustov, D.E. Ohotzimsky, V.A. Sarychev, and A.P. Torzhevsky.
\newblock {Investigation of a satellite oscillations in the plane of an
  elliptic orbit}.
\newblock In Görtler H., editor, {\em {Applied Mechanics. Proceedings of the
  Eleventh International Congress of Applied Mechanics Munich (Germany) 1964}},
  pages 436--439. Springer, Berlin, Heidelberg, 1966.

\end{thebibliography}

\end{document}